\documentclass[12pt]{amsart}
\usepackage{amsmath,amssymb,a4wide,epsfig,enumerate,psfrag}
\usepackage[latin1]{inputenc}
\usepackage[english]{babel}
\usepackage[active]{srcltx}
\usepackage[hypertex]{hyperref}

\date{\today}

\newcommand{\f}{_{fuchs}}

\newtheorem{theorem}{Theorem}[section]
\newtheorem{prop}[theorem]{Proposition}
\newtheorem{cor}[theorem]{Corollary}
\newtheorem{lemma}[theorem]{Lemma}

\newtheorem{example}[theorem]{Example}

\newtheorem{definition}[theorem]{Definition}
\newtheorem{rem}[theorem]{Remark}

\date{\today}

\title{Branched projective structures with quasi-Fuchsian
holonomy}

\author{Gabriel Calsamiglia}
\address{Instituto de Matem\'atica, Universidade Federal Fluminense, Rua M\'ario Santos Braga s/n, 24020-140, Niter\'oi, Brazil} \email{gabriel@mat.uff.br}
\author{Bertrand Deroin}
\address{D\'epartement de Math\'ematiques d'Orsay, Universit\'e Paris 11, 91405 Orsay Cedex , France}
\email{bertrand.deroin@math.u-psud.fr}
\author{Stefano Francaviglia}
\address{Dipartimento di Matematica Universit\`a di Bologna, P.zza
Porta S. Donato 5, 40126 Bologna, Italy }
\email{francavi@dm.unibo.it}


\begin{document}
\begin{abstract} We prove that if $S$ is a closed compact
surface of genus $g\geq 2$, and if $\rho : \pi_1(S)\rightarrow
\mathrm{PSL}(2,\mathbb C)$ is a quasi-Fuchsian representation, then the
deformation space $\mathcal M_{k,\rho}$ of branched projective structures
on $S$ with total branching order $k$  and holonomy $\rho$ is connected, as
soon as $k>0$. Equivalently, two branched projective structures with
the same quasi-Fuchsian holonomy and the same number of branch points are
related by a movement of branch points. In particular grafting annuli are obtained by moving branch points. In the
appendix we give an explicit atlas for $\mathcal M_{k,\rho}$ for non-elementary representations $\rho$. It is
shown to be  a smooth complex manifold modeled on
Hurwitz spaces.
\end{abstract}

\maketitle
\markright{BPS with Fuchsian holonomy \today}

\tableofcontents
\section{Introduction}
A $\mathbb C \mathbb P^1$-structure on a surface is a geometric structure
modeled on the Riemann sphere and its group of holomorphic
automorphisms, identified with $\mathrm{PSL}(2,\mathbb C)$. A chart of
a
$\mathbb C \mathbb P^1$-structure can be developed (i.e.
continued with the use of charts) to a map defined on the universal
cover of the surface,
which is equivariant with respect to a certain representation of the
fundamental group of the surface to $\mathrm{PSL} (2,\mathbb C)$,
called the holonomy. This is well-defined up to composition by inner
automorphisms of $\mathrm{PSL}(2,\mathbb C)$. Such a
$\mathbb{CP}^1$-structure will be referred to as projective structure.

Projective structures were introduced by studying second order ODE's, with
applications to the uniformization theorem,
which states that the universal cover of every Riemann surface is
biholomorphic to either $\mathbb{CP}^1$, $\mathbb C$, or $\mathbb
H^2$ corresponding to whether the Euler characteristic is positive, zero or negative. The
composition of the said biholomorphism with the natural inclusion in
$\mathbb{CP}^1$ defines a developing map of a projective structure on
the topological surface whose holonomy representation is the
identification of the fundamental group with a subgroup of
automorphisms of the uniformized covering map, which lies in
$\mathrm{PSL}(2,\mathbb C)$ in either case. In particular, hyperbolic
structures on
closed surfaces are examples of $\mathbb{CP}^1$-structures: the
developing map takes its values in the upper-half plane model of $\mathbb
H^2$ --- viewed as a subset of $\mathbb{CP}^1$ --- and the holonomy in
a discrete co-compact subgroup of $\mathrm{PSL}(2,\mathbb
R)<\mathrm{PSL}(2,\mathbb C)$. In general the holonomy of a
$\mathbb{CP}^1$-structure on a closed surface $S$ is said to be {\em
Fuchsian} if it
is faithful and its image is conjugated to a discrete co-compact
subgroup of $\mathrm{PSL}(2,\mathbb{R})$. For such a
representation we can always consider the corresponding hyperbolic
structure on $S$ which is called the {\em uniformizing} structure. A representation is
called quasi-Fuchsian if it is topologically conjugated to a Fuchsian
representation when acting on the Riemann sphere.

In the 70's, some exotic $\mathbb C \mathbb P^1$-structures with
quasi-Fuchsian holonomy were discovered (see
\cite{Hej},\cite{Mas69},\cite{SuThu}). More precisely, given a
$\mathbb C\mathbb P^1$-structure, there is a surgery called grafting,
which enables to produce a different projective structure without changing its
holonomy. A grafting of the uniformizing structure along a simple closed
curve $\gamma$ is the result of cutting $S$ along $\gamma$ and
gluing back a flat annulus of an appropriate modulus.
In~\cite{Goldman1}, Goldman showed that any projective
structure with quasi-Fuchsian holonomy is obtained by grafting the
uniformizing structure along a
multi-curve .

In particular, this implies that the set of projective structures with the same
quasi-Fuchsian holonomy is discrete.
Baba recently extended Goldman's result to the case where
the holonomy is a generic representation in
$\mathrm{Hom}(\pi_1,\mathrm{PSL}(2,\mathbb C))$, see~\cite{Baba}.

In this paper, we are interested in branched projective structures on closed orientable surfaces. These are given by atlas where
 local charts are finite branched coverings and transition maps lie in
$\mathrm{PSL}(2,\mathbb C)$. Such structures arise  naturally in many
contexts. For instance in the
theories of conical $\mathbb
H^2$-structures, of branched coverings, of locally flat projective
connections or of transversally projective holomorphic foliations (more
details are given later in this introduction and in
sections~\ref{s:definition_and_preliminaries} and~\ref{s:moduli}). As
in the unbranched case, a chart can be continued to define a
developing map on the universal cover of the surface equivariant with
respect to a holonomy representation of the fundamental group of the
surface in $\textrm{PSL}(2,\mathbb{C})$. In the spririt of Goldman's
theorem we give an explicit construction of any branched projective
structure with quasi-Fuchsian holonomy by elementary surgeries that
preserve the holonomy. These new surgeries can be varied continuously and allow to define a topology on the set $\mathcal M_{k,\rho}$ of branched projective
structures with fixed holonomy $\rho : \pi_1(S) \rightarrow
\mathrm{PSL}(2,\mathbb C)$ and total branching order $k$ on a
marked surface $S$ of genus $g$. We show that, unlike in the unbranched case, for quasi-Fuchsian $\rho$ and $k>0$, the deformation space is connected. 

\begin{theorem}[Main result] \label{t:main} Let $S$ be a compact
oriented closed surface. Every branched projective structure with
quasi-Fuchsian holonomy on $S$ having at least one branch point is obtained
from a uniformizing structure by bubblings and moving
branch points. Equivalently, if $k>0$ and $\rho : \pi_1(S) \rightarrow
\mathrm{PSL}(2,\mathbb C)$ is a quasi-Fuchsian representation, then the
deformation space $\mathcal{M}_{k,\rho}$  is
non-empty if and only if $k$ is even, and in this case it is connected
if $k>0$.
\end{theorem}
Let us define the elementary surgeries and the topology on
$\mathcal{M}_{k,\rho}$ properly.

{\em Bubbling} a given branched projective structure consists in
cutting the surface along an embedded arc $\gamma$ whose image by the
developing map
is an embedded arc $\eta\subset \mathbb{CP}^ 1$, and gluing the disc
$\mathbb{CP}^1\setminus \eta$ endowed with the canonical
projective structure.
Obviously, the topology of the surface does not change nor does the
holonomy representation. At the endpoints of $\gamma$ the new branched
projective structure has two branch points.
By bubbling several arcs we can produce examples of branched
projective structures with any even number of simple branch points.

{\em Moving a branch point} is a local surgery that allows to change
the position of a branch point, collapse two or more branch points or
split a non-simple
branch point into several branch points of lower branching order. In
either case the holonomy of the resulting projective structure remains
fixed and the total order of the branching divisor too. This surgery
is a particular case of a Schiffer  variation and allows to
understand the local topology of $\mathcal{M}_{k,\rho}$.

Gallo-Kapovich-Marden~\cite[Problem 12.1.2, p. 700]{GKM} asked whether
any couple of branched projective structures with the same holonomy
are related by a sequence of elementary operations: grafting, degrafting,
bubbling and debubbling. For
quasi-Fuchsian representations Theorem~\ref{t:main} gives a positive answer
by replacing the elementary operations by: bubbling, debubbling and moving
branch points. The connectedness of $\mathcal{M}_{k,\rho}$  and a continuity argument shows that the answer to their original question
is positive for $k=2$. We believe that the argument generalizes for bigger $k$.

It is worth pointing out that in full generality the spaces
$\mathcal{M}_{k,\rho}$ are not necessarily connected. For instance,
consider $\rho$ to be the holonomy of a hyperbolic metric with two
conical points of angle $4\pi$ on a compact surface of genus bigger
than two. It can be thought as a branched projective structure whose
developing map has image in $\mathbb{H}^2$. On the other hand we can
construct examples with the same holonomy and branching order by doing
a bubbling to a $\mathbb{CP}^1$-structure with holonomy $\rho$ --which exists by applying \cite{GKM}-- and
this time the developing map is surjective onto $\mathbb{CP}^1$. Since
this last property is stable under moving branch points the two
projective structures lie in different connected components of
$\mathcal{M}_{2,\rho}$ (see \cite{Tan}, \cite{Mat1} and \cite{Mat2} for related arguments).

Another interesting example is the case where $\rho$ is the trivial
representation. The deformation spaces $\mathcal M_{k,\rho}$ are then the
so-called Hurwitz spaces, namely the moduli spaces of branched
coverings over $\mathbb C \mathbb P^1$. Since Clebsch and Hurwitz we
know that these spaces are connected (see for instance~\cite{HsM} and the more recent generalizations in \cite{LiuO}).

Let us focus on the surgery operations that preserve the holonomy defined so
far. Remark that bubbling adds two branch points, moving branch points
does not change the total branching order, and grafting does not
involve branch points at all. However, these surgeries relate to each
other in interesting ways. Simple examples of such relations can be easily produced: a bubbling followed by local movement of
branch points can be still interpreted as a bubbling, $k$ consecutive
bubblings are related by moving branch points independently of the
order and arcs where we bubble (see Corollaries \ref{c:bc} and
\ref{lemmaunico}). One of the most striking relations between
bubbling and grafting is:

\begin{theorem}\label{l:graftbub}
Over a given branched projective structure a bubbling followed by a
debubbling can produce any grafting on a simple closed curve.
\end{theorem}

It immediately implies one of the key pieces for the proof of Theorem \ref{t:main}, namely that by moving branch points we can produce any grafting annulus. We point out that there are no restrictions on the holonomy representation in Theorem \ref{l:graftbub}.

As will be explained in next subsection, our initial motivation was to study holomorphic curves of general type in a class of non-K\" ahler
threefolds, and the problem led us to a question on the existence of rational curves in $\mathcal{M}_{k,\rho}$. Tan observed (in \cite{Tan}) that each $\mathcal{M}_{k,\rho}$ admits a complex structure. However, the obvious generalization of the complex structure defined in the  absence of branch points presents some subtleties that we want to point out.

It is well known that the space of unbranched projective structures on a given
compact orientable surface has a natural complex structure by using
the Schwarzian derivative of developing maps. In fact it is an affine
bundle over Teichm\" uller space whose fibers are affine spaces over
the vector space of holomorphic quadratic differentials. Its direct
generalization to branched projective structures does not have such a
regular structure. For a branched projective structure we can still
define its underlying complex structure and determine a point in
Teichm\" uller space.  At a branch point of the developing map we can
calculate the Schwarzian derivative with respect to the uniformizing
coordinate of the complex structure. It presents a pole of order two
regardless of the order of branching. In fact it is the coefficient of
the lowest term in the Laurent series that gives the order of
branching.  For functions with this type of development there are even
some extra algebraic conditions on the coefficients to be satisfied to
guarantee that the inversion of the Schwarzian operator produces a
holomorphic germ (as opposed to a pole or logarithmic pole). When we
want to allow to  collapse two different branch points we have a
discontinuity in the lowest coefficient of the series (see the
Appendix for more details). The spaces $\mathcal{M}_{k,\rho}$ are
subspaces of this ``singular complex space'' bundle. Nevertheless, in the
spirit of the topology defined by moving branch points, the spaces
$\mathcal{M}_{k,\rho}$ --where $\rho$ is non-elementary-- admit a
natural complex {\em manifold} structure of dimension $k$. The subject
is discussed in detail for future reference in the Appendix.

\subsection{Additional remarks and open problems} \label{ss:problems}

Determining all  components of
$\mathcal M_{k,\rho}$ seems interesting in general. In some cases we can identify special components. For instance, when the holonomy representation $\rho$ is purely loxodromic, all branched projective structures obtained by bubbling (unbranched) $\mathbb{CP}^1$-structures with the given holonomy, belong to the same connected component. Indeed, by Baba's theorem \cite{Baba}, Theorem \ref{l:graftbub}, and Corollary \ref{c:bc} we can join any pair of such structures by a movement of branch points. As was said before, sometimes it is not the unique connected component.

The next challenging problem is to understand the higher
homology/homo\-topy groups of the deformation spaces $\mathcal M_{k,\rho}$
when $\rho$ is quasi-Fuchsian. These spaces are all homeomorphic if we  fix the genus of the underlying surface and  $k$ (see Proposition \ref{p:homeomorphism}).
The understanding of the second homotopy group of $\mathcal
M_{k,\rho}$ has a strong relation with
 monodromies of linear differential equations on
curves and more precisely, to the Riemann Hilbert problem. Namely,
consider a differential equation of the form
\begin{equation}\label{eq:EDO} dv = \omega \cdot v \end{equation}
where $S$ is a complex algebraic curve, $\omega$ is a given $1$-form over $S$ with
values in the Lie algebra $sl (2,\mathbb C)$, and $v: S \rightarrow
\mathbb C^2$ is a holomorphic map. The Riemann-Hilbert problem
consists in characterizing the representations arising as the monodromy
of the solutions of an equation of type~\eqref{eq:EDO}. For instance, it is not known
whether a non-trivial real monodromy is possible.

If $g$ is the genus of $S$ and $\rho$ is the monodromy of~\eqref{eq:EDO}, then the deformation space
$\mathcal M_{2g-2,\rho}$ has a non trivial second homotopy
group. This is because for each initial value, the
solution $v$  of \eqref{eq:EDO} defines a branched projective structure on $S$ with
monodromy $\rho$, whose total branching order is easily seen to be
$2g-2$. The resulting rational curve in $\mathcal{M}_{k,\rho}$ projects to a rational curve in the moduli space of branched projective structures, whose homological class is non-trivial, since the moduli space is K\"ahler. Hence, proving that $\mathcal M_{2g-2,\rho}$
has trivial second homotopy group would prove that $\rho$ does not
appear as the monodromy of an equation of type~\eqref{eq:EDO}.

This problem is also related to the study of holomorphic curves in the non algebraic
manifolds $\Gamma \backslash \mathrm{SL}(2,\mathbb C)$ where $\Gamma$ is a
lattice in $\mathrm{SL}(2,\mathbb C)$. This space can be thought as the space
of orthonormal frames on a hyperbolic $3$-manifold (see \cite{Ghys}). If we have a solution of
\eqref{eq:EDO} whose monodromy lies in a lattice $\Gamma$ of $\mathrm{SL}(2,\mathbb C)$,
then the matrix formed by two independent solutions of~\eqref{eq:EDO}
defines a curve isomorphic to $S$ in the quotient. We mention here that Huckleberry and Margulis
proved that there is no complex hypersurface in such a complex manifold, see~\cite{HM}.

More generally, one could ask whether $\mathcal M_{k,\rho}$ is a
$K(\pi, 1)$ when $\rho$ is quasi-Fuchsian. This problem can be
compared to a problem of Kontsevich-Zorich on the topology of
connected components of the moduli space of translation surfaces
(which are particular branched projective structures), see~\cite{KoZo}
and the list of problems~\cite{Farbjuly}.

\subsection{Structure of the paper}
After introducing the basic concepts and tools for branched projective structures in section \ref{s:definition_and_preliminaries}, we analyze some special properties of those having Fuchsian holonomy in  Sections \ref{s:decomposition} and \ref{s:index_formula}. Then we proceed to prove Theorem \ref{l:graftbub}.

The proof of the main theorem is carried first under the hypothesis of Fuchsian holonomy. To generalize to quasi-Fuchsian representations, we prove that the space $\mathcal{M}_{k,\rho}$ is homeomorphic to some $\mathcal{M}_{k,\rho'}$ where $\rho'$ is Fuchsian (see Proposition \ref{p:homeomorphism}).

After Theorem  \ref{l:graftbub} the proof reduces to an induction argument that shows that, after moving branch points of a given branched projective structure, it coincides with a finite number of bubblings on a (possibly exotic)
$\mathbb{CP}^1$-structure. This argument takes up most of the paper and we have split it into different steps in Sections \ref{ap:wecanmove} to \ref{s:proofoftheorem}.
 As we mentioned before, there is an Appendix
where we describe the complex structure of the
deformation space $\mathcal M_{k,\rho}$, providing an explicit atlas modeled on Hurwitz
spaces. Other parameterisations of $\mathcal M_{k,\rho}$ are also
discussed.

Let us comment further on the details of the inductive argument. What
we prove is that given a branched projective structure with Fuchsian
holonomy, we can move the branch points so that the structure can be
debubbled. Since debubbling decreases the number of branch points by
$2$, after a finite number of debubblings we find an unbranched
projective structure, hence it is a grafting over a multi-curve of the uniformizing structure by
Goldman's theorem.

We use the point of view of Faltings (\cite{Fal}) and Goldman (\cite{Goldman1}), that is, for a branched projective structure with Fuchsian holonomy we look at
the decomposition of the surface obtained as the pull-back of the
$\mathrm{PSL}(2,\mathbb R)$-invariant decomposition of the Riemann
sphere $\mathbb C\mathbb P^1 = \mathbb H^+ \cup \mathbb R \mathbb P^1
\cup \mathbb H^-$, where $\mathbb H^\pm$ are the upper and lower half
planes. The components of the positive and negative parts
inherit a branched hyperbolic structure, i.e. a conical
hyperbolic metric. Coarse properties of these metrics are explained in
Section~\ref{s:decomposition}, where peripheral geodesics and
peripheral annuli are defined. Most of the work consists in
understanding the geometry of these components in detail, especially that of the peripheral geodesics.

Some topological invariants of the decomposition in positive and negative components
are described by an index formula which we prove by closely following the
ideas of Goldman's thesis (see Section~\ref{s:index_formula}).

To begin moving branch points, we first need to know how and where one can
move them. Sections~\ref{ap:wecanmove} and~\ref{s:deg} provide
sufficient conditions to move branch points.
In particular, in Section~\ref{s:deg}  we deal with possible degenerations to
nodal curves when two branch points collide.

The next step of the proof is  to reduce to the case where
all the branch points belong the positive part (see
Section~\ref{s:moving_positive}). The index
formula then tells us that  there
exist some negative discs isomorphic to a hyperbolic plane.

After that we are able to prove that the peripheral geodesic of the
juxtaposed component of
some negative disc, has a simple topology, namely, it
is a bouquet of at most three circles. This is done by moving the
branch points belonging to a positive component to a
single branch point (see Section~\ref{s:fuchspositivesing}). We then invoke a result proved in
Section~\ref{s:debubbling} by a direct case by case analysis,
which says that the structure can be debubbled.

One of the cases we have to deal with is a particular configuration  that
we called the ``triangles''. They constitute an especially interesting instance and we discuss an example in detail in
subsection~\ref{ss:_triangle}.

The main technical tool that is used along the way is the notion of
embedded twin paths for a branched projective structures. These allow us to move in each component of the deformation space of
branched projective structures.

\subsection{Acknowledgments}

 We are pleased to thank Shinpei Baba, Francesco Bonsante, Bill Goldman, Misha Kapovich, Cyril Lecuire, Samuel Leli\`evre, Frank Loray, Peter Makienko, Luca Migliorini, Gabriele Mondello, Joan Porti and Ser Peow Tan for interesting discussions along the elaboration of the paper.

We are grateful to the following institutions for the very nice
working conditions provided: CRM Barcelona, Institut Henri Poincar\'e,
Mittag-Leffler Institute, Orsay, Pisa University,
Universidade Federal Fluminense (UFF).

G. Calsamiglia's research is supported by CNPq/FAPERJ and CAPES-Mathamsud; B. Deroin's by the ANR projects: 08-JCJC-0130-01,
09-BLAN-0116, and received support from the Brazil - France
cooperation agreement.  S. Francaviglia's research received support from UFF.

\section{Definitions and preliminaries}
\label{s:definition_and_preliminaries}

\subsection{Branched projective structures (BPS)} For
$g\geq2$ let $\Gamma_g$ be a group isomorphic to the fundamental group
of a closed surface of genus $g$. A marked surface of genus $g$ is an
oriented closed surface $S$ of genus $g$ together with the data of a universal
cover $\pi: \widetilde{S} \rightarrow S$ and an identification of $\Gamma_g$
with the covering group of $\pi$.

\begin{definition} A {\bf branched projective structure} (in short
BPS) on a marked surface $S$ is a maximal atlas whose charts are
finite-sheeted, orientation preserving, branched coverings over open
subsets of $\mathbb{CP}^1$, and such that the transition functions belong
to $\mathrm{PSL}(2,\mathbb C)$. We identify two structures if there is a {\bf
projective} (in local charts) diffeomorphism which lifts to a $\Gamma_g$-equivariant diffeomorphism between the universal covers.\end{definition}

A BPS induces a complex structure and thus angles on $S$.  Unbranched
points are called {\bf regular}, the total angle around them
is $2\pi$.  The cone-angle around branch points is $2\pi$ times the
branch-order.

Given a BPS on a marked surface $S$, every local chart can be extended
to a projective map $D:\widetilde S\to \mathbb{CP}^1$, which is
equivariant w.r.t. a representation $\rho:\pi_1(S)\to \mathrm{PSL}(2,\mathbb
C)$:
$$D(\gamma x)=\rho(\gamma)D(x)
\qquad\forall x\in\widetilde S,\gamma\in\pi_1(S).$$

The map $D$ is well-defined up to left-composition by elements of
$\mathrm{PSL}(2,\mathbb C)$. Any representative of its $\mathrm{PSL}(2,\mathbb C)$-left
class is called a {\bf developing map} for the structure and the
representation $\rho$ is called the {\bf holonomy} of the developing
map. If $D_1$ and $D_2=\varphi\circ D_1$ are two developing maps for
the same structure, then the corresponding representations are related
by $$\rho_2=\varphi \rho_1\varphi^{-1}.$$ The conjugacy class of the
holonomy representation is called the holonomy of the structure.  Note
that if the holonomy has trivial centraliser-- for intance if its image
is a non-elementary group-- then once a representative in the conjugacy
class of the holonomy has been fixed, there is only one developing map
for the structure with that holonomy.

In the present paper we are interested in studying projective
structures having a fixed holonomy with some prescribed properties. In
particular we will treat the Fuchsian case. In the literature a Fuchsian group is a discrete subgroup of ${\rm PSL}(2,\mathbb R)$.

\begin{definition} Let $S$ be an oriented closed surface and $\rho:\pi_1(S)\to {\rm PSL}(2,\mathbb C)$ a representation. We say $\rho$ is {\bf Fuchsian} if it is {\bf faithful},  its image is conjugated to a
discrete subgroup $\Gamma$ of ${\rm PSL}(2,\mathbb R)$ with {\bf no parabolic
nor elliptic} elements other than the identity and there exists a $\rho$-equivariant diffeomorphism between $\widetilde{S}$ and $\mathbb H^2$ preserving the orientation.
\end{definition}

\subsection{Examples} \label{ss:examples}

The first obvious examples are complete {\bf
hyperbolic structures}: under the natural inclusion $\textrm{Aut}(\mathbb {H}^2)\hookrightarrow\textrm{Aut}(\mathbb C\mathbb P^1)$, any hyperbolic structure on a closed
surface can be considered as a $\mathbb C\mathbb P^1$-structure,
having no branch points and Fuchsian holonomy.

\begin{definition}[Uniformizing structures] Let $S$ be a closed surface of
genus at least two, and $\rho:\pi_1(S)\to {\rm PSL}(2,\mathbb C)$
be a Fuchsian representation. The {\bf uniformizing structure} on
$S$ is the projective structure induced by the hyperbolic metric on
$S$ with holonomy $\rho$.
\end{definition}

Next, we have {\bf branched coverings}. Let $S$ be a compact
hyperbolic surface and $S_1\rightarrow S$ be a branched covering. By
pulling back the atlas of the uniformizing structure of $S$ we get a
branched projective structure on $S_1$. In general the holonomy of
$S_1$ is not Fuchsian.
\newline

More interesting examples are produced by considering holomorphic
singular codimension one {\bf transversely projective foliations} on
complex manifolds. Such foliations satisfy that the changes of
coordinates of the foliated charts can be written as $(x,z)\mapsto
(h(x,z),\varphi (z))$ for some
$\varphi\in\textrm{PSL}(2,\mathbb{C})$. The foliated charts of a
transversely projective foliation $\mathcal{F}$ on a manifold $M$
induce a branched projective structure on any Riemann surface
$S\subset M$ that avoids the singular set of $\mathcal{F}$ and is
generically transverse to $\mathcal{F}$. It suffices to restrict the
local projections $(x,z)\mapsto z$ to $S$. At the points of tangency
between $S$ and $\mathcal{F}$ we obtain branch points for the induced
BPS on $S$. Transversely projective foliations have been extensively
studied and some accouts can be found in \cite{LoPe}, \cite{Sca} and \cite{Tou}. A particularly
interesting and important family of examples are regular holomorphic
foliations on $\mathbb{CP}^1$-bundles $B\rightarrow S$ over a Riemann
surface $S$ that are transverse to the $\mathbb{CP}^1$-bundle at all
points. Each local chart of the foliation can be defined on the local
trivializing coordinates for $B$. By lifting paths in $S$ starting at
$x_0\in S$ to the leaves of the foliation we can construct a
representation $\rho:\pi_1(S,x_0)\rightarrow
\textrm{PSL}(2,\mathbb{C})$ that actually characterizes the bundle $B$
up to biholomorphisms. In fact, the foliation is equivalent to the
suspension foliation constructed by quotient of the horizontal
foliation on $\widetilde{S}\times\mathbb{CP}^1$ by the action of
$\pi_1(S)$ defined by $\gamma\cdot(x,z)=(\gamma\cdot x,
\rho(\gamma)(z))$.

Now, by the previous construction the foliation induces a BPS on the
image of any holomorphic section $D:S\rightarrow B$ of the
$\mathbb{CP}^1$-bundle that is \emph{not} invariant by the
foliation. In this case the charts of the BPS can be taken as holonomy
germs of the foliation from the image of $D$ to the
$\mathbb{CP}^1$-fibre over a point $x_0\in S$. This BPS can be pulled
back to $S$ via $D$ to produce a BPS on $S$ whose holonomy is
precisely $\rho$. By varying the section (if possible) we can
construct families of branched projective structures on $S$ with the
same holonomy representation $\rho$.

Remark that any BPS on a  Riemann surface $S$ can be realized as the
one induced by a regular holomorphic foliation on a section of a
$\mathbb{CP}^1$- bundle over $S$. As can be readily seen from the
suspension construction, the graph of the developing map
$D:\widetilde{S}\rightarrow\mathbb{CP}^1$ of a BPS with holonomy
$\rho$ is invariant by the defined action of $\pi_1(S)$ on
$\widetilde{S}\times\mathbb{CP}^1$ and hence defines a section of the
quotient $\mathbb{CP}^1$-bundle. The quotient foliation induces the
initially given BPS on $S$ via the constructed section.

Under a more topological viewpoint, we can glue branched projective structures by {\bf cut and
paste}. Given a surface $S$ equipped with a BPS and $\gamma\subset S$ an embeddded curve we consider the surface with boundary obtained by cutting $S$ along $\gamma$ --which topologically can be thought as removing a disc-- and considering its geometric completion $GC(S,\gamma)$ with respect to some riemannian metric on $S\setminus\gamma$.  The curve $\gamma$ corresponds to two curves $\gamma^+$ and $\gamma^-$ in the boundary of $GC(S,\gamma)$, one for each side of the cut. We will sometimes refer to this surface with boundary as {\bf S cut along $\gamma$}.
Given two closed surfaces $S_0$ and $S_1$ equipped with
branched projective structures, let $\gamma_0\subset S_0$ and $\gamma_1\subset S_1$ be embedded segments, containing no
branch points and having
neighborhoods $U_0$ and $U_1$ such that there is a projective
diffeomorphism $f:U_0\to U_1$ mapping $\gamma_0$ to $\gamma_1$. The map $f$ can be defined as a diffeomorphism from $GC(S,\gamma_0)$ to $GC(S,\gamma_1)$ sending $\gamma_0^{\pm}$ to $\gamma_1^{\pm}$ and preserving orientations. By using this diffeomorphism as a gluing we get a new closed oriented surface equipped with a BPS. See Figure~\ref{fs1}.

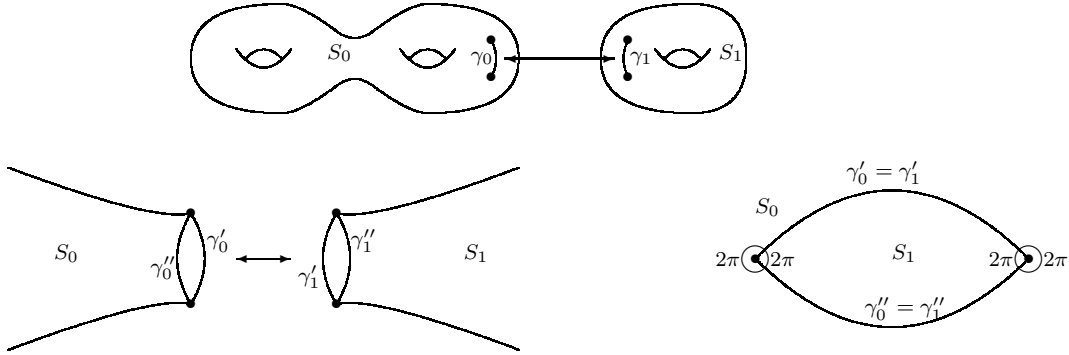
\begin{figure}[htbp] \centering
\tiny{ \setlength{\unitlength}{1ex}
\begin{picture}(115,35)

\put(0,-30){ \put(20,59){
\qbezier(0,0)(0,6)(10,6)\qbezier(10,6)(12,6)(16,3)
\qbezier(16,3)(17,2.25)(18,2.25) \qbezier(6,0)(8,2)(10,0)
\qbezier(0,0)(0,-6)(10,-6)\qbezier(10,-6)(12,-6)(16,-3)
\qbezier(16,-3)(17,-2.25)(18,-2.25) \qbezier(5,1)(8,-3)(11,1)
\put(36,0){ \qbezier(-0,0)(-0,6)(-10,6)\qbezier(-10,6)(-12,6)(-16,3)
\qbezier(-16,3)(-17,2.25)(-18,2.25) \qbezier(-8,0)(-10,2)(-12,0)
\qbezier(-0,-0)(-0,-6)(-10,-6)\qbezier(-10,-6)(-12,-6)(-16,-3)
\qbezier(-16,-3)(-17,-2.25)(-18,-2.25) \qbezier(-7,1)(-10,-3)(-13,1) }
\qbezier(33,2)(34,0)(33,-2) \put(33,2){\makebox(0,0){$\bullet$}}
\put(33,-2){\makebox(0,0){$\bullet$}}
\put(32,0){\makebox(0,0){$\gamma_0$}}

                \put(45,0){
\qbezier(0,0)(0,6)(10,6)\qbezier(10,6)(16,6)(16,0)
\qbezier(7,0)(9,2)(11,0)
\qbezier(0,0)(0,-6)(10,-6)\qbezier(10,-6)(16,-6)(16,0)
\qbezier(6,1)(9,-3)(12,1) } \qbezier(48,2)(47,0)(48,-2)
\put(48,2){\makebox(0,0){$\bullet$}}
\put(48,-2){\makebox(0,0){$\bullet$}}
\put(49.5,0){\makebox(0,0){$\gamma_1$}}

         \put(40,0){\vector(1,0){6.5}} \put(40,0){\vector(-1,0){5.5}}

         \put(15,0){$S_0$} \put(58,0){$S_1$} }

       \put(20,37){ \qbezier(0,5)(-4,4)(-20,10)
\qbezier(0,-5)(-4,-4)(-20,-10) \qbezier(0,5)(-3,0)(0,-5)
\put(0,5){\makebox(0,0){$\bullet$}}
\put(0,-5){\makebox(0,0){$\bullet$}}
\put(-3,-1){\makebox(0,0){$\gamma_0''$}}
\put(3,2){\makebox(0,0){$\gamma_0'$}} \qbezier(0,5)(3,0)(0,-5)
\put(-15,0){$S_0$} \put(16,0){ \qbezier(0,5)(-3,0)(0,-5)
\qbezier(0,5)(3,0)(0,-5) \qbezier(0,5)(4,4)(20,10)
\qbezier(0,-5)(4,-4)(20,-10) \put(0,5){\makebox(0,0){$\bullet$}}
\put(0,-5){\makebox(0,0){$\bullet$}}
\put(-3,-2){\makebox(0,0){$\gamma_1'$}}
\put(3,2){\makebox(0,0){$\gamma_1''$}} \put(14,0){$S_1$} }
\put(8,0){\vector(1,0){3}} \put(8,0){\vector(-1,0){3}} }

          \put(97,37){ \qbezier(-15,0)(0,15)(15,0)
\put(-15,0){\makebox(0,0){$\bullet$}} \qbezier(-15,0)(0,-15)(15,0)
\put(15,0){\makebox(0,0){$\bullet$}} \put(-15,5){$S_0$}
\put(0,0){$S_1$} \put(-3,-6){$\gamma_0''=\gamma_1''$}
\put(-5,9){$\gamma_0'=\gamma_1'$} \put(-15,0){\circle{3}}
\put(15,0){\circle{3}} \put(-18,0){\makebox(0,0){$2\pi$}}
\put(-12,0){\makebox(0,0){$2\pi$}} \put(18,0){\makebox(0,0){$2\pi$}}
\put(12,0){\makebox(0,0){$2\pi$}} } }
\end{picture} }
\caption{Gluing two surfaces along a segment}
\label{fs1}
\end{figure}

The topological result of the entire operation is the connected sum
$S_0\sharp S_1$. As for the branched projective structures, two new branch points
appeared: the end-points of $\gamma_0$ now identified with
the endpoints of $\gamma_1$.

The holonomy of the resulting structure can be computed from the two
initial holonomies. In particular we note that the loop corresponding
to $\gamma_0'\cup\gamma_0''$ has trivial holonomy. Thus, if both
surfaces have non-trivial topology ({\em i.e.} with non-positive Euler characteristic) then the
resulting holonomy is {\bf not faithful} and in general it is not
discrete.

A slightly subtler example is the
{\bf conical cut an paste}\label{conicalcutandpaste}, which is a surgery
as before that allows irrational cone-singularities.  For instance, suppose
$S_0$ and $S_1$ have complete hyperbolic metric, each with one
cone-singularity with angle respectively $\alpha_0$ and
$\alpha_1$. Suppose further that the cone-points are exactly the ends of two
geodesic embedded segments $\gamma_0$ and $\gamma_1$ of the same
length, and suppose moreover that
$$\alpha_0+\alpha_1=2\pi.$$

Cut $S_0$ and $S_1$ along $\gamma_0$ and $\gamma_1$ and glue the result isometrically
along the boundary. In Figure~\ref{fs1} one has to consider the bottom
right picture.  In the former example we had angles $4\pi$ at both
marked points, whereas now those are $4\pi$ at one point and $\alpha_0+\alpha_1$
at the other. Therefore, the resulting structure has only one new
branch point, as one of the marked point of the loop
$\gamma_0'\cup\gamma_0''$ has total angle $2\pi$, and so it is
regular.  The holonomy of the loop
$\gamma_0'\cup\gamma_0''$ is an elliptic transformation of $\mathrm{PSL}(2,\mathbb C)$.

\subsection{Grafting, Bubbling and moving branch points} Here we
describe three ways of producing new structures starting from a given
one, without changing the holonomy. We will need the following definition.

\begin{definition} Let $S$ be a surface equipped with a BPS with
developing map $D$. For any subset $K\subset S$ contained in some
simply connected open set $U$, the {\bf developed image} of $K$ is the
projective class of $D(\widetilde K)$, where $\widetilde K$ is any
lift of $K$ to the universal cover $\widetilde S$ of $S$. For any
continuous map $f$ with values in $S$ that lifts to $\widetilde S$, the {\bf
developed} of $f$ is the projective class of $D\circ f$.
\end{definition}

The first construction is the so-called {\bf grafting} (of angle
$2\pi$), and it can be described as follows. Let $S$ be a marked
surface equipped with a BPS with holonomy $\rho$. Suppose that there is a simple closed
curve in $S$ with loxodromic holonomy, and such that
any of its lift $\widetilde{\gamma}$ in $\widetilde{S}$ develops
injectively in $\mathbb {CP}^1$. Hence the path $D \circ
\widetilde{\gamma}$
tends to the fixed points of the corresponding holonomy map. Cut
$\widetilde{S}$ on each $\widetilde{\gamma}$, and glue a copy of the
 canonical projective structure on  $\mathbb {CP}^1$ cut along $D\circ \widetilde{\gamma}$, by using the
developing map.

We
obtain in this way a simply connected surface $\widetilde{S}'$, with a free and
discontinuous action of $\Gamma_g$, and a $\rho$-equivariant map $D' :
\widetilde{S'} \rightarrow \mathbb{CP}^1$ which is a local branched
covering. As the endpoints of the cut are \emph{not} on $\widetilde{S}$, we have not added any new branch points when gluing. Hence, this defines a new BPS on the marked surface $S' :=
\Gamma_g \backslash \widetilde{S'}$, which is called the grafting of $S$ along $\gamma$. The
quotient $S'$ is obtained from $S$ by replacing $\gamma$ with a
cylinder.A detailed description of the projective structure on the cylinder can be found in Section \ref{s:graftbub}.

In general it is not easy to find a \textbf{graftable curve on a BPS}, that is, a simple closed curve in a given BPS with loxodromic holonomy which develops injectively when lifted to the universal cover. Baba showed in~\cite{Baba} that this is always possible if the projective structure on $S$ has no branch points. However, we do not know whether it is still true when there is at least one branch point. Remark that in the case where the original structure on $S$ is a uniformizing structure, then {\bf every} simple closed curve on $S$ has this property, giving rise to a lot of possible graftings.

The second construction is what we denote by {\bf bubbling}, which is
nothing but the cut and paste with a $\mathbb{CP}^1$ along an embedded arc. In this case the number of branch points changes by two.

\begin{definition} Let $S$ be a surface endowed with a branched
projective structure $\sigma$. Let $\gamma$ be an embedded segment in
$S$ having embedded developed image $D\circ\gamma$ in $\mathbb{CP}^1$. Let $\sigma_1$
be the branched projective structure  obtained by cutting $S$ along $\gamma$ and gluing
a copy of the canonical projective structure on $\mathbb{CP}^1$ cut along $D\circ \gamma$ via the developing map.
 We say that $\sigma_1$ is obtained by {\bf bubbling}
$\sigma$ and that $\sigma$ is obtained by {\bf debubbling} $\sigma_1$.
\end{definition}

Bubbling is topologically the connected sum with a sphere, so the
fundamental groups before and after bubbling are canonically
isomorphic. Thus the marking is preserved, and doing the construction
at the level of the fundamental group shows that
the holonomy does not change under bubbling.

Our third way to constructs new structures keeping the holonomy fixed is
the procedure of {\bf moving branch points}. Let $S$ be an oriented
closed surface equipped with a BPS with developing map $D$.

\begin{definition} Two distinct paths $\gamma_0$ and $\gamma_1$ on
$S$, both defined on the same interval $[0,T]$, are {\bf twins} if they
overlap once developed, i.e.:
\begin{itemize}
\item $\gamma_0(0)=\gamma_1(0)$ is a branch point of $S$;
\item If $\alpha:[-T,T]\to S$ is given by $\alpha(t)=\gamma_0(t)$ for
$t\geq 0$ and $\gamma_1(-t)$ for $t\leq 0$, then the developed
$\bar\alpha$ of $\alpha$ is even: $\bar\alpha(t)=\bar\alpha(-t)$.
\end{itemize} If $\gamma_0$ and $\gamma_1$ are embedded and disjoint
appart from $\gamma_0(0)= \gamma_1(0)$, they are called {\bf embedded
twin paths}.
\end{definition}

Let $x$ be a branch point of $S$. Let $\gamma_0$ and $\gamma_1$ be
embedded, piecewise smooth,
twin paths starting from $x$ and defined on $[0,T]$. We
denote by $\alpha$ and $\beta$ the two angles that they form at $x$,
and by $\theta_i$ the angle around $\gamma_i(T)$, $i=0,1$ ($\theta_i$
is $2\pi$ if $\gamma_i(T)$ is a regular point).  We cut $S$ along the
images of the $\gamma_i$'s. The resulting surface $S_0$ has a boundary
formed by two copies $\gamma_0'$ and $\gamma_0''$ of $\gamma_0$ and
two copies $\gamma_1'$ and $\gamma_1''$ of $\gamma_1$, all of them
parameterized by $[0,T]$, and so that $\gamma_0'(0)=\gamma_1'(0)$ and
$\gamma''_0(0)=\gamma_1''(0)$. (See Figure~\ref{f:movsing}.)

Now we glue back by identifying, for any $t\in [0,T]$, $\gamma_0'(t)$
with $\gamma_1'(t)$ and $\gamma_0''(t)$ with $\gamma_1''(t)$.

\begin{figure}[h] \tiny{\setlength{\unitlength}{.4ex}
\makebox[\textwidth]{
\begin{picture}(170,50) \put(0,0){

\multiput(10,39)(40,0){3}{$\bullet$} \put(11,40){\line(1,0){80}}

\multiput(11,11)(79,0){2}{\makebox(0,0){$\bullet$}}
\multiput(49,20)(0,-20){2}{$\bullet$}

\put(11,11){\line(4,1){39}} \put(11,11){\line(4,-1){39}}
\put(50,21){\line(4,-1){40}} \put(50,1){\line(4,1){40}}

\qbezier(12,9)(8.5,9)(8.5,12) \qbezier(12,9)(13,9)(14,10.1)
\qbezier(12,14)(8.5,14)(8.5,12) \qbezier(12,14)(13,14)(14,12)

\put(101,0){ \qbezier(-12,9)(-8.5,9)(-8.5,12)
\qbezier(-12,9)(-13,9)(-14,10.1) \qbezier(-12,14)(-8.5,14)(-8.5,12)
\qbezier(-12,14)(-13,14)(-14,12) }

\put(0.45,-6.1){ \qbezier(50,30)(47,30)(47,26.5)
\qbezier(50,30)(53,30)(53,26.5)}

\put(0.45,28.2){ \qbezier(50,-30)(47,-30)(47,-26.5)
\qbezier(50,-30)(53,-30)(53,-26.5)}

\put(8,17){\makebox(0,0){$\theta_0$}}
\put(91,18){\makebox(0,0){$\theta_1$}}
\put(50,27){\makebox(0,0){$\alpha$}}
\put(50,-5){\makebox(0,0){$\beta$}}

\multiput(151,0)(0,20){3}{\makebox(0,0){$\bullet$}}
\put(151,1){\line(0,1){40}}

\put(10,45){$\gamma_0(T)$} \put(90,45){$\gamma_1(T)$}
\put(34,45){$x=\gamma_0(0)=\gamma_1(0)$}

\put(30,20){$\gamma_0'$} \put(61,20){$\gamma_1'$}
\put(27,0){$\gamma_0''$} \put(64,0){$\gamma_1''$}

\put(70,35){\vector(0,-1){15}} \put(72,27){\textrm{cut}}

\put(155,20){$\theta_1$} \put(142,20){$\theta_0$}
\put(162,9){\vector(-1,1){9}}\put(163,7){$y$}
\put(162,29){\vector(-1,1){9}}\put(163,27){$x'$}
\put(138,0){\vector(1,0){9}}\put(133,-1){$x''$}

\multiput(151,40)(0,-20){3}{\circle{5}} \put(145,43){$\alpha$}
\put(152,-6){$\beta$}

\put(110,20){\vector(4,1){30}} \put(105,25){\textrm{gluing}} }
\end{picture} } }
\caption{Moving a branch point}\label{f:movsing}
\end{figure}
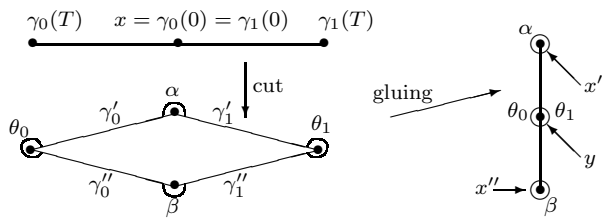

The result is a surface $S_1$, with three distinguished points:
\begin{itemize}
\item The point $y$ resulting from the identification of
$\gamma_0(T)'$ with $\gamma_1'(T)$. The total angle around that point
is $\theta_1+\theta_2$.
\item The point $x'=\gamma_0'(0)=\gamma_1'(0)$. The angle around it is
$\alpha$.
\item The point $x''=\gamma_0''(0)=\gamma_1''(0)$. The angle around it
is $\beta$.
\end{itemize}

\medskip

Note that angles $\alpha$ and $\beta$ are both multiples of $2\pi$. If
$x$ was a branch point of order two, then
$\alpha=\beta=2\pi$. Similarly, the $\theta_i$'s may be different from
$2\pi$ but they are integer multiples of $2\pi$.  Note also that
segments $[y,x']$ and $[y,x'']$ are twin and that, in local charts,
the developed image of $[y,x']$ is the same as the one of $\gamma_0$
and $\gamma_1$.

The surface $S_1$, which is endowed with a BPS, is clearly
diffeomorphic to $S$, and an isotopy-class of diffeomorphisms between
$S_1$ and $S$ is well-defined, so that the marking and the holonomy
are preserved.

\begin{definition}
 We say that a branched projective structure
$\sigma_1$ is obtained from $\sigma$ by {\bf moving branch points} if
it is the resulting structure after a finite number of cut-and-paste
procedures as above. Two structures obtained one from the other by
moving branch points are {\bf connected by moving branch points}.
\end{definition}

The following two lemmas are easy to establish and the proofs are left to
the reader.
\begin{lemma}\label{lemma28}
  Let $S$ and $P$ be surfaces endowed with BPS's $\sigma_S$ and $\sigma_P$.
  Let $\gamma_0\subset S$ and $\eta_0\subset P$ be embedded paths
  that have neighborhoods $V_S$ and $V_P$ so that $(V_S,\gamma_0)$ and
  $(V_P,\eta_0)$ are
  projectively equivalent. Suppose that
  $\gamma_0$ is isotopic to $\gamma_1$ via an isotopy
  $\{\gamma_t\}_{t\in[0,1]}$ that fixes end-points, and that $\eta_0$
  is isotopic to $\eta_1$ via $\{\eta_t\}_{t\in[0,1]}$ with fixed
  end-points. Suppose moreover $(V_S,\gamma_t)$ and $(V_P,\eta_t)$
  are projectively equivalent for any $t$.  Let $R_t$ be the
  surface obtained by cut-and-pasting $S$ and $P$ along $\gamma_t$ and
  $\eta_t$, endowed with the BPS $\sigma_t$ induced by $\sigma_S$ and
  $\sigma_P$. Then, $\sigma_t$ is projectively equivalent to
  $\sigma_\tau$ for any $t,\tau\in[0,1]$.
\end{lemma}
Applied to $P=\mathbb{CP}^1$, Lemma~\ref{lemma28} says that bubblings do not
depend on the local isotopy class of the segment chosen to do the
cut-and-paste procedure.

\begin{lemma}\label{l:connectingpaths}
  Let $S$ be a surface endowed with a BPS $\sigma$. Let
  $\gamma:[0,1]\to S$ be an
  embedded path having embedded developed image. Let $\tau:[1,2]\to S$
  be another embedded path so that $\tau(1)=\sigma(1)$. Suppose that
  $\gamma*\tau:[0,2]\to S$ is embedded with embedded developed
  image. Then the bubbling $\sigma$ along $\gamma$ is
  obtained by that along $\gamma*\tau$ by moving branch points.
\end{lemma}

\begin{cor}[Bubblings commute]\label{c:bc}
  Let $S$ be a surface endowed with a BPS $\sigma$. Let
  $\beta_1$ and $\beta_2$ be bubblings of $\sigma$ along paths
  $\gamma_1$ and $\gamma_2$ respectively. Then $\beta_1$ is connected to
  $\beta_2$ by moving branch points.
\end{cor}

\begin{proof}
  For $i=1,2$, let $\gamma_i'$ and $\gamma_i''$ be the twin
  paths in $\beta_i$ arising from $\gamma_i$, and let $x_i,y_i$ be their
  common end-points. By moving $x_i$ and $y_i$ along initial segments
  of  $\gamma_i'\cup\gamma_i''$ we reduce to the case that $\gamma_1$
  and $\gamma_2$ are disjoint.

Since a BPS has an atlas which is a local homeomorphism outside
branch points, there is a finite sequence of embedded paths with
embedded developed images, connecting
$\gamma_1$ and $\gamma_2$. That is to say paths  $\tau_i:[i,i+1]\to S$,
$i=0,\dots,n$  so that:
\begin{itemize}
\item $\tau_0=\gamma_1$ and $\tau_n=\gamma_2$;
\item $\tau_i(i+1)=\tau_{i+1}(i+1)\ \forall i=0,\dots, n-1$ ;
\item $\tau_i(t)\in S\setminus(\{$branch points of $\sigma\}\cup
  Im(\tau_{i-1})\cup Im(\tau_{i+1}))\ \forall t\in(i,i+1)$ and  $i=1,\dots,n-1$;
\item $\tau_{i-1}*\tau_i:[i-1,i+1]\to S$ is embedded with embedded
  developed image.
\end{itemize}
By Lemma~\ref{l:connectingpaths} recursively applied to
$\tau_i,\tau_{i+1}$, we get the desired claim.
\end{proof}

\begin{cor}[Cut-and-paste and moving commute]\label{lemmaunico}
 Let $B$ and $C$ two surfaces equipped with BPS's. Let
 $\gamma_B\subset B$ and
 $\gamma_C\subset C$ be segments with neighborhoods that are
 projectively equivalent and with regular end-points.
 Let $A$ be the surface obtained by cut-and-pasting $B$ and
 $C$ along $\gamma_B$ and $\gamma_C$ (see Figure~\ref{fs1}), endowed
 with the BPS's induced by those of $B$ and $C$. Let $D$ be a BPS obtained from
 $C$ by moving branch points.

Then, $A$ is connected by moving branch points to a cut-and-paste of
$B$ and $D$.
\end{cor}
\begin{proof}
We parameterize $\gamma_B$ and $\gamma_C$ in a projectively equivalent way.
The cut-and-paste consists in cutting $B$ along $\gamma_B$ and $C$
along $\gamma_C$, so that each $\gamma_\bullet$ splits in two copies
$\gamma_\bullet'$ and $\gamma_\bullet''$ (for $\bullet=B,C$), and then
in identifying  $\gamma_B'(t)$ with $\gamma'_C(t)$ and $\gamma_B''(t)$
with $\gamma''_C(t)$. The two resulting twin paths in $A$ are named
 $\gamma'$ and $\gamma''$, and their end-points are named
 $\gamma'(0)=x,\gamma'(1)=y$.  Also, we name
 $x_B=\gamma_B(0), y_B=\gamma_B(1), x_C=\gamma_C(0), y_C=\gamma_C(1)$.

Let $\mathfrak M$ be the finite sequence of movements on $C$ that produces
$D$. By arguing by induction on the number of cut-and-paste procedures
of $\mathfrak M$ we reduce to the case of a single cut-and-paste along
twin paths $\tau_0,\tau_1$ in $C$.

If the $\tau_i$'s do not intersect $\gamma_C$
the claim is obvious.
The $\tau_i$'s are piece-wise smooth by
definition of moving. Therefore, by Lemma~\ref{lemma28} we can perturb
 $\gamma_C$ and $\gamma_B$ via isotopies, without changing $A$,
so to reduce to
the case where $\gamma_C$ is transverse to the $\tau_i$'s.

Let $U$ be a
neighborhood of a point $p\in\gamma_C$
so that $U\cap \tau_0=U\cap\tau_1=\emptyset$.
In $A$, we move the branch points $x$ and $y$ by using, as twin paths,
initial segments of $\gamma'$ and $\gamma''$. Of course, this affects
the structure on $A$ but not those of $B$ and $C$ because $x_B,y_B$ are
regular points in $B$, and $x_C,y_C$ are regular in $C$. After such a move,
the new structure $\bar A$ is the cut-and-paste of $B$ and $C$ along
segments $\bar\gamma_B\subset \gamma_B$ and
$\bar\gamma_C\subset\gamma_C$. In particular we can move $x,y$ enough to
obtain $\bar\gamma_C\subset U$. Since $\bar\gamma_C$ does not
intersect the $\tau_i$'s, the claim is true for $\bar A$.
Since $\bar A$ is connected to $A$ by
moving branch points, the claim follows.
\end{proof}

In general one can always move branch points locally, but a priori there is no guarantee that
one can do it along any given path. More precisely, if one starts with a germ of
embedded twin paths, it is possible that their analytic continuations cease to be embedded very
soon. In general, there does not exist an a priori lower bound on the
maximal size where twin paths are embedded.
In Section~\ref{ap:wecanmove} we give precise statements
ensuring that all moves needed throughout our
proofs are possible under the given hypotheses.

\section{Fuchsian holonomy: real curve and decomposition into
hyperbolic pieces} \label{s:decomposition}

In this section $S$ is a closed oriented surface endowed with a BPS
$\sigma$ and $D:\widetilde S\to\mathbb{CP}^1$ is a developing map for
$\sigma$ with Fuchsian holonomy $\rho:\pi_1(S)\to \mathrm{PSL}(2,\mathbb R)$. All this material can be extended to the case where the representation is quasi-Fuchsian, but for simplicity we restrict ourselves to Fuchsian ones.

\subsection{Real curve and decomposition}

The decomposition $\mathbb{CP}^1=\mathbb{H}^+\sqcup\mathbb{R
P}^1\sqcup\mathbb{H}^-$ into the real line and the two hemispheres
$\mathbb{H}^+=\mathbb{H}^2=\{\Im(z)>0\}$ and
$\mathbb{H}^-=\{\Im(z)<0\}$ can be pulled back via $D$ to
$\widetilde{S}$ and defines a decomposition of $S=S^+\sqcup
S_{\mathbb{R}}\sqcup S^-$.

\begin{definition} The {\bf real curve} is the set $S_{\mathbb R}$,
the {\bf positive} part (resp. {\bf negative}) is the set $S^+$
(resp. $S^-$).
\end{definition}

Since the holonomy takes values in $\mathrm{PSL}(2,\mathbb R)$, the real
curve is a compact real analytic sub-manifold of $S$ of dimension $1$ --- possibly singular if it
contains some branch point --- and the lifts $\widetilde{S^\pm}$ of
$S^{\pm}$ to $\widetilde{S}$ are precisely $D^{-1}(\mathbb{H}^\pm)$.
Each connected component $C$ of $S\setminus S_{\mathbb{R}}$ inherits a
branched $(\mathbb{H}^2,{\rm {\rm PSL}}(2,\mathbb{R}))$-structure by
restriction of $D$ (in the case $C\subset S^+$, or of its complex
conjugate $\overline{D}$ in the case $C\subset S^-$) to a lift
$\widetilde{C}\subset\widetilde{S}$ of $C$. In a similar way every connected component $l$ of $S_{\mathbb{R}}$ inherits a branched $(\mathbb{RP}^1, {\rm {\rm PSL}}(2,\mathbb{R}))$-structure. Indeed, it suffices to consider $D|_{\widetilde{l}}$ where $\widetilde{l}$ of is a lift of $l$ to $\widetilde{S}$.  In the next two subsections we will analyze the properties of the geometric structures induced by $\sigma$ on the real curve and on its complement in $S$.

\subsection{The real projective structure on $S_{\mathbb{R}}$} \label{ss:real projective structures}

 On each connected component $l$ of $S_{\mathbb{R}}$ we distinguish some special points corresponding to the fixed points by $\alpha:=\rho([l])\in Aut(\mathbb{RP}^1)$. If $p_i\in\mathbb{RP}^1$ is fixed by $\alpha$ the set $D|_{\widetilde{l}}^{-1}(p_i)$ is discrete and invariant by the action of $[l]$ on $\widetilde{S}$ and thus defines a finite set $P_i$ of points in $l$. The cardinality $I_l$ of $P_i$ is independent of the choice of fixed point $p_i$ and will be defined as the \textbf{index} of the $\mathbb{RP}^1$-structure on $l$. In the case of trivial $\alpha$, the map $D$ descends to a map $l\rightarrow\mathbb{RP}^1$ and the index coincides with the degree of this map.

 As with complex projective structures, we say that two $\mathbb{RP}^1$-structures on a circle $l$ are equivalent if there is a diffeomorphism between the two structures which is projective in the charts of the given projective structures. The following proposition gives the classification of unbranched $\mathbb{RP}^1$-structures on $l$ having some fixed point in the holonomy.

\begin{prop}\label{p:RP1structure}
Two unbranched $\mathbb{RP}^1$-structures on an oriented circle $l$  whose respective holonomies $\alpha$ and $\alpha'$ fix at least one point and with indices $I,I'$  are equivalent if and only if $I=I'$ and $\alpha'=\varphi\circ \alpha\circ \varphi^{-1}$ for some $\varphi\in {\rm PSL}(2,\mathbb{R})$. The only case that cannot occur is $\alpha=id$ and $I=0$.
\end{prop}

\begin{proof}
 If $\alpha$ and $\alpha'$ are trivial, we just need to prove that two coverings of $l\rightarrow\mathbb{RP}^1$ are equivalent if and only if they have the same degree, which is obviously true. The degree zero covering is impossible since there are no branch points. The proof of the proposition is a generalization of the proof of the previous fact.
  We first  construct a model of a $\mathbb{RP}^1$-structure on the circle  $\mathbb{S}^1$ with prescribed index $I$ and holonomy $\alpha$. Let $\widetilde{\mathbb{RP}^1}\rightarrow\mathbb{RP}^1$ denote a universal covering map and $T:\widetilde{\mathbb{RP}^1}\rightarrow\widetilde{\mathbb{RP}^1}$ denote the action of the positive generator of $\pi_1(\mathbb{RP}^1)$ on the universal cover $\widetilde{\mathbb{RP}^1}$ (positive means that $T (x)$ is on the right of $x$ for every $x\in \widetilde{\mathbb{RP}^1}$). Lift $\alpha:\mathbb{RP}^1\rightarrow \mathbb{RP}^1$ to a map $\widetilde{\alpha}$ from $\widetilde{\mathbb{RP}^1}$ to itself which has at least one fixed point. Since $\widetilde{\alpha}$ and $T$ commute, the quotient of $\widetilde{\mathbb{RP}^1}$ by $\widetilde{\alpha}\circ T^I$ is homeomorphic to a circle $\mathbb{S}^1$ equipped with a $\mathbb{RP}^1$-structure with index $I$ and holonomy $\alpha$. Of course, if we compose the chosen universal covering map on the left by an element $\varphi\in\mathrm{Aut}(\mathbb{RP}^1)$ we get an equivalent $\mathbb{RP}^1$-structure with holonomy $\alpha'=\varphi\circ\alpha\circ\varphi^{-1}$.

Given any $\mathbb{RP}^1$-structure on a circle $l$ with holonomy $\alpha$, its developing map $d:\widetilde{l}\rightarrow\mathbb{RP}^1$ satisfies $d([l]\cdot z)=\alpha(d(z))$ and lifts to the covering $\widetilde{\mathbb{RP}^1}\rightarrow(\widetilde{\mathbb{RP}^1}/T)\cong\mathbb{RP}^1$ as a map $\widetilde{d}:\widetilde{l}\rightarrow \widetilde{\mathbb{RP}^1}$, such that $$\widetilde{d}([l]\cdot z)=(\widetilde{\alpha}\circ T^I)(\widetilde{d}(z)),$$
for some integer $I$. Observe that the integer $I$ is necessary non negative since $T$ is positive and $\widetilde{d}$ preserves orientation; this integer is nothing but the index of the projective structure.

In the case where $I=0$, the image of $\widetilde{d}$ is an interval between two consecutive fixed points of $\widetilde{\alpha}$, hence the structure is the quotient of the (unique) open interval in $\mathbb {RP}^1$ between consecutive fixed points of $\alpha$ where $\alpha$ acts as a positive map.

In the case where $I>0$, the image of $\widetilde{d}$ is the whole $\widetilde{\mathbb{RP}^1}$, since $T$ acts discretely; hence $\widetilde{d}$ is a diffeomorphism, which induces a projective diffeomorphism between the given $\mathbb{RP}^1$-structure on $l$ and the model $\mathbb{RP}^1$-structure on $\mathbb{S}^1$ with index $I$ and holonomy $\alpha$ constructed above. Hence the result.
\end{proof}

In the Fuchsian case the hypothesis on the holonomy of Proposition~\ref{p:RP1structure} is always satisfied, since we have either trivial holonomy or precisely two fixed points $p_1,p_2\in\mathbb{RP}^1$ for the loxodromic holonomy $\alpha$. In the latter case we can carry the decomposition of $S$ induced by the properties of the holonomy representation further. Indeed, after conjugation we can suppose $p_1=0$, $p_2=\infty$ and $\alpha(z)=\lambda z$ for some $\lambda>0$. The partition $\mathbb{RP}^1=0\sqcup\mathbb{R}^+\sqcup\infty\sqcup\mathbb{R}^-$ is thus invariant by $\alpha$ and induces, via the developing map, a partition of $l$ as $l=P_1\sqcup l^+\sqcup P_2\sqcup l^-$ where $l^+$ and $l^-$ are unions of disjoint oriented intervals and $P_1$, $P_2$ correspond to the sets used in the definition of the index of the $\mathbb{RP}^1$- structure.

\subsection{Geometry of the hyperbolic structures on $S\setminus S_{\mathbb{R}}$}

The pull-back of the hyperbolic metric
on $\mathbb H ^+$ by the developing map defines on $S^+$ a metric
which is smooth and has curvature $-1$ away from the branch points. At a point with branching order $n\geq 1$ the metric is singular and
it has conical angle $2(n+1)\pi$.  Denote by $d$ the induced distance.
Completeness of $d$ is tricky in a general setting, and the matter is
settled in~\cite{ChoiLee}. In our case there is an easy proof that we
include for the reader's convenience.  First, we need a family of
nice neighborhoods of the points of $\partial S^+$, that we call
hyperbolic semi-planes.

\begin{definition} A hyperbolic semi-plane in $S^+$ is a closed set
$\delta\subset S^+$, whose closure in $\overline{S^+}$ is a closed
disc and such that, for any lift $\tilde \delta\subset\widetilde S$ of
$\delta$, the restriction of $D$ to $\tilde\delta$
is a homeomorphism onto a closed hyperbolic
semi-plane of $\mathbb H^+$, that is, a sub set of $\mathbb H^+$ isometric to
$\{\Im(z)>0, \Re(z)\geq0\}$.)
\end{definition}

\begin{lemma} \label{l:completeness} The metric space $(S^+, d) $ is
complete.
\end{lemma}
\begin{proof} For every hyperbolic semi-plane $\delta$ in $S^+$, and
every $r>0$, we denote by $\delta_r$ the set of points of $\delta$
which are
at distance more than
$r$ from $\partial\delta$ with respect to the hyperbolic metric of $S^+$.  Observe that for any fixed $r>0$, for $\delta$
varying among all  hyperbolic semi-planes of $S^+$,
the union of all the sets $\delta_r$  is an open set
whose exterior in $S^+$ is a compact set $K_r$.

Let $(p_n)$ be a Cauchy sequence in $S^+$. Let $n_0$ be such that for
$m,n\geq n_0$, the distance between $p_m$ and $p_n$ is less than $1$.

First suppose that there is $m\geq n_0$ such that $p_m $ belongs
$K_1^c$, i.e. $p_m\in \delta_1$ for some hyperbolic semi-plane $\delta
$ in $S^+$. Then because the hyperbolic distance in $\delta$ is not
bigger than the restriction of the distance $d$ to $\delta$, the
points $p_n$ belong to $\delta$ for every $n\geq n_0$, and form a
Cauchy sequence for the hyperbolic distance in $\delta$. Hence, the
sequence $p_n$ has a limit in $\delta$.

The remaining case to consider is when for all $m\geq n_0$, the point
$p_m$ belongs to $K_1$. Since $K_1$ is compact, the Cauchy sequence
$(p_n)$ converges to a point. Thus $(S^+, d)$ is complete.
\end{proof}

Geodesics of components are curves that locally minimize distance. In
fact, they are piecewise smooth geodesics (for the hyperbolic metric
defined outside the branch points) with singularities at
branch points, where they form angles always bigger or equal than $\pi$.

\begin{lemma} \label{l:limit} Let $\gamma : [0,\infty) \rightarrow
S^+$ be a geodesic which exits all compact sets of $S^+$.  Then
$\gamma$ has a limit $\gamma(\infty)\in\partial S^+$.  If
$\gamma(\infty)$ is not a branch point, then $\gamma$ analytically
extends to a curve ending in $S^-$. The statement remains true if we exchange the roles of $S^+$ and $S^-$.
\end{lemma}

\begin{proof} By hypothesis, $\gamma$ eventually exists any $K_r$
(defined as in the proof of Lem\-ma~\ref{l:completeness}), so it
enters a hyperbolic semi-plane $\delta$ and never exits again. The
claim follows because $\delta$ is isometric to a half-plane in the
hyperbolic plane, where geodesics have limits on the boundary.
\end{proof}

\begin{lemma} Let $C$ be a conneccted component of $S\setminus
S_{\mathbb R}$.  The universal cover $\widetilde{C}$ is a
$\mathrm{CAT}(-1)$-space, whose geometric boundary is an oriented circle so
that $\widetilde C \cup \partial\widetilde C$ is a closed disc.
\end{lemma}

\begin{proof} Since the conical singularities at branch points have
angles bigger that $2\pi$, and the metric is hyperbolic elsewhere, the
singular metric $ds^2$ of $C$ can be approximated by smooth metrics of
curvature less than $-1$, hence CAT$(-1)$ inequalities hold for
triangles and pass to the limit. Thus $\widetilde{C}$ is a
$\mathrm{CAT}(-1)$-space.  Let $ds^2 _ {smooth}$ be a smooth metric of
curvature less than $-1$ on $C $, which equals $ds^2$ outside some
compact neighborhood of branch points, and let $d_{smooth}$ be the
induced distance. Then the identity is a quasi-isometry between
$(\widetilde{C}, d)$ and $(\widetilde{C}, d_{smooth})$, hence these
two spaces have the same boundaries. On the other hand, complete,
simply connected, Riemannian surfaces of uniformly negative curvature
are open discs whose geometric and topological boundaries are homeomorphic.
\end{proof}

\begin{cor}\label{cor:cat-1} Any path in a component $C$ of $S\setminus S_{\mathbb{R}}$ is homotopic with fixed
end-points to a unique geodesic. Any closed loop in $C$ which is not
null-homotopic is freely homotopic to a unique closed
geodesic. Between any two points in $\widetilde C$ there is a unique
geodesic. Geodesics of $\widetilde C$ are simple. Two non-disjoint
geodesics of $\widetilde C$ intersect either transversally, or in a
connected geodesic segment (possibly a point) with end-points at
branch points.
\end{cor}

\subsection{Ends of components} Let $C$ be a connected component of
$S\setminus S_{\mathbb R}$.  We identify oriented bi-infinite
geodesics of $\widetilde{C}$ (up to parametrization) with the couples
$(a,b)$ of their end-points in $\partial\widetilde C$. By Jordan's
theorem, any $(a,b)$ divides $\widetilde C$ in two discs.

\begin{definition} Let $(a,b)$ be an oriented geodesic in
$\widetilde{C}$. We denote by $R(a,b)$ and $L(a,b)$ the component of
$\widetilde{C} \setminus (a,b)$ which is respectively at the right and
left-side of $(a,b)$.
\end{definition}

\begin{lemma} \label{l:disjointsemi-planes} Let $a,b,c,d$ be distinct
points.  Then $R(a,b)$ and $R(c,d)$ are disjoint if and only if
$a,b,c,d$ are disposed in a positive cyclic order.
\end{lemma}

\begin{proof} Suppose $a,b,c,d$ are cyclically ordered. Then, $(c,d)$
starts and ends in $L(a,b)$. By Corollary~\ref{cor:cat-1}, it cannot
enters $R(a,b)$ and exits again, so it stays always on its
complement. The orientation of $(c,d)$ tells us that $R(c,d)$ is contained
in $L(a,b)$. The converse is immediate.
\end{proof}

\begin{definition} Let $l_1,\dots,l_k$ be the boundary components of
$C$ (which are components of the real curve).  The {\bf peripheral
geodesic} $\gamma_i$ corresponding to $l_i$ is its geodesic
representative in $C$, oriented as in $\partial C$. The {\bf end}
$E_i$ corresponding to $l_i$ is the connected component of $C\setminus
\gamma_i$ having $l_i$ in its boundary.
\end{definition}

Peripheral geodesics can be complicated.
However, ends are simple.
\begin{lemma} [Annular ends] \label{l:peripheralannulus} Any end of
$C$ is an open annulus.
\end{lemma}

\begin{proof} Let $\overline C$ be a compact surface with boundary
whose interior is $C$. Let $l$ be a component of $\partial C$ and
consider a neighborhood of $l$ in $\overline C$ homeomorphic to
$l\times[0,1)$. Let $l_t=l\times\{t\}$. The length of $l_t$ tends to
$\infty$ for $t\to 0$, so we can choose $t_0$ so that the peripheral
geodesic $\gamma$ corresponding to $l$ belongs to the complement of
the annulus $A_{t_0}=l\times [0,t_0]$.

Any lift $\tilde{l_t}$ has distinct end-points
$a,b\in\partial\widetilde C$, because $\tilde{l_t}$ stays at a finite
distance from the corresponding lift $\tilde{\gamma}= (a,b)$ of
$\gamma$.  For any lift $\tilde{l_t}$ of $l_t$ in $\widetilde{C}$, we
denote by $R(\tilde{l_t})$ the component of $\widetilde{C} \setminus
\tilde{l_t}$ which is on the right of $l_t$. Since it is a topological
disc, $R(\tilde{l_t})$ is the universal covering of $A_t$. Hence, the
discs $R(\tilde{l_t})$ are disjoint for distinct lifts
$\tilde{l_t}$. Thus, if we denote by $a_i,b_i$ the ends of two
distinct lifts $\tilde{\gamma}_ i $ of $\gamma$, then
$a_1,b_1,a_2,b_2$ are in cyclic order, and by
Lemma~\ref{l:disjointsemi-planes}, we get that the discs $R(a_i, b_i)$ are
disjoint.

Hence, the quotient of $R(a,b)$ by the action of $\pi_1(C)$ is the
same as its quotient by the stabilizer of $(a,b)$, ant so it is an
open annulus. Since it is connected, open and closed in $C\setminus
\gamma$, and contains $A_{t_0}$, it is the end $E$ corresponding to $l$.
\end{proof}

Any end is therefore an open annulus $E$ embedded in $C$, but not
necessarily properly embedded. Indeed, there is no reason for the
peripheral geodesic $\gamma$ to be embedded (and in fact in general it
is not). However, from the fact that for any two lifts
$\tilde\gamma_1=(a_1,b_1)$ and $\tilde\gamma_2(a_2,b_2)$ of $\gamma$,
the discs $R(a_1,b_1)$ and $R(a_2,b_2)$ are disjoint, it follows that
the right-side of $\gamma$ in $C$ is well-defined and it is an
embedded annulus (which actually equals $E$).  In other words:

\begin{lemma}\label{halftub} Let $\gamma$ be a peripheral geodesic of
$C$ and $E$ be the corresponding end. For any $\varepsilon >0$ and for
any $x\in\gamma$ the set Right$_\varepsilon(\gamma,x)=\{p\in C: d(p,x)
<\varepsilon\}\cap E$ is non-empty, and the set
Right$_\varepsilon(\gamma)=\cup_{x\in\gamma}$Right$_\varepsilon(\gamma,x)$
is an embedded annulus.
\end{lemma}

\begin{lemma}\label{endisjoint} Ends corresponding to different
components of the boundary of $C$ are disjoint.
\end{lemma}
\begin{proof} Let $l$ and $l'$ be two distinct components of
$\partial C$.  The proof goes as in Lemma~\ref{l:peripheralannulus},
from which we borrow notations. Choose $t,s$ so that the annuli
$A_{l_t}$ and $A_{l'_s}$ are disjoint.  Then, for any two lifts
$\tilde{l_t}$ and $\tilde{l'_s}$, the right components $R(\tilde
l_t)$ and $R(\tilde l'_s)$ do not intersect. By denoting
$a_l,b_l$ the extremities of $\tilde{l_t}$, and similarly $a_l',b_l'$ for
$\tilde{l'_s}$, we get that $a_l,b_l,a_{l'},b_{l'}$ are in
cyclic order. Hence Lemma~\ref{l:disjointsemi-planes} shows that
$R(a_l,b_l)$ and $R(a_{l'},b_{l'})$ are disjoint.  This being
true for any choice of the lifts, we deduce that the ends
corresponding to $l$ and $l'$ are disjoint.
\end{proof}

Note that the closure of different ends may possibly touch.
Nonetheless, as a direct corollary of Lemmas~\ref{halftub}
and~\ref{endisjoint} we get that this happens in a controlled way.

\begin{definition} The exterior angle at a point $x$ of a peripheral
geodesic is the angle that is seen on the right of the geodesic at
$x$.
\end{definition}

\begin{cor} Let $x$ be a branch point in $C$ of angle $2\pi
(n+1)$. The exterior angles of all peripheral geodesics passing throgh
the point $x$ are disjoint. In particular, their sum is not bigger
than $2\pi (n+1)$.
\end{cor}

\subsection{Example: The triangle} \label{ss:_triangle}

Here we describe the example of a bran\-ched projective structure $\sigma$ on a compact surface $S$ with the following properties: the holonomy is Fuchsian, and there exists a component $l$ of the real curve, bounding a negative disc $D$ on the right isomorphic to the lower half plane, and a positive pair of pants $C$ on the left containing a unique branch point (of angle $6\pi$), such that the peripheral geodesic corresponding to $l$ in $C$ is a bouquet of three circles that develops as a geodesic triangle in the upper half plane. Such an example will be called a "triangle". This kind of structure shows up in the proof of the main theorem, see case~\ref{enum:3} of Lemma~\ref{l:triangle}.

We begin by constructing a branched projective structure $\sigma_{\Pi}$ on a pair of pants $\Pi$ with a unique branch point (of angle $6\pi$), whose boundary components are positive geodesics not containing the branch point and whose decomposition into real, positive and negative parts is as follows (see Figure~\ref{fig:pants}):
\begin{enumerate}
\item  the real part $\Pi ^{\mathbb R}$ is the union of $\partial \Pi$ and a bouquet $\mathcal B$ of three circles attached on a branch point of angle $6\pi$,
\item  the negative part $\Pi^-$ consists of the component on the right of $\mathcal B$ being isomorphic to the lower half plane, and
\item the positive part consists of the disjoint union of three hyperbolic annuli on the left of $\mathcal B$.
\end{enumerate}
The structure $\sigma$ will then be obtained from the structure
$\sigma_{\Pi}$ by the following operations: first, moving the branch
point in the positive component (as a point of angle $6\pi$), and then
attaching a pair of pants with geodesic boundary to the boundary of
$\Pi$.
\begin{figure}[httb] \centering
\centering

\includegraphics[width=1.5in]{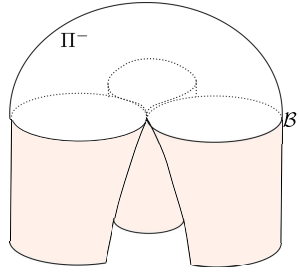}
\caption{The pants for the triangle}\label{fig:pants}
\end{figure}

Let us start with a Schottky group of a pair of pants. To introduce this group, let $\alpha$ and $\beta$ be elements of $\mathrm{PSL}(2,\mathbb R)$ and $A_{\alpha}$, $R_{\alpha}$, $A_{\beta}$, $R_{\beta}$ be disjoint closed intervals in $\mathbb R \mathbb P^1$, such that $\alpha (R_{\alpha} ^c) =  \mathrm{Int} (A_{\alpha} ) $ and $\beta ( R_{\beta} ^c) = \mathrm{Int} (A_{\beta} )$. The group $\Gamma$ generated by $\alpha$ and $\beta$ is a discrete group. The condition that the quotient $\Gamma \backslash \mathbb H^+$ is a pair of pants -- as opposed to a punctured torus -- is that the intervals $A_{\alpha}, R_{\alpha}, A_{\beta}, R_{\beta}$ are in cyclic ordering. Introduce the transformation $\gamma $ in $\mathrm{PSL}(2,\mathbb R)$ such that $\gamma \beta \alpha = \mathrm{id}$.

Let $q$ be a point in the region delimited by the three axes of $\alpha , \beta, \gamma$ in $\mathbb H^+$, and $T$ be the triangle $q= \gamma \beta \alpha (q), \alpha (q) , \beta \alpha (q)$.  The union of the images of $T$ by the elements of $\Gamma  $ is a connected part of $\mathbb H^+$ (see Figure~\ref{fig:triangle}). The quotient of the $1$-squeleton of $T$ in $\Gamma \backslash \mathbb H^2$ is a bouquet of three circles, and the restriction to $T$ of the quotient map $\mathbb H^+ \mapsto \Gamma \backslash \mathbb H^+$ just consists in identifying the vertices of $T$. We aim to find our branched projective structure on $\Gamma \backslash \mathbb H^-$ with the image of the interior of $T$ as the negative component, and the branch point of angle $6\pi$ the image of the vertices. To define this structure we will define its developing map $D: \mathbb H^+ \rightarrow \mathbb C \mathbb P^1$, equivariant with respect to the identity.

\begin{figure}[httb] \centering
\centering
\includegraphics[width=3in]{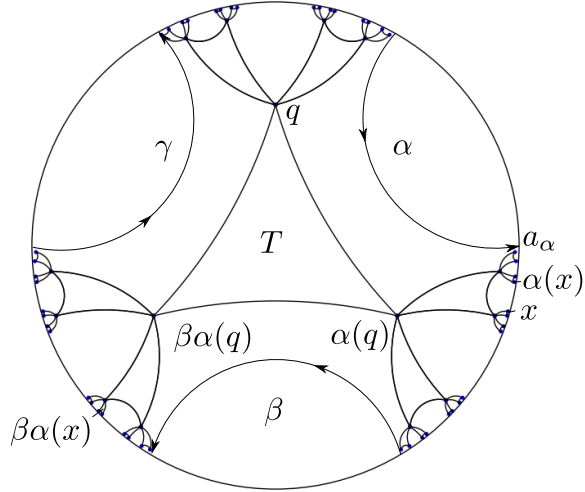}

\caption{Construction of the triangular real curve}\label{fig:triangle}
\end{figure}

The interior of the triangle $T$ should be negative, and should not contain any branch point, so that the developing map in restriction to $\mathrm{Int} (T)$ needs to be a diffeomorphism from $\mathrm{Int}(T)$ to $\mathbb H^-$ (by completeness of the hyperbolic metric in the negative component), that extends to a diffeomorphism from $T$ to $\overline{\mathbb H^-}$. For our purpose, it will be sufficient to consider any diffeomorphism from $T$ to $\overline{\mathbb H ^-}$ such that the points $x = D(q)$, $\alpha (x) $ and $\beta \alpha (x) $ are in cyclic order. We claim that a point $x$ which sits short before the attracting point of $\alpha$ -- i.e. the fixed point $a_{\alpha}$ of $\alpha$ lying in $A_{\alpha}$ -- is such a point. Indeed, $\alpha (x)$ is between $x$ and , and then $\beta \alpha(x)$ is between $a_{\alpha}$ and the attractive fixed point of $\beta$. i.e. the fixed point of $\beta$ lying in $A_{\beta}$.

Hence we have chosen the diffeomorphism from $T$ to $\overline{\mathbb H^-}$ as before, we extend $D$ to the union of the images of $T$ by the group $\Gamma$ using the equivariance relation $D (\gamma z ) = \gamma D (z)$. The complement of the union of the images of $T$ by the elements of $\Gamma$ is an infinite set of semi-planes. There are three particular ones which are the semi-planes $P_ {\alpha}$, $P_{\beta}$ and $P_{\gamma}$ at the left of the piecewise geodesic curves defined respectively by $\cup _{n\in \mathbb Z} [\alpha^n q, \alpha^{n+1} q]$, $\cup _{n\in \mathbb Z} [\beta^n \alpha q, \beta ^{n+1} \alpha q]$ and $\cup _{n\in \mathbb Z} [\gamma^n q, \gamma^{n+1} q]$. These curves are mapped by $D$ to the intervals between the repulsive fixed points and the attractive fixed points of $\alpha$, $\beta$ and $\gamma$ respectively. One extends $D$ to a diffeomorphism from $P_{\alpha}$, $P_{\beta}$ and $P_{\gamma}$ to $\mathbb H^{+}$ which is equivariant with respect to $\alpha$, $\beta$ and $\gamma$ respectively. All the other components of the complement of $\cup_\gamma \gamma T$ is the image of one of the semi-planes $P_{\alpha}$, $P_{\beta}$ or $P_{\gamma}$ by an element of $\Gamma$. Hence, one can extend $D$ to the whole upper half plane $\mathbb H^+$ by equivariance. This defines a branched projective structure $\sigma_{\Pi}$ on the pair of pants $\Pi= \Gamma \backslash \mathbb H^+$. By construction it satisfies conditions (1), (2) and (3). We denote by $p$ the branch point of angle $6\pi$ of this structure, and $A_x$, $x=\alpha,\beta,\gamma$ the three positive annuli of $\sigma_{\Pi}$ (those are the quotients of $P_{\alpha}, P_{\beta}$ ad $P_{\gamma}$ respectively).

To construct an example of a branched projective structure with a "triangle" peripheral geodesic as described above, we move the branch point $p$ of $\sigma_{\Pi}$ in the positive component.
\begin{figure}[httb] \centering
\centering
\includegraphics[width=3in]{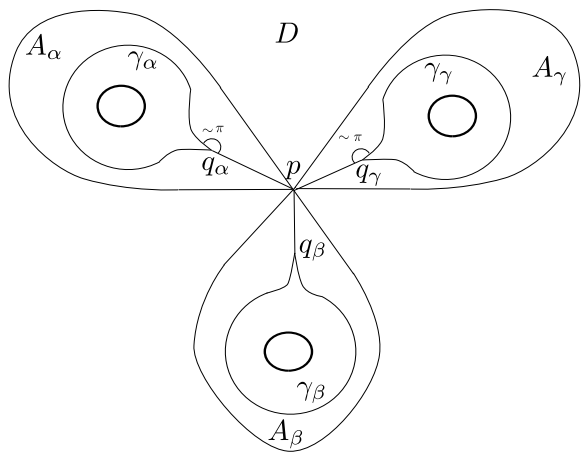}

\caption{Moving the $6\pi$-point to the positive part}\label{fig:moving6pi}
\end{figure}

This movement is done by cutting and pasting $\Pi$ along three curves going from $p$ and entering inside the three positive annuli of $\Pi$ (see Figure~\ref{fig:moving6pi}). We denote these curves by $[p, q_y]$, $y=\alpha,\beta,\gamma$, where $q_y$ are points in the respective annuli $A_y$. After the cut and paste, we get a new structure $(\Pi', \sigma_{\Pi'})$ on a pair of pants, the three points $q_y$'s being identified to a single conical positive point $q$ of angle $6\pi$. We may assume that the segments $(p,q_y]\subset \Pi$ are geodesics. Let $\gamma_y\subset A_y$ be the geodesic loop starting and ending at $q_y$ and making a turn around $A_y$. Observe that, up to shortening the segments $[p,q_y]$, we may assume that the angle between the two branches of $\gamma_y$ at $q_y$ and $[q_y,p]$ is approximately $\pi$, and that $\gamma_y$ intersects $[q_y,p)$ only at $q_y$. This shows that after the cut and paste, the curves $\gamma_y$ produces closed curves $\gamma_y'$ in $\Pi'$ passing through $q$, and that the concatenation $\gamma'_\alpha * \gamma'_\beta * \gamma'_\gamma \subset \Pi'$ is the peripheral geodesic associated to the curve $\partial D$. This is due to the fact that the exterior angles of this curve are approximately $2\pi$ at $q$ (see Figure~\ref{fig:moving6pi}).

Then, to get an example on a compact surface, it suffices to glue on the other side of $\Pi'$ a pair of pants equipped with a non branched projective structure consisting of a positive component being a pair of pants, and three negative annuli attached to it. We leave the details to the reader.

\section{Index formul\ae}
\label{s:index_formula}
We provide useful index formul\ae\ \`a la Goldman (see \cite{Goldman1}),
for branched projective structures with Fuchsian holonomy relating properties of the previously defined real curve decomposition. Again, these formulas extend to the case where the representation is quasi-Fuchsian, but for simplicity we restrict ourselves to the Fuchsian case.

In this section $S$ is a compact surface equipped with a branched
projective structure $\sigma$ with Fuchsian holonomy $\rho$ and
developing map $D$. The assumption that no element in the holonomy is
elliptic will be of particular importance.
 Moreover, we suppose that the real curve $S_{\mathbb R}$
contains no branch points, so that the components of the real curve
are simple closed curves in $S$.  Proposition~\ref{p:RP1structure} and
an analytic continuation argument shows that
the holonomy of any component of the real curve together with its
index (see~\ref{ss:real projective structures}) completely determine
the projective structure in its neighborhood.

Our aim is to describe numerical relations between the topological invariants of the
decomposition, those of the holonomy representation and the
indices of the real curves. In particular, inspired by the techniques used by Goldman
in~\cite{Goldman1} for the case of unbranched structures, we provide a
useful index formula relating the Euler invariant of $\rho$, the Euler
characteristic of the components of $S^\pm$, the number of their
branch points and the indices of their boundaries
(see Theorem~\ref{t:first_index_theorem} below).

Next we will focus on the topological properties of the representation
$\rho$. Recall that given an oriented closed surface with boundary $C$
and a Fuchsian representation $\rho:\pi_1(C)\rightarrow {\rm
PSL}(2,\mathbb R)$ we can naturally associate a $\mathbb{RP}^1$-bundle
$F_{\rho}\rightarrow C$ equipped with a flat connection. Indeed,
$F_{\rho}$ is obtained as the quotient of
$\widetilde{C}\times\mathbb{RP}^1$ by the action of $\pi_1(C)$ such
that for $\gamma\in\pi_1 (C)$ and
$(p,z)\in\widetilde{C}\times\mathbb{RP}^1$

\begin{equation}
\label{eq:id_rho_action} \gamma \cdot (p, z) = (\gamma (p) , \rho
(\gamma) (z)).
\end{equation} If the boundary is empty, we can define the Euler
number of $\rho$ as the element $eu(\rho)=eu(F_{\rho})\in
H^2(C,\mathbb{Z})=\mathbb{Z}$ defined by the Euler class of the bundle
$F_{\rho}$. Otherwise, if there are no elliptic elements, over each
component $l\subset\partial C$ we can define a section of
$s_{\rho}:l\rightarrow (F_{\rho})|_{l}$ by following a fixed point of
the action of $\rho(l)$ on $\mathbb{RP}^1$ along $l$ with the use of
the connection. If $\rho(l)$ is the identity or loxodromic, the
homotopy class of the section is independent of the chosen fixed
point. As we will show shortly we can associate an Euler number
$eu(\rho)\in\mathbb{Z}$ to the representation by using the pair
$(F_{\rho},s_{\rho})$.
In the sequel we will prove the following

\begin{theorem}[First Index Formula]
\label{t:first_index_theorem} Let $S$ be a compact surface equipped
with a BPS $\sigma$ with Fuchsian holonomy. Suppose no
branch point belongs to $S_{\mathbb R}$.  Let $C$ be a component of
$S\setminus S_{\mathbb{R}}$ with the orientation induced by that of
$S$ and denote by $\rho_{C}$ the restriction of $\rho$ to $\pi_1(C)$.

If $k$ denotes the number of branch points in $C$ and $l_1,\ldots l_n$
are the components of $\partial C\subset S_{\mathbb{R}}$, then $$\pm
eu(\rho_C)=\chi(C)+k-\sum_{i=1}^n I_{l_i}$$ where the sign is positive
if $C\subset S^+$ and negative otherwise.
\end{theorem}

\begin{cor}[Second Index Formula]
\label{c:secondindextheorem} If there are no branch points on the
real curve and $k^\pm$ denotes the number of branch points contained
in $S^{\pm}$ then $$eu (\rho)=(\chi(S^+)+k^+)-(\chi(S^-)+k^-).$$
\end{cor}

For the proof of the theorem it will be convenient to have the theory
of Euler classes of sections of oriented circle bundles at hand.

Let $F\rightarrow C$ be an oriented $\mathbb{RP}^1$-bundle over a
compact oriented surface with boundary $C$. For each section
$s:\partial C\rightarrow \mathbb{RP}^1$ we define the Euler number
$eu(F,s)$ as follows. Consider a triangulation $\tau$ of $C$ such that
over each triangle $T$ of $\tau$ the bundle is isomorphic to
$T\times\mathbb{RP}^1$. By connectedness of $\mathbb{RP}^1$ the
section $s$ can be extended continuously to a section $\overline{s}$
defined on the 1-skeleton of $\tau$. The restriction of $\overline{s}$
to $\partial T$ can be thought of as a map $s:\partial
T\rightarrow\mathbb{RP}^1$ that has degree $n_T\in \mathbb{Z}$ with
respect to the given orientations. The sum $\sum n_T$ can be shown to
be independent of the triangulation and the chosen extension
$\overline{s}$ through basic algebraic topology methods. This allows
to define $$eu(F,s)=\sum n_T.$$ In fact $eu(F,s)$ depends only on the
homotopy class of $s$.

\begin{rem}\label{r:annulus}  If
$C=\mathbb{S}^1\times[0,1]$ is an annulus and $s=\{s_i\}$ is a section
of $F$ over $\mathbb{S}^1\times \{i\}$ for $i=0,1$ then $eu(F,s)=\deg
f$ where $f:\mathbb{S}^1\rightarrow \mathbb{RP}^1=\mathbb{S}^1$ is
such that $s_0=f\cdot s_1$ and the degree is computed with respect to
the orientation induced by $C$ on the component where $s_0$ is
defined.
\end{rem}

The following lemma is immediate.

\begin{lemma} \label{l:Euler_decomposition} Let $F$ be an oriented $\mathbb{RP}^1$-bundle over $\overline{C}$ and $\{\lambda_i\}$ be a
finite family of disjoint simple closed curves in $\overline{C}$
containing the boundary components of $C$. Let $s$ be a continuous
section of $F$ defined on $\cup_i \lambda_i$. Denote by $\{ C_j \}_j$
the collection of the closure of connected components of $C \setminus
\big( \cup _i \lambda_i \big)$. Then
\[ eu (F , s|_{\partial C} ) = \sum _j eu (F|_{C_j}, s|_{\partial C_j}
) . \]
\end{lemma}

To abridge notations, rename $\rho_C$ as $\rho$. We
define $$eu(\rho):=eu(F_{\rho},s_{\rho})$$ where the pair
$(F_{\rho},s_{\rho})$ was defined by the relations in
(\ref{eq:id_rho_action}), shortly before the statement of
Theorem~\ref{t:first_index_theorem}.
\begin{rem}\label{rem:eulerfuchsian}
 If $S$ is a compact surface and $\rho:\pi_1(S)\rightarrow \mathrm{PSL}(2,\mathbb{C})$ is a Fuchsian representation, by using the uniformizing structure on $S$ it is easy to show that for any incompressible subsurface $C\subset S$ we have  $$eu(\rho|_{\pi_1(C)})=\chi(C).$$
\end{rem}

Let us proceed to the proof of Theorem~\ref{t:first_index_theorem}. Given the connected component $C$ of $S\setminus S_{\mathbb{R}}$, we
introduce $E=\mathbb{P}^+(TC)$ the $\mathbb{RP}^1$-bundle over $C$
whose fibre over $p\in C$ is the set of semi-lines in $T_pC$.  For
each branch point $p\in C$ we consider a small open disc $B$ in $C$
centered at $p$.  We number such discs $B_1,\dots,B_k$ and call
$\lambda_i$ their boundary curves. On the other hand, for each
boundary component $l_i$ of $\partial C$ consider a curve $\bar{l}_i$
in $C$ that is isotopically equivalent to $l_i$ in
$\overline{C}\setminus\cup_j B_j$. The proof of
Theorem~\ref{t:first_index_theorem}
consists in using the developing map $D$ to define a
bundle isomorphism $\mathrm{D}:E\rightarrow F_{\rho}$ over
$C\setminus\cup_j B_j$ , which allows to define a section of
$F_{\rho}$ over the family of curves
$\{\bar{l}_1,\ldots,\bar{l}_n,\lambda_1,\ldots, \lambda_k\}$ and apply
Lemma~\ref{l:Euler_decomposition}. The conclusion will follow from
the knowledge on the topology of the associated decomposition and the
properties of $D$.

Consider a lift $\widetilde{C^*}\subset \widetilde{S}$ of
$C^*:=C\setminus\{\textrm{branch points}\}$. The restriction of the
developing map $D$ to $\widetilde{C^*}$ or its complex conjugate
defines a local diffeomorphism $\mathbf{D}:\widetilde{C^*}\rightarrow
\mathbb{H}^2$ that preserves orientation if $C$ is positive and
reverses it otherwise. In either case $\mathbf{D}$ induces a
map $$\mathbb{P}^+(T\widetilde{C^*})\rightarrow \widetilde{C^*}\times
\mathbb{P}^+(T\mathbb{H}^2)$$ defined by $(p,[v_p])\mapsto
(p,[d\mathbf{D}_p(v_p)])$, where the brackets denote equivalence
classes under multiplication by a positive real number. Recall that
the complete hyperbolic metric on $\mathbb{H}^2$ induces a map
$\infty: P^+(T\mathbb{H}^2)\rightarrow \mathbb{RP}^1$ that is
equivariant under the natural actions of ${\rm PSL}(2,\mathbb{R})$ on
source and target. Indeed, for each point $(p,[v_p])\in
P^+(T\mathbb{H}^2)$ we associate the point
$\infty(p,[v_p])\in\partial\mathbb{H}^2=\mathbb{RP}^1$ obtained by
following the unique geodesic passing through $p$ tangent to $v_p$
until infinity in the direction of $v_p$. This allows to consider the
map
$$\mathbb{P}^+(T\widetilde{C^*})\rightarrow
\widetilde{C^*}\times\mathbb{RP}^1$$ defined by $(p,[v_p])\mapsto
(p,\infty(p,[d\mathbf{D}_p(v_p)]))$ which is equivariant with respect
to the actions of $\pi_1(C)$ on $\mathbb{P}^+(T\widetilde{C^*})$ by
deck transformations and on $\widetilde{C^*}\times\mathbb{RP}^1$ by
$id\times\rho$.  Hence it induces an isomorphism of
$\mathbb{RP}^1$-bundles over
$C^*$ $$\mathrm{D}:\mathbb{P}^+(TC^*)=E|_{C^*}\rightarrow
(F_{\rho})|_{C^*}.$$ Next we define a section $t$ of $F_{\rho}$ over
the family of curves
$$L=\{l_1,\ldots,l_n,\bar{l_1},\ldots,\bar{l_n},\lambda_1,\ldots,\lambda_k\}.$$
Over each of the boundary components $l_i$, $t$ is the section defined
by a fixed point of $\rho(l_i)$; over any other component $c$, $t$ is
the image by $\mathrm{D}$ of the section $p\mapsto (p,c'(p))$ of $E$
where the orientation of the parametrization is that of $l_i$ if
$c=\bar{l_i}$ and that of $\partial B_i$ if $c=\lambda_i$.  The
complement of $L$ in $\overline{C}$ is a disjoint union of annuli
$A_1,\ldots, A_n$ each having exactly one boundary component in
$\partial C$, discs $B_1,\ldots,B_R$ and a component $C'\subset C$. By
Lemma~\ref{l:Euler_decomposition}
\begin{eqnarray*} & &eu(\rho)=eu(F_{\rho},t)\\
&=&eu\big((F_{\rho})|_{C'},t|_{\partial C'}\big) +\sum_{j=1}^R
eu\big((F_{\rho})|_{B_j},t|_{\partial B_j}\big) +\sum_{i=1}^n
eu\big((F_{\rho})|_{A_i},t|_{\partial A_i}\big).
\end{eqnarray*}

Now, since $\mathbf{D}$ is a local diffeomorphism when restricted to a
lift $\widetilde{C'}\subset \widetilde{C}$ of $C'$, $\mathrm{D}|_{C'}$
is a bundle isomorphism $E|_{C'}\rightarrow (F_{\rho})|_{C'}$ and
hence
$$eu\big((F_{\rho})|_{C'},t|_{\partial C'}\big)=\pm\chi(C')$$
where the sign is positive if $\mathbf{D}$ preserves orientation and
negative otherwise. On the other hand since $\mathbf{D}$ has a single
simple branch point on the disc $B_j$,
$$eu\big((F_{\rho}) |_{B_j},t|_{\partial B_j}\big)=\deg (t|_{\partial
  B_j})=\pm 2$$ where the sign is positive if $\mathbf{D}$ preserves
orientation and negative otherwise.  Finally for an annulus $A_i$
denote by $\phi: l_i\rightarrow \bar{l}_i$ the homeomorphism induced
by the isotopy joining $l_i$ with $\bar{l}_i$.  As noted in
Remark~\ref{r:annulus}, if we write $t|_{l_i}=f\cdot
(t|_{\bar{l}_i}\circ\phi)$, then by the definition of the index of
$l_i$
$$eu\big((F|_{\rho})_{A_i},t|_{\partial A_i}\big)=\deg f=\mp I_l.$$
Since $\deg f$ is measured with respect to the orientation induced on
$l_i$ by that of $C$ , the sign is negative if $\mathbf{D}$ preserves
orientation and positive otherwise. By summing up we get
$$eu(\rho)=\pm
\big(\chi(C')+2k-\sum_{i=1}^{n} I_{l_i}\big)
=\pm\big(\chi(C)+k-\sum_{i=1}^{n} I_{l_i}\big)$$ where the sign is
positive if if $\mathbf{D}$ preserves orientation and negative
otherwise.  This finishes the proof of
Theorem~\ref{t:first_index_theorem}.

For the proof of Corollary~\ref{c:secondindextheorem}, by
considering over each $l_i$ the section of $F_{\rho}$ associated to a
fixed point of $\rho(l_i)$, and applying
Lemma~\ref{l:Euler_decomposition},
we have $$eu(\rho)=eu(\rho^+)+eu(\rho^-)$$ where
$eu(\rho^{\pm})$ denotes the Euler number of $\rho$ restricted to
$\pi_1(S^\pm)$. An instance of Theorem~\ref{t:first_index_theorem} on
each connected component of $S^{\pm}$ and the fact that each curve
$l_i$ is the boundary of exactly one positive and one negative
component give
\begin{eqnarray} eu(\rho) &=& \big( \chi(S^+) + k^+
-\sum_{i=1}^{n}I_{l_i} \big) - \big( \chi( S^- ) + k^-
-\sum_{i=1}^{n}I_{l_i}\big)\nonumber\\ &=& \big( \chi( S^+ ) + k^+
\big) - \big( \chi(S^-) + k^- \big)\nonumber.
\end{eqnarray}

As another application of Theorem~\ref{t:first_index_theorem}, we note
that if $\chi(S)\leq0$, one has $\chi(S)=eu(\rho)$ because $\rho$ is
Fuchsian. From $\chi(S)=\chi(S^+)+\chi(S^-)$ we therefore obtain
$$2\chi(S^-)=k^+-k^-.$$

\begin{cor}\label{cor:eulerannulus}
  If $A$ is an annulus with loxodromic holonomy $\rho$ then $eu(\rho)=0$.
\end{cor}

\section{Grafting and bubbling}\label{s:graftbub}

In this section we will prove that grafting can be obtained by a bubbling followed by a debubbling, as was stated in Theorem ~\ref{l:graftbub} in the Introduction. We recall that a graftable curve is a simple closed curve with loxodromic holonomy such that the developing map is injective on one of its lifts. A more precise restatement of Theorem~\ref{l:graftbub} is:

\begin{theorem}\label{t:graftingvsbubbling}
Let $\sigma$ be a BPS on a surface $S$ and $\gamma$ be a graftable simple closed curve in $S$ that does not pass through the branch points of $\sigma$. Then the grafting of $\sigma$ along $\gamma$ can be obtained by a bubbling followed by a debubbling on $\sigma$.
\end{theorem}

\begin{proof}
First, remark that a small annular neighborhood $U$ of $\gamma$ has a lift in the universal cover that develops injectively in $\mathbb{CP}^1$. Everything will take place in that annular neighborhood. The whole process of bubbling and debubbling is sketched in Figure~\ref{fig:bubdebub}, and details are described below.

\begin{figure}[httb]\centering
\includegraphics[width=6in]{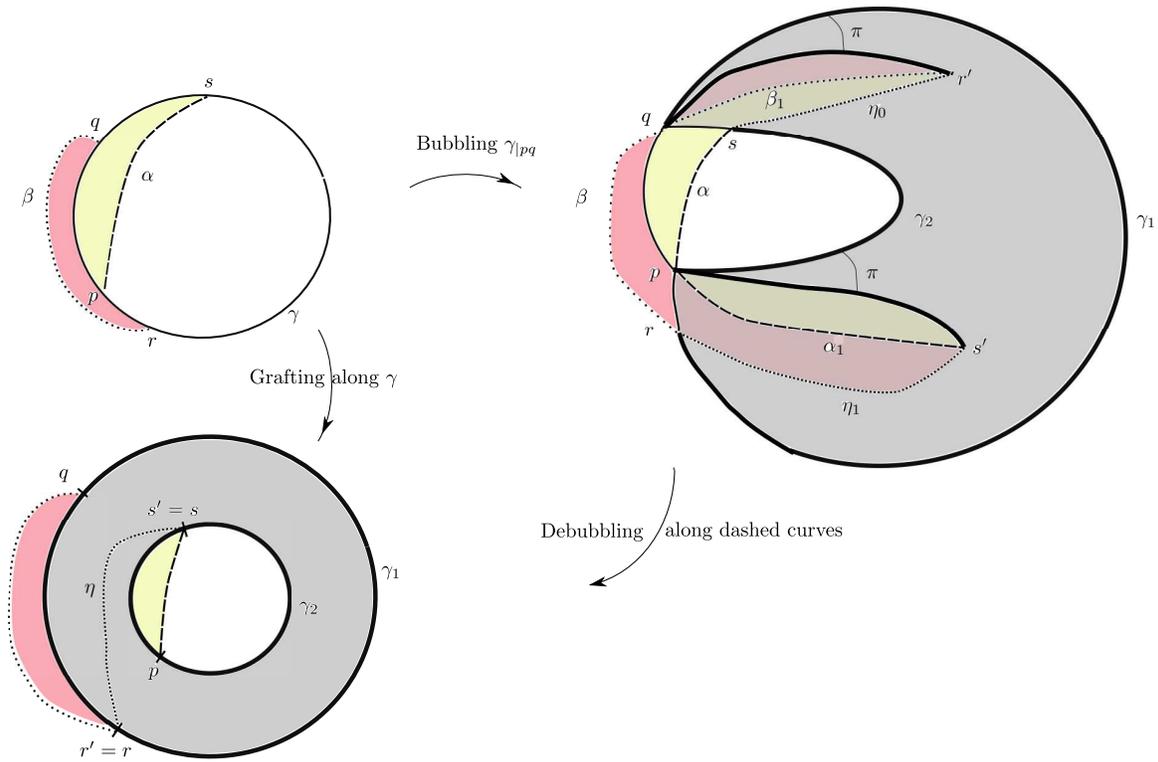}

\caption{Grafting can be obtained by bubbling and debubbling}\label{fig:bubdebub}
\end{figure}

We choose an orientation for $\gamma$. Consider four points $s,q,p,r\in\gamma$ in cyclic order. Choose paths $\alpha$ and $\beta$ joining $s$ to $p$ and $q$ to $r$, obtained by pushing the segments $[s,p]$ and $[q,r]$ on the left and on the right side of $\gamma$ respectively, as in the upper left corner of Figure~\ref{fig:bubdebub}. 
We denote by $\widetilde{\gamma}$ a
lift of $\gamma$ to the universal cover of $S$, $\widetilde{U}$ the corresponding lift of $U$, by  $D$ the developing
map of $\sigma$ and by $\rho$ the (loxodromic) holonomy of $\gamma$. Consider one of the images
$p_0,q_0,r_0,s_0, \alpha_0,\beta_0\subset\mathbb{CP}^1$ by $D$ of each
of the corresponding elements in the initial situation and call
$p_1,q_1,r_1,s_1, \alpha_1,\beta_1\subset\mathbb{CP}^1$ the images of
the latter by $\rho$ (see Figure~\ref{fig:imagebubble}).

\begin{figure}[httb]\centering
\includegraphics[width=5in]{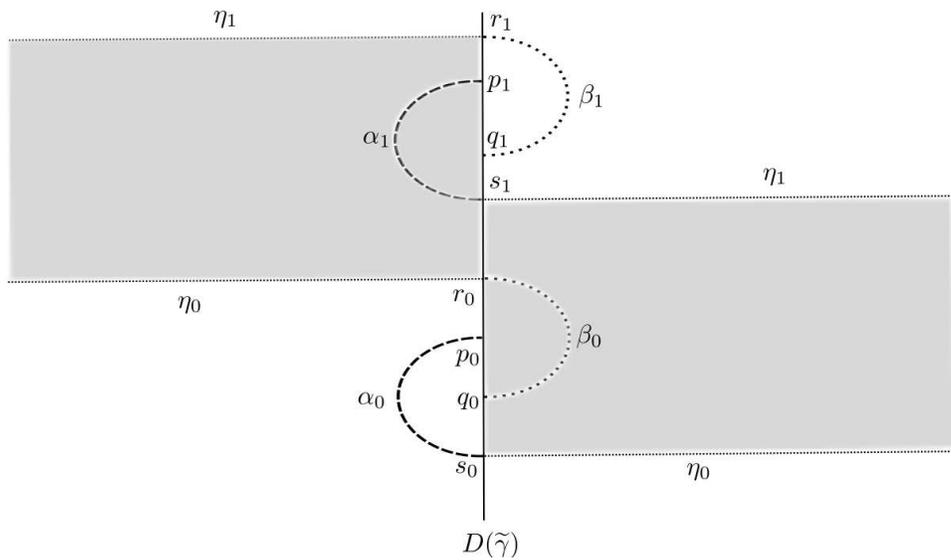}

\caption{Developed image in $\mathbb{CP}^1$}\label{fig:imagebubble}
\end{figure}

Consider the annulus
$A=(\mathbb{CP}^1\setminus\overline{D(\widetilde{\gamma})})/<\rho>$
equipped with its natural projective structure. Still denote by $\alpha$ and $\beta$ of the image in $A$ of the $\alpha_i$'s and $\beta_i$'s. We denote by $s,q,p,r$ their respective extremities in $\partial A$ (therefore $q,r$ lie in a component of $\partial A$ and $p,s$ on the other one). In particular we can
find a simple arc $\eta$ in $A$ joining the points $r$ to $s$
and avoiding the developed images of $\alpha$ and $\beta$. In
Figure~\ref{fig:imagebubble} we find a sketch of two lifts $\eta_0$ and
$\eta_1$ of $\eta$ to $\mathbb{CP}^1$. They will be important
for the construction of twin paths.

Next we  consider the bubbling $\mathrm{bub}(\sigma)$ of $\sigma$ along the oriented arc $[p,q]$ of
$\gamma$ between $p$ and $q$ (the one that contains the points $r$ and $s$), see Figure~\ref{fig:bubdebub}, right side. This is obtained by cutting $\sigma$ along $[p,q]$ and $\mathbb{CP}^1$ along $[p_0,q_1]$ (see Figure~\ref{fig:imagebubble}) and pasting together. Two
branch points of angle $4\pi$ appear at $p$ and $q$. The oriented segment $[p,q]$ in $\sigma$ is separated into twins segments $[p,q]_{\rm right}$ and $[p,q]_{\rm left}$ in $\mathrm{bub}(\sigma)$. Observe that the arcs $\alpha$ and $\beta$ in $S$ survive after the bubbling. We denote their
endpoints in $\mathrm{bub}(\sigma)$ with the same
letters, $s,p,q,r$ as before the bubbling ($r$ lives in $[p,q]_{\rm right}$ and $s$ lives in $[p,q]_{\rm left}$).


We proceed now to see that there is another bubble in $\mathrm{bub}(\sigma)$, such that its debubbling is the grafting $\textrm{Gr}_{\gamma}(\sigma)$ of $\sigma$ along $\gamma$. To identify a bubble it is sufficient to find a pair of twin paths that join two simple branch points, bound a disc, and develop to a segment. Consider the paths in $\mathrm{bub}(\sigma)$ starting at $q$,
\[ \tau_1 = \beta \star \eta_0 \star \alpha_0  \ \ \ \text{and}\ \ \ \tau_2 = \beta_1 \star \eta_1 \star \alpha, \]
that are drawn with dashed lines on the right part of Figure~\ref{fig:bubdebub}. These paths are twins. Moreover, we claim that $\tau_1 * \tau_2^{-1}$ bounds a disc. Indeed, the curve $[p_0,r_0]\star \eta_0 \star \alpha_0$ in $\mathbb{CP}^1\setminus [p_0,q_1]$ bounds a disc $\mathrm{Disc}_0$, and similarly $\beta_1 \star \eta_1 \star [s_1,q_1]$ bounds a disc $\mathrm{Disc}_1$. Clearly $\beta \star [p,r]^{-1} \star \alpha^{-1} \star [s,q]$ bounds a disc in $\sigma$. These three discs glue together to a disc bounded by $\tau_1 * \tau_2^{-1}$.


We are left to prove that the debubbling of these twin paths
produces a new branched projective structure that coincides with $\textrm{Gr}_{\gamma}(\sigma)$. To this end, it
suffices to find
two parallel curves whose developed image is precisely
$D(\widetilde{\gamma})$, bounding an annulus with the projective
structure of $A$, and whose complement has the structure induced by $\sigma$ on $S \setminus \gamma$.

 We proceed to analyze the preimages of
$D(\widetilde{\gamma})$ via the developing map of $\mathrm{bub}(\sigma)$
to identify such curves. Denote by $\gamma_1$ the segment $[r,q]\subset [p,q]_{\mathrm{right}}$ and $\gamma_2$ the segment $[p,s]\subset [p,q]_{\mathrm{left}}$. as in
Figure~\ref{fig:bubdebub}. Remark that the twin path of the segment
$[s,p]$ in $\gamma$ starting from $p$ is the segment $[p,s_0]$ (in the bubble $\mathbb{CP}^1 \setminus [p_0,q_1]$)
that joins $p=p_0$ and $s_0$ and that does not enter the second
bubble till $s_0$, as $\mathrm{Disc}_0$ is delimited by $\alpha_0$, see Figure~\ref{fig:imagebubble}.
Similarly, the twin path $[q,r_1]$ of the segment $[q,r]$ joins $q$ and $r_1$ in the first bubble without entering the second
bubble. These twins correspond to the thick segments inside the shaded bubble
in  Figure~\ref{fig:bubdebub}). Thus they
appear in the structure after the debubbling and will have developed
image contained in $D(\widetilde{\gamma})$. By construction, after the
debubbling their union with $\gamma_1$ and $\gamma_2$ form a pair of
parallel closed curves having the same developed image.

The shaded part in Figure~\ref{fig:imagebubble} is, by
construction, a fundamental domain for the action of $\rho$ on
$\mathbb{CP}^1\setminus\textrm{Fix}(\rho)$. Moreover, the region bounded by $\alpha_0 \star [p,s_0]$ is projectively equivalent to the region delimited by $\alpha \star [s,p]^{-1}$. Similarly, the region bounded by $\beta \star [r_1,q]$ is projectively equivalent to the one delimited by $\beta \star [q,r]^{-1}$. This implies both properties we need: namely that the projective structure in the region between $\tau_1$ and $\tau_2$ after debubbling
coincides with $A$, and that the structure on the complement of $A$ is the one induced by $\sigma$ on $S\setminus \gamma$.
\end{proof}

\section{Finding embedded twin paths}\label{ap:wecanmove}
In this section, we give two criteria to ensure that a pair of twin
paths is embedded. This is necessary to perform all the movements of branch points we carry in
Sections~\ref{s:debubbling},~\ref{s:deg}, and~\ref{s:moving_positive}.

The pathologies that one has to avoid are mainly two. Suppose that we
have a geodesic ray $\tau$ emanating from a branch point and want to
follow its twin $\tau'$, which is locally well-defined. Even if $\tau$ is
embedded it could happen that $\tau'$ is wild (remark that a bubbling introduces  a whole copy of the universal cover of the surface!). Secondly, it could happen that
$\tau'$ crosses $\tau$ very soon, say
$\tau(\varepsilon)=\tau'(\varepsilon)$ at a smooth point (as in fact
happens in a conical cut and paste described at
page~\pageref{conicalcutandpaste}) with no a priori control on
$\varepsilon$. In both cases a cut and paste procedure would change  the
topology of $S$.

Here we prove two lemmas. The first one ensures that if we follow the
pre-image of a geodesic under a projective map, the twin paths we
obtain are in fact embedded. The second shows that if the
holonomy is Fuchsian, then pathologies like the conical cut and
paste cannot occur. Both lemmas rely on the hypothesis of Fuchsian holonomy and their falseness in  more general settings constitutes one of the main obstructions to generalize the arguments to other types of representations.

\begin{lemma}[Twin geodesics are embedded] \label{l:embedded_twins}
Suppose $S$ is a surface equipped with a BPS having Fuchsian
holonomy. Let $U\subset S^\pm$ be an open domain in $S^\pm$ with smooth
boundary and corners. Let
$\Sigma$ be a complete hyperbolic surface, and $f: U\rightarrow
\Sigma$ be a local isometry (on the complement of branch points).
Let $T\in [0,\infty]$ and $(\tau_1,\tau_2)$ be a pair of twin
geodesics $\tau_i : [0, T)\rightarrow U$ starting at a branch point
$p\in\overline U$ such that
\begin{itemize}
\item for every $i=1,2$ and $t\in (0,T)$, $\tau_i(t)$ belongs to $U$
and is not a branch point of $S$,
\item $f\circ \tau_1 = f\circ \tau_2 = \tau$ is a properly embedded geodesic
in $\Sigma$.
\end{itemize} Then, $(\tau_1, \tau_2)$ is a pair of embedded twin
paths in $S$.
Moreover, suppose that $\Sigma$ does not have parabolic ends and that
$T= +\infty$. Then $\tau_i(t)$
tends to a point $u_i$ in $S_{\mathbb R}$ when $t$ tends to infinity,
for $i=1,2$, with $u_1 \neq u_2$.
\end{lemma}

\begin{proof}
Each of the paths $\tau_i$'s are embedded since $\tau$
is embedded. The first part of the lemma says that the images of
$\tau_1$ and $\tau_2$ are disjoint. We argue by contradiction. Suppose
that there are two numbers $0\leq s,t < T$, not both equal to $0$,
such that $\tau_1 (s)= \tau_2(t)$. Because $\tau$ passes once through
the point $p$, both $s$ and $t$ are positive. By exchanging the roles of
$\tau_1$ and $\tau_2$ if necessary, we can suppose that $s\geq t$.

At the point $q = \tau_1(s) = \tau_2 (t)$, the geodesics $\tau_1$ and
$\tau_2$ cannot be transverse, because the map $f$ is a local
diffeomorphism at $q$.
 Hence, we necessarily have $\tau_ 1 (s + u ) =
\tau_2 (t+u)$ for small values of $u$, or $\tau_1 (s-u) = \tau_2
(t+u)$ for small values of $u$. In the first case, we have $\tau_1
(s+u) = \tau_2(t+u)$ for every $u \geq -t$ by analytic
continuation. Because $\tau_1$ and $\tau_2$ are different, we have $s
>t$. At $u= -t$, we find $\tau_1 (s-t) = \tau_2 (0) = p$. Hence
$\tau(s-t) = p$ which contradicts that $\tau$ is embedded. In the
second case, we get $\tau_1 (s-u) = \tau_2 (t+u)$ for $0\leq u\leq
(s-t)/2$ by analytic continuation (note that $t+(s-t)/2\leq
s<T$). In particular we get
$\tau_1(\frac{s+t}{2}) = \tau_2(\frac{s+t}{2})$.
At $u=(s+t)/2>0$, we therefore obtain that $f$ is a branched covering,
contradicting the hypothesis that $\tau_1(t)$ is not a branch point of
$S$ for $t>0$. The first claim is proved.

Since $\tau$ is properly embedded both $\tau_1$ and $\tau_2$
must exit any compact
set of $S$ as $t\to\infty$, otherwise an accumulation point would
exists. By Lemma~\ref{l:limit} both $\tau_i$'s  have limits $u_i$ as
$t\to\infty$. Such limits must belong to $S_{\mathbb R}$ because
both $\tau_i$ exist all compact.

If $u_1 = u_2$, then $\tau_1$ and $\tau_2$ are exponentially
asymptotic at infinity in $U$ for the hyperbolic distance. However,
they have the same image $\tau$ by $f$, and $\Sigma$ has no parabolic
end, so this is impossible. Hence the limits are distinct, and the
lemma is proved. \end{proof}

Suppose that the holonomy of $S$ is Fuchsian.
For a component $C\subset S^\pm$ we denote by $C\f$ the
hyperbolic surface $\rho(\pi_1(C))\backslash\mathbb H^\pm$. The following lemma shows that we can always move branch points at least a distance bounded below by the injectivity radius of
$C\f$.

\begin{lemma}[Local movements]\label{wecanmove}
Let $S$ be a closed surface equipped with a BPS with Fuchsian
holonomy. Let $C$ be a component of $S^\pm$ and let $\varepsilon>0$ be
smaller than the injectivity radius of $C\f$.
Let $\gamma$ be a geodesic segment starting from a branch point $p$ of
$C$ and shorter than $\varepsilon$. Let $\gamma'$ be a geodesic
segment starting from $p$, of the same length as $\gamma$,
and forming with $\gamma$ at $p$ an angle $2k\pi$,
with $0<k\in\mathbb N$.
Suppose that both $\gamma$ and $\gamma'$ do not
contain branch points other than $p$.
Then $\gamma$ and $\gamma'$ form a pair of embedded twin paths.
 \end{lemma}
\proof
A developing map for $S$ induces a map $f:C\to C\f$ which is a local
isometry. The image of $\gamma$ is therefore a geodesic in
$C\f$. Since $\gamma$ is shorter than the injectivity radius of $C\f$,
then $f(\gamma)$ is properly embedded. As $\gamma'$ forms an angle
$2k\pi$, we have $f(\gamma)=f(\gamma')$ and
Lemma~\ref{l:embedded_twins} concludes.
\qed

\section{Debubbling adjacent components} \label{s:debubbling}

In this section, we give a criterion ensuring that a BPS can be
debubbled, after possibly moving the branch points. The main result
is
the following.

\begin{theorem} [Debubbling] \label{t:debubbling} Let $S$ be a
compact surface equipped with a BPS $\sigma$ having Fuchsian
holonomy. Suppose that there exists a positive and a negative
component, that we denote $C^+$ and $C^-$, with a common boundary
component $l$, such that
\begin{enumerate}
\item the index of $l$ is $1$, and its holonomy is loxodromic,
\item the index of any component of $\partial C^+$ or $\partial C^-$
other than $l$ vanishes,
\item each component $C^+$ and $C^-$ contains a single branch point of
angle $4\pi$.
\end{enumerate} Then, after possibly moving the branch points in the
components $C^+$ and $C^-$, the branched projective structure on $C^+
\cup C^-$ is a bubbling.
\end{theorem}

Before entering into the details, let us explain the strategy for the
proof of this result, and introduce the notion of
half-bubble:

\begin{definition} [Half-bubble] \label{def:half-bubble} Given a
positive or negative component $C$ of $\sigma$ and a component $l$ of
$\partial C$, a {\bf half-bubble} in the direction of $l$ is a pair of
embedded twin geodesics $(\tau_1, \tau_2)$ contained in $C$ and
tending at different points $u_1$ and $u_2$ of $l$ at infinity, such
that $C \setminus (\tau_1 \cup \tau_2)$ has two connected components,
one of them being isometric via the developing map to $\mathbb H^2$ minus a
semi-infinite geodesic. We require moreover that the oriented angle
$\measuredangle \tau_1 \tau_2$ is the $2\pi$-angle of
that region. \end{definition}

In other words, the branched $\mathbb H^2$-structure on $C$ has been
obtained from another branched $\mathbb H^2$-structure $C'$ by
inserting a hyperbolic plane with a cut and paste procedure along a
properly embedded semi-infinite geodesic of $C'$.

The proof of Theorem~\ref{t:debubbling} consists in finding
half-bubbles in the direction of $l$ in each of the components $C^+$
and $C^-$ and ensure that, after possibly moving the branch points in $C^+$ and
$C^-$, they glue together to produce a bubble. This is done in Proposition~\ref{p:existenceofhalf-bubble}.

Recall that a connected subsurface $C\subset S$ is called {\bf
  incompressible} if any loop in $C$, which is homotopically trivial in
$S$, is also homotopically trivial in $C$.

\begin{lemma} \label{l:incompressible} The components $C^+$ and $C^-$ are
incompressible in $S$.
\end{lemma}

\begin{proof} A well-known criterion for
a connected subsurface of $S$
to be incompressible is that its boundary components are not
homotopically trival in $S$.
By hypothesis $l$ has loxodromic holonomy. Since any
component $l'$ of $\partial C^+$ or $\partial C^-$ different from $l$ has index $0$, its
holonomy is non-trivial (see Proposition~\ref{p:RP1structure}.) Thus,
neither $l$ nor any $l'$ can be homotopically trivial.
\end{proof}
Remark that
Lemma~\ref{l:incompressible}, Remark \ref{rem:eulerfuchsian} and the
index formula~\ref{t:first_index_theorem} show
that necessarily $C^-$ is an annulus. However, $C^+$ may have more
complicated topology.

It is necessary now to fix some {\bf notation}.
Let $\rho$ denote the holonomy
of a developing map $D$ for $\sigma$. The facts we are going to prove
hold true for both $C^+$ and $C^-$. For lightening
notations we fix $C=C^+$. In order to obtain the proofs for $C^-$ one
has just
to replace the upper half-plane model for $\mathbb H^2$ with the lower
half-plane model for $\mathbb H^-$.

 Let $\pi : \widetilde{S} \rightarrow S $ be the universal covering of $S$, with
covering group $\pi_1(S)$. We chose a connected component $\widehat C$
of $\pi^{-1}(C)$. The restriction of $\pi$ to $\widehat{C}$
is a Galois covering over $C$, with
Galois group the stabilizer of $\widehat{C}$ in $\pi_1(S)$. We set
$\pi_1(C)=\operatorname{Stab}(\widehat C)<\pi_1(S)$
(note that using this notation, in
general the group
$\pi_1(C)$ could be different from the fundamental group of $C$, for
instance if $C$ were compressible in $S$).
As above, we denote by $C_{fuchs}$ the complete hyperbolic surface
\[ C_{fuchs} := \rho(\pi_1(C)) \backslash \mathbb H^2.\] The
restriction of $D$ to $\widehat{C}$, induces a map $D_C : C
\rightarrow C_{fuchs}$, which is a local isometry.

The topology of $C_{fuchs}$ may be very different from that of
$C$. For instance, for a BPS  with discrete holonomy in $PSL(2,\mathbb
R)$, a positive
component $C$ may be diffeomorphic to a pair of pants but $C_{fuchs}$
to a disc (case of a branched covering over $\mathbb C \mathbb P ^1$)
or to a punctured torus.

\begin{example}
Consider a
complete hyperbolic metric on a punctured torus $\Sigma$, with a cusp
of infinite area. Let $\gamma\subset \Sigma$ be a properly embedded
semi-infinite geodesic. Let $\alpha$ and $\beta$ be some generators of
the fundamental group of $\Sigma$, and let $$\pi_{\alpha} :
\Sigma_{\alpha}= \widetilde{\Sigma}/\alpha \rightarrow \Sigma, \ \ \ \
\ \pi_{\beta}: \Sigma_{\beta}=\widetilde{\Sigma} /\beta \rightarrow
\Sigma$$ be the intermediate coverings defined by $\alpha$ and
$\beta$; both are isometric to loxodromic annuli. Let
$\gamma_{\alpha}$ and $\gamma_{\beta}$ be some lifts of $\gamma$ by
$\pi_{\alpha}$ and $\pi_{\beta}$. These semi-infinite paths are
properly embedded as well. Let us cut $\Sigma_{\alpha}$ and
$\Sigma_{\beta}$ and paste these annuli along these cuts. One obtains a
pair of pants $\Pi$ together with a hyperbolic metric with one conical
point of angle $4\pi$. Moreover, the maps $\pi_{\alpha}$ and
$\pi_{\beta}$ glue together to produce a map $D : \Pi \rightarrow
\Sigma$, which is a local isometry and a $\pi_1$-isomorphism (see \cite[Lemma 4,
p. 658]{Tan} for more details).
\end{example}

However, such examples are incompatible with our Fuchsian
assumption:

\begin{lemma} [Identification between $C$ and $C_{fuchs}$]
\label{l:identification_with_Cfuchs} Let $S$ be a closed surface
equipped with a BPS with Fuchsian holonomy. If $C$ is an
incompressible component of $S^\pm$, then
there is an orientation preserving diffeomorphism $\Phi: C_{fuchs} \rightarrow C$
such that the map induced by $D_C \circ \Phi$ on the set of free-homotopy
classes of closed loops is the identity.
\end{lemma}

\begin{proof} Since $C$ is incompressible, any connected component
of $\pi^{-1} (C)$ is simply connected. If follows that $\widehat C$ is the
universal covering of $C$ and the Galois group
$\pi_1(C)$ is indeed isomorphic to the
fundamental group of $C$, via an isomorphism that is well-defined up
to conjugation.

It is classical to see that there exists an orientation preserving
diffeomorphism $\widetilde F:\widetilde S\to\widehat C$ that is
$\pi_1(C)$-equivariant and $\pi_1(C)$-equivariantly homotopic to the identity.
 On the other hand, since  $\rho$ is Fuchsian and $S$ is closed,
there is a $\rho$-equivariant orientation preserving diffeomorphism
$\widetilde G:\widetilde S\to\mathbb H^2$. The map $\widetilde
\Phi=\widetilde F\circ\widetilde G^{-1}$ descends to a diffeomorphism
$\Phi:C\f\to C$. Since $D$ is $\rho$-equivariant, the
map $D\circ \widetilde \Phi$
has the property that $$D(\widetilde
\Phi(h\cdot x))=h\cdot D(\widetilde\Phi(x))$$
for all $h\in\rho(\pi_1(C))$. Since $\rho(\pi_1(C))$ is the Galois
group of the universal covering $\mathbb H^2\to C\f$, it follows that
$D\circ\Phi$ fixes free-homotopy classes of loops.

\end{proof}

We identify $C$ with $C_{fuchs}$ using the diffeomorphism $\Phi$ of
Lem\-ma~\ref{l:identification_with_Cfuchs} so that now it makes sense
to say that $D_C$ fixes free homotopy classes of loops. We still use
the notation $C$ and $C\f$
to mean that the structure of $C$ is the branched one
while that of $C\f$ is the hyperbolic unbranched one.

Since the holonomy of
$l$ is loxodromic, we consider the geodesic representatives $\gamma$ in
$C$ and $\gamma_{fuchs}$ in $C_{fuchs}$ in the respective homotopy classes. Note that
$D_C(\gamma)$ is in general different from $\gamma_{fuchs}$ (they
only belong to the same homotopy-class),
since $D_C$ is not a global
isometry. As above, we denote by $E_l$ the end of $C$ corresponding to $l$.
If $C\f$ is an annulus let $E\f$ be the end which is on the same
side of $\gamma\f$ as $E_l$ of $\gamma$. Otherwise
$E\f$ is just the end of $C\f$ corresponding to $l$. In both cases
$E\f$ is an annulus  with
geodesic boundary $\gamma\f$ and a complete hyperbolic metric with
loxodromic holonomy.

In order to prove Theorem~\ref{t:debubbling},
we begin by moving the branch point $p$ in $C$ so that its image $q=
D_C(p)$ in $C_{fuchs}$ belongs to $\gamma_{fuchs}$. To this end,
let  $\delta$ be an embedded
geodesic segment starting at $q$ and ending at a point of
$\gamma_{fuchs}$, not containing the image of branch points other
than $q$. By Lemma~\ref{l:completeness}, the geodesic $\delta$
can be lifted to a pair of twin geodesics $(\delta_1, \delta_2)$ in
$C$.  Thanks to Lemma~\ref{l:embedded_twins}, these twins are in fact
embedded. By cutting and pasting along these twins, we obtain a new
BPS on $S$ such that $D_C(p)$ belongs to $\gamma_{fuchs}$.

Under this condition, the geodesic $\gamma$ passes through $p$.
Indeed, otherwise $\gamma$ would be a smooth geodesic and since $l$ has
index $1$,  $p$ would  belong to the end $E_l$. In this case we could
chose $\sigma$ the orthogonal segment from $p$ to $\gamma$. Since
$D_C$ fixes free homotopy classes of curves, we would get
 $D_C(\gamma)=\gamma\f$ as oriented loops, and $D_C(\sigma)$
should be a smooth geodesic segment starting from $q=D_C(p)\in\gamma\f$, and
orthogonally ending to $\gamma\f$ on the side of $E\f$. But a geodesic segment
starting at a point of $\gamma_{fuchs}$ and entering $E\f$ never  comes
back to $\gamma_{fuchs}$ again (because $E\f$ is a genuine hyperbolic
surface) that would be a contradiction.
Hence the only possibility is that $p\in\gamma$.

The total angle at $p$ is
$4\pi$ and the exterior angles of $\gamma$ at $p$ are $\geq \pi$ and
disjoint (by Lemma~\ref{halftub}). Hence $\gamma$ cannot pass through $p$
more than $4$ times. Examples of peripheral geodesics passing four
times through a conical point of angle $4\pi$ exist in
general.

\begin{example}
Consider an annulus $A$ equipped with a complete
hyperbolic metric, with a loxodromic end, and with a geodesic boundary
component $\gamma$. Cut $\gamma$ in four segments of equal length
$I_1,I_2, I_3,I_4$ arranged in cyclic order. Glue $I_1$ with $I_3$ and
$I_2$ with $I_4$ by reversing the orientation. We obtain a punctured
torus $P$ in which the extremities of the intervals $I_k$ are glued to
a same point $p$ of total angle $4\pi$, and in which the peripheral
geodesic is the image of $\gamma$ by the quotient map $A\rightarrow
P$. This peripheral geodesic passes through $p$ four times.
\end{example}
However
these cases do not arise under our assumptions.

\begin{lemma}\label{l:twice} Suppose that $D_C(p)\in\gamma\f$.
Then, the peripheral geodesic $\gamma$  passes through $p$ exactly once
if and only if it forms a pair of  angles $(\pi,3\pi)$ at $p$,
and in this case $D_C$ embeds $\gamma$ to $\gamma\f$.
Moreover, if $\gamma$ forms no angle $\pi$ at $p$, then
it passes through $p$ exactly twice.
\end{lemma}

\begin{proof}  Let $B$ be an embedded metric ball around
$q=D_C(p)$. Let $E$ denote the sector $E\subset B$ on the side of $\gamma_{fuchs}$ belonging to $E_{fuchs}$ and $F$ the sector on the other side. Since $p$ is a branch point of
total angle $4\pi$, there exists a neighborhood $B'$ of $p$ such that
$D_C: B' \to B$ is a double covering, branched at $p$. Let $E_i$ and
$F_i$, $i=1,2$ be the preimages of $E$ and $F$ in $B'$; these are four
sectors of angle $\pi$ at $p$ arranged in a cyclic order
$E_1,F_1,E_2,F_2$. See Figure~\ref{f:BEF}.

\begin{figure}[httb] \centering

\centering
\includegraphics[width=4in]{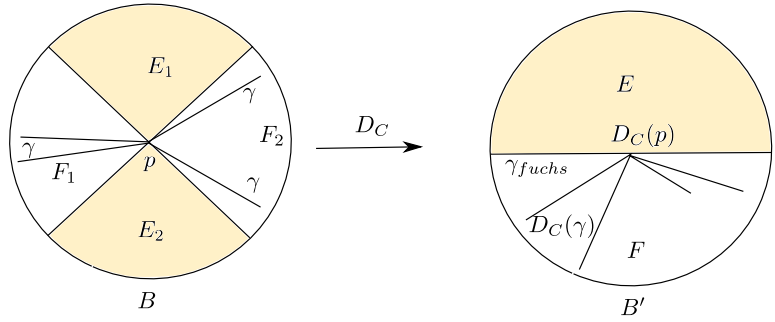}
  \caption{The map $D_C$ restricted to $B'$}
  \label{f:BEF}
\end{figure}

We call an {\bf edge  of $\gamma$} any embedded sub-loop of $\gamma$,
that is to say, a segment of $\gamma$ starting
and ending at $p$ but not passing through $p$ appart at its
extremities. Any edge is smooth outside $p$ so its image in
$C_{fuchs}$ is a geodesic starting and ending at $q$.  Note that
any semi-geodesic in $C_{fuchs}$ starting at $\gamma_{fuchs}$ and
entering the interior of $E_{fuchs}$ tends to infinity in $C_{fuchs}$
without coming back to $\gamma_{fuchs}$. It follows that $D_C(\gamma)
\cap B \subset \overline{C\f\setminus E}$. In particular, $\gamma\cap
B' \subset F_1\cup F_2$ (possibly the intersection is empty if
$\gamma$ does not contain $p$).

Let us prove that if an exterior angle of $\gamma$ at $p$ is equal to
$\pi$ or $3\pi$, then $\gamma$ passes through $p$ once. Suppose by
contradiction that $\gamma$ has $s\geq 2$ edges and that the exterior
angle between two consecutive edges is $\pi$ or $3\pi$. Then up to
cyclic permutation of the edges, one can write $\gamma = \gamma_1
\star \ldots \star \gamma_s$, where the $\gamma_i$'s are the edges of
$\gamma$, for $i=1,\ldots ,s$, and the exterior angle between
$\gamma_1$ and $\gamma_2$ is $\pi$ or $3\pi$. Set
$I= \gamma_1\star \gamma_2$.
Then, the image of $I$ by $D_C$ is a union of
geodesic lines contained
in $\overline{C\f\setminus E}$, passing through $q$ and forming
an angle of $\pi$ at $q$.  Hence $D_C (I)\subset\gamma_{fuchs}$
at the neighborhood of $q$. Since the
image of the $\gamma_i$'s by $D_C$ are geodesic segments, this implies
that the image of $\gamma_1$ and $\gamma_2$ by $D_C$ are some powers
of $\gamma_{fuchs} $: $D_C (\gamma_1) = \gamma_{fuchs}^{n_1}$,
$D_C(\gamma_2) = \gamma_{fuchs}^{n_2}$ for some integers $n_1,n_2$.
Because $D_C$ induces the identity on the set
of free homotopy classes of loops, the
loops $\gamma_1$ and $\gamma_2$ are homotopic to the loops
$l^{n_1}$ and $l^{n_2}$. As $S$ is oriented, no simple loop can be a
proper power, and since both $\gamma_1$ and
$\gamma_2$ are embedded loops, we get $n_1=n_2=1$. But then $\gamma$
is no longer the curve minimizing the length in the free homotopy
class of $l$, a contradiction.

On the other hand, if $\gamma$ passes
only once through $p$, then it is a simple loop which is
mapped to a geodesic loop through $q$ and in the same homotopy
class as $\gamma\f$. This shows that $D_C(\gamma)=\gamma\f$ and that
$(D_C)|_{\gamma}$ is in fact an embedding.
It follows that the angles that $\gamma$ forms
at $p$ are $\pi$ on one side and $3\pi$ on the other.

It remains to prove the last claim.
Since the exterior angles are always in $[\pi, 3\pi]$, we proved that
if $\gamma$ passes through $p$ at least twice, then the exterior angles of
$\gamma$ at $p$ belong to $(\pi, 3\pi)$. Hence, because $\gamma \cap
B' \subset F_1\cup F_2$, any exterior angle of $\gamma$ must cover one
of the $E_i$. Because these angles at $p$ are disjoint, their number
cannot exceed $2$.
\end{proof}

Let $\widehat{l}$ be a lift of $l$ in the boundary of
$\widehat{C}$ and $\alpha\in \pi_1(C)$ be the generator of the group
stabilizing $\widehat{l}$ in $\pi_1(C)$ that acts as a positive
translation on $\widehat{l}$ for the orientation of $\widehat{l}$
given by the $\mathbb R \mathbb P ^1$-structure of $\widehat{l}$.
Since $l$ is loxodromic, we can use the upper half-plane model for
$\mathbb H^2$, and choose the developing map in such a
way that for some real number $\lambda >1$
\[ D \circ \alpha = \lambda D \] on $\widehat{l}$. That is to say,
$\rho (\alpha)(z)=\lambda z$.

The map $D$ induces a decomposition $\widehat{l} = D^{-1} (0) \cup
D^{-1} (\mathbb R^{>0}) \cup D^{-1}(\infty) \cup D^{-1} (\mathbb
R^{<0})$ which is $\alpha$-invariant.  Since $l$ is of index
$1$, the parts $\pi(D^{-1}(0))$ and $\pi(D^{-1}(\infty))$
consist of single points.
We denote by $l = \{0\} \cup l^+ \cup \{\infty\} \cup l^-$ the decomposition
induced on $l$ (see Figure~\ref{fig:lu1u2}, left side). Observe
that the map $D$ induces a projective diffeomorphism between $l^+$ and
$\mathbb R^{>0}$ which is well defined up to multiplication by a power
of $\lambda$. The same for $l^-$. Hence the multiplication by $\lambda$ is
well defined on $l^+\cup l^-$.

Note that because of the choice of the upper half-plane model,
we have $E\f$ is half of $\lambda\backslash\mathbb
H^2$, and $\gamma\f$
lifts to the imaginary axis of the upper half-plane.

\begin{prop}[Existence of a half-bubble]
\label{p:existenceofhalf-bubble}
If $D_C(p)$ belongs to
$\gamma_{fuchs}$, and $\gamma$ passes through $p$ exactly once, then
for each $u\in l^+$ there exists a half bubble
in the direction of $l$ whose endpoints in $l$ are  $(u,u/\lambda)$.
\end{prop}
\setlength{\unitlength}{1ex}
\begin{figure}[htbp]
  \centering
  \begin{picture}(66,14)

    \qbezier(0,7)(0,14)(7,14)
    \qbezier(0,7)(0,0)(7,0)
    \qbezier(14,7)(14,14)(7,14)
    \qbezier(14,7)(14,0)(7,0)
    \put(2,12){\makebox(-1,0){$\bullet$}}
    \put(1,13){\makebox(-2,1){$u$}}
    \put(12,12){\makebox(1,0){$\bullet$}}
    \put(13,13){\makebox(2,1){$u/\lambda$}}
    \put(2,2){\makebox(-1,0){$\bullet$}}
    \put(1,1){\makebox(-1,-1){$\infty$}}
    \put(12,2){\makebox(1,0){$\bullet$}}
    \put(13,1){\makebox(1,-1){$0$}}


    \put(20,2){\vector(1,0){46}}
    \multiput(25,2)(5,0){8}{\makebox(0,0){$\bullet$}}
    \multiput(25,0)(20,0){1}{\makebox(-2,-1){\footnotesize $\alpha^{-1}(\widehat{u})$}}
          \put(45,0){\makebox(-1,-1){\footnotesize$\widehat{u}$}}
    \multiput(30,0)(20,0){1}{\makebox(0,-1){\footnotesize$\widehat\infty$}}
          \put(50,0){\makebox(0,-1){\footnotesize$\alpha(\widehat\infty)$}}
    \multiput(35,0)(20,0){1}{\makebox(0,-1){\footnotesize$\widehat 0$}}
          \put(55,0){\makebox(0,-1){\footnotesize$\alpha(\widehat 0)$}}
    \multiput(40,0)(20,0){1}{\makebox(0,-1){\footnotesize$\widehat u/\lambda$}}
          \put(60,0){\makebox(1,-1){\footnotesize$\alpha(\widehat u/\lambda)$}}
    \put(25.5,3){\line(2,3){5}}
    \put(39,3){\line(-2,3){5}}
    \put(32,12){\makebox(0,0){Half-bubble}}
    \put(40,2){\line(3,2){15}}
    \put(45,2){\line(1,1){10}}
    \put(55,14){\makebox(14,1){\parbox{25ex}{Fundamental domain for the
        multiplication by $\lambda$}}}
  \end{picture}
  \caption{The curve $l$ oriented counterclockwise, and its lift
    $\widehat l$}
  \label{fig:lu1u2}
\end{figure}
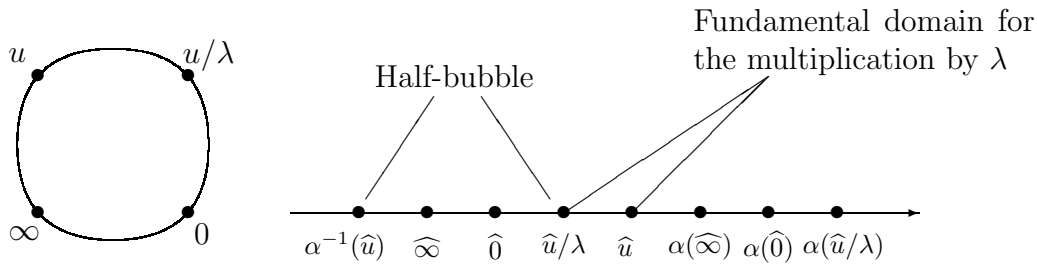

\begin{proof}
By Lemma~\ref{l:twice} $\gamma$ forms at $p$ angles $\pi$ and $3\pi$.
Since $l$ has
index $1$ it follows that the angle $G$ on the side of the end is the
one of $3\pi$. Otherwise we would be able to move the singularity out
of the end, which is in contradiction with the index hypothesis and
corollary \ref{cor:eulerannulus}. Hence, the restriction of $D$ to $G$
covers twice the
exterior angle of $\gamma_{fuchs}$ in $C_{fuchs}$ (which is the sector $E$ of
Figure~\ref{f:BEF}).

Let $\widehat p$ be a lift of $p$ in $\widehat C$.
In the upper half-plane model of $\mathbb H^2$ that we are using,
the point $D(\widehat p)$ belongs to the positive imaginary
axis. Since $l$ has index $1$,
the segment from $0$ to $D(\widehat p)$ has two pre-images joining
$\widehat p$ to two consecutive elements of $D^{-1}(0)$ that bound a
segment $J$ in $\widehat l$ which is a fundamental domain for the
action of $\alpha$ (see Figure~\ref{fig:lu1u2}).

Let $u$ be a point of $l^+$ and $\widehat u$ be its lifts in
$J$. The point $D(\widehat u)$ belongs to the positive real line by
definition of $l^+$. Let $\widetilde\tau$
be a geodesic in $\mathbb H^2$
from $D(\widehat p)$ to $D(\widehat u)$, and let $\tau$ be its projection
to $E\f\subset\lambda\backslash\mathbb H^2$.

By construction, $\tau$ is a geodesic in $E\f$ starting from
$q=D_C(p)$. As the angle of $\gamma$ at $p$ on the side of the end $E_l$ is
$3\pi$, $\tau$ lifts to two twin geodesic paths $(\tau_1,\tau_2)$ in
$E_l$, with the convention that the angle $2\pi$ is the one from
$\tau_1$ to $\tau_2$ in the positive direction
given by the orientation of $S$.

By Lemma~\ref{l:embedded_twins} (applied with $U=\overline{E_l}$,
$\Sigma=\overline{E\f}\subset \lambda\backslash \mathbb H^2$,
and $f=D_C$) the pair $(\tau_1, \tau_2)$ is embedded in
$E_l$, with distinct limits at infinity $u_1,u_2$.
The pair $(\tau_1,\tau_2)$ lifts to a pair of twin
geodesic $(\widehat\tau_1,\widehat\tau_2)$ starting from $\widehat
p$.
The arc
$\tau_1\star\tau_2$ cuts $E_l$ --- which is a topological annulus ---
 in two parts, one of which is a disc.
Since the angle from $\tau_1$ and $\tau_2$ is $2\pi$ and it is contained in
$E_l$, the disc is the part going from
$\tau_1$ to $\tau_2$ in the positive sense.

The disc between $\tau_1$ and $\tau_2$ lifts to a disc in
$\widehat C$, bounded by $\widehat\tau_1\star\widehat\tau_2$ and a
segment of $\widehat l$.
Such a disc  is homeomorphically mapped to $\mathbb H^2$ minus $\widetilde
\tau$ by construction. Therefore, $(\tau_1,\tau_2)$ is a half-bubble.

If we show that $(u_1,u_2)=(u,u/\lambda)$ we are done.
 By construction we have
$D(\widehat\tau_1)=D(\widehat\tau_2)=\widetilde\tau$. It follows that
the end point $\widehat u_1$ of $\widehat\tau_1$ in $\widehat l$
belongs to $D^{-1}(D(\widehat u))$, and the same for $\widehat u_2$.
Since the angle between
$\tau_1$ and $\tau_2$ is $2\pi$ in the positive sense, it follows that
$\widehat u_1\in J$ and $\widehat u_2\in\alpha(J)$. In particular,
$\widehat u_1=\widehat u$, hence $u_1=u$ as we needed.
Also, the point $\widehat u_2$ is the successor
of $\widehat u_1$ w.r.t. the ordering of $D^{-1}(D(\widehat u))$
induced by $\widehat l$.
From the fact that $l$ has
index $1$ we get that when restricted to $[\widehat
u,\alpha(\widehat u)]$ (which is a fundamental domain for the
action of $\alpha$ on $\widehat l$) the map $D$ starts from $D(\widehat
u)$, makes a whole turn around $\mathbb{RP}^1$ and finally
covers the fundamental domain $[D(\widehat u),\lambda D(\widehat
u)]$. (See Figure~\ref{fig:lu1u2}.) It follows that, according to our
notation for the multiplication by $\lambda$ on $\widehat l$, $\widehat
u_2=\lambda^{-1}\alpha(\widehat u_1)$. So $u_2=u_1/\lambda=u/\lambda$.
 This ends the proof of
Proposition~\ref{p:existenceofhalf-bubble}.
\end{proof}

\begin{lemma}\label{l:transformtocase1}
By moving $p$ in $C$ we can reduce to the case
that $\gamma$ passes through $p$
precisely once, and $D_C(p)\in\gamma\f$.
\end{lemma}

\begin{proof}
Suppose that $\gamma$ passes twice through $p$. Let $\widehat\gamma$ be
the lift of $\gamma$ to $\widehat C$ corresponding to $\widehat l$ (so
that $\alpha$ acts by translations on $\gamma$).
We look at the developed image $D (\widehat\gamma)$.
It draws a ``zig-zag'' line in $\mathbb H^2$.
More precisely, let $\widehat p$ be a lift of $p$ so that $\widetilde
q=D(\widehat p)$ belongs to the imaginary axis in our upper half-plane
model. Since $\gamma$
consists of two segments, $D (\widehat
\gamma)$ starts from a lift $\widetilde q$ in the imaginary axis, goes to
some point $\widetilde{r}$, comes back to $\lambda \widetilde q$
and then repeats the same path, multiplied by
$\lambda,\lambda^2,...$. (See Figure~\ref{fig:zigzag}, left side.)

\setlength{\unitlength}{1ex}
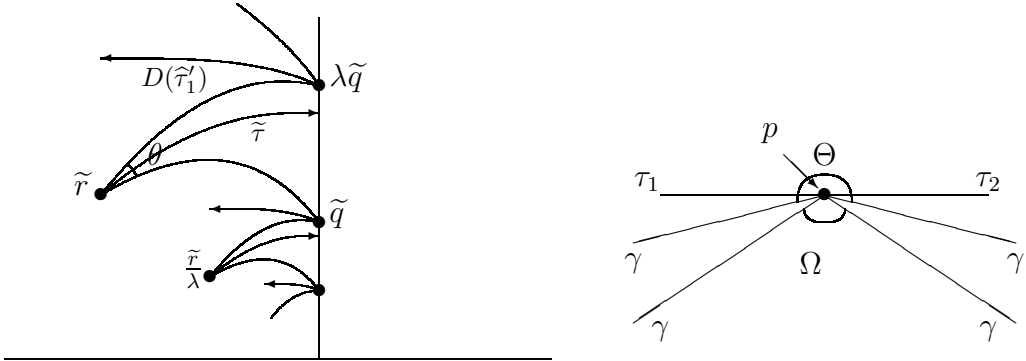
\begin{figure}[htbp]
  \centering
  \begin{picture}(76,25)
  \put(-10,0){
    \put(10,0){\line(1,0){40}}
    \put(33,0){\line(0,1){25}}

    \put(33,5){\makebox(0,0){$\bullet$}}
    \put(33,10){\makebox(0,0){$\bullet$}$\ \widetilde q$}
    \put(33,20){\makebox(0,0){$\bullet$}$\ \lambda\widetilde q$}

    \put(25,6){\makebox(0,0){$\bullet$}$\!\!\!\!\! \frac{\widetilde r}{\lambda}$}
    \put(17,12){\makebox(0,0){$\bullet$}$\!\!\!\!\! \widetilde r$}
    \qbezier(19,14.2)(19.4,13.9)(19.6,13.4)
    \put(21,15){\makebox(0,0){$\theta$}}

    \qbezier(33,5)(30,9)(25,6)
    \qbezier(25,6)(29,11)(33,10)

    \qbezier(33,10)(27,18)(17,12)
    \qbezier(17,12)(25,22)(33,20)

    \qbezier(29.5,3)(31,5)(33,5)
    \qbezier(33,20)(31,23)(27,26)

    \qbezier(25,6)(28.5,9)(32,9)
    \put(32,9){\vector(1,0){1}}

    \qbezier(17,12)(24,18)(31,18)
    \put(31,18){\vector(1,0){2}}

    \put(28,16){\footnotesize $\widetilde\tau$}
    \put(20,20){\footnotesize $D(\widehat\tau_1')$}
    \qbezier(33,5)(32,5.5)(30,5.5)
    \put(30,5.5){\vector(-1,0){1}}
    \qbezier(33,10)(31,11)(27,11)
    \put(27,11){\vector(-1,0){2}}
    \qbezier(33,20)(29,22)(21,22)
    \put(21,22){\vector(-1,0){4}}
  }

  \put(60,12){
    \put(0,0){\makebox(0,0){$\bullet$}}
    \put(0,0){\line(4,-1){14}}
    \put(0,0){\line(-4,-1){14}}
    \put(0,0){\line(3,-2){14}}
    \put(0,0){\line(-3,-2){14}}
    \put(-12,0){\line(1,0){24}}

    \qbezier(-2,-0.5)(-2,1.5)(0,1.5)
    \qbezier(2,-0.5)(2,1.5)(0,1.5)
    \put(0,3){\makebox(0,0){$\Theta$}}

    \qbezier(-1.5,-1)(-1.5,-2)(0,-2)
    \qbezier(1.5,-1)(1.5,-2)(0,-2)
    \put(-1,-5){\makebox(0,0){$\Omega$}}

    \put(-14,-5){\makebox(0,0){$\gamma$}}
    \put(-12,-10){\makebox(0,0){$\gamma$}}
    \put(14,-5){\makebox(0,0){$\gamma$}}
    \put(12,-10){\makebox(0,0){$\gamma$}}

    \put(-5,3){\makebox(2,3){$p$}}
    \put(-3,3){\vector(1,-1){2.5}}

    \put(-13,1){\makebox(0,0){$\tau_1$}}
    \put(12,1){\makebox(0,0){$\tau_2$}}
}
  \end{picture}
  \caption{The ``zig-zag'' line $D(\widehat \gamma)$ and the local
    picture near $p$}
  \label{fig:zigzag}
\end{figure}

If we look at $D_C(\gamma)$, we see that the two segments emanating
from $D_C(p)$ does not enter the end $E\f$, by the usual argument that
a geodesic cannot enter and exit $E\f$. Translated to the universal
covering, this says that the point $\widetilde{r}$ belongs to the
interior of the left
quadrant, i.e. its real part is negative.
It follows that the angles that we see
on the right-side of $D(\widehat\gamma)$ are,
alternatively bigger than $\pi$ (at the points $\widetilde q, \lambda
\widetilde q, \ldots$) and smaller than $\pi$ (at the points $\widetilde{r},
\lambda \widetilde{r}, \ldots$).

Let $\widehat s$ be the point of $\widehat{\gamma}$
 such that $D(\widehat{s}) = \widetilde{r}$.  Denote
by $\Theta$ the exterior angle
 seen on the right of $\widehat{\gamma}$ at
$\widehat s$, and set $\theta= \measuredangle \widetilde{q} \widetilde{r}
(\lambda\widetilde{q})<\pi$.
As $D$ is a double branched covering at $\widehat s$,
either $\Theta=\theta$ or $\Theta=2\pi+\theta$, and since $\Theta$
must exceed $\pi$ we get $\Theta=2\pi+\theta$.
Moreover,
since $\gamma$ passes
twice through $p$, locally at $p$ we see
the angular sector  $\Theta$, plus the sector $\Omega$ which
corresponds to the exterior angle at $\widetilde
q$. Figure~\ref{fig:zigzag}, right side, shows a local picture in
which the angles we see on the paper are half of those around
$p$ (so segments forming an angle $\pi$ represent twin
segments). Since $\Theta$ and $\Omega$ are disjoint by
Lemma~\ref{halftub}, there exist a pair of twin geodesic
segments $\tau_1$ and
$\tau_2$ starting from $p$ and pointing into
$\Theta$, and both lying in the complement of $\Omega$ (for instance
the paths obtained as preimages of the bisector of $\theta$).
Such segments are in $E_l$ because the angular sector $\Theta$ is that on the
side of $E_l$.

For $i=1,2$ let $\widehat\tau_i$ be the lift of $\tau_i$
starting at $\widehat s$ and let
$\widetilde\tau=D(\widehat\tau_i)$. We have that
$\widetilde\tau$ points in the angular sector $\theta$ because
$\tau_i$'s are in $\Theta$. Similarly, if $\widehat\tau_i'$
is the lift of $\tau_i$ that starts at $\widehat p$, then
$D(\widehat\tau_1')=D(\widehat\tau_2')$
points left as in Figure~\ref{fig:zigzag}, left side because
$\tau_i$'s are in the complement of $\Omega$.
Therefore we can  geodesically
extend $\widetilde\tau$ until the imaginary axis,
without intersecting $D(\widehat \gamma)$ nor $D(\widehat\tau'_i)$.

Since $\tau_1$ and $\tau_2$ both enter $E_l$, they never exit $E_l$.
By construction
$\lambda\widetilde\tau \cap\widetilde \tau=\emptyset$. So, it projects to
a segment $\tau$ in $\lambda\backslash \mathbb H^2$.
Now, we consider the map $D_E:E_l\to\lambda\backslash \mathbb H^2$
induced by the developing map.
Lemma~\ref{l:embedded_twins}, applied with $U$ the open end $E_l$,
$\Sigma=\lambda\backslash \mathbb H^2$, and $f=D_E$ tells us that in
fact $(\tau_1,\tau_2)$ are embedded twin paths.

We claim that after
the moving along such twin paths, we reduce to the case that $\gamma$
passes only once through $p$, and  $D_C(p)\in\gamma\f$.
For that, we have to focus on the ``zig-zag'' line. Since any vertex of the ``zig-zag'' is the
image of a lift of $p$, we see at any vertex
the image of $\widetilde \tau$ via a certain element of
$\rho(\pi_1(C))$. Namely, at $\lambda^n\widetilde r$ we see
$\lambda^n\widetilde\tau$ and at $\lambda^n\widetilde q$ we see
$\lambda^nD(\widehat\tau_1)$.

Now, we parameterize $\tau$ with times $t\in[0,1]$ and perform the cut
and paste procedure continuously on $t$. With the same notation we are
using, we
put an index $t$ at the bottom of every object to mean its evolution at
time $t$. Namely, $\gamma_t$ is the geodesic representative of
$l$, and $\widehat \gamma_t, \widetilde r_t, \widetilde
q_t$ are the corresponding of $\widehat \gamma,
\widetilde r, \widetilde q$ and so on.
By the uniqueness of the geodesic representative of $l$ we get that
$\gamma_t$ changes continuously on $t$. It follows that, as far as the
angles that $\gamma_t$ forms at $p_t$ are bigger than $\pi$,
the developed image of $\widehat \gamma_t$ is the ``zig-zag'' line trough
$\widetilde q_t,\widetilde r_t,\lambda\widetilde q_t,\lambda\widetilde
r_t...$, which are nothing but the point $\widetilde\tau(t)$ and its images.
Since $\widetilde r_t$ moves right and $\widetilde
q_t$ moves left, eventually $\widetilde q$ and its images detach
from $D(\widehat\gamma_t)$ (in fact one can easily check with some
elementary hyperbolic trigonometry, that $\Theta_t$ increases on $t$
while $\Omega_t$ decreases, and there is a time $t_0$ where $\Omega_t$
becomes $\pi$). From this time on, $\gamma_t$ passes
through $p_t$ only once, and for $t=1$ the developed image of $\widehat
\gamma_1$ is the imaginary axis, so $D_C(p_1)\in\gamma\f$.
\end{proof}

We remark that putting Proposition~\ref{p:existenceofhalf-bubble} and Lemma~\ref{l:transformtocase1} together, we can directly find a (non-geodesic) half-bubble in the case where the peripheral geodesic passes twice through the branch point. This is done by considering a continuation of the path $\tau$ defined in the proof of Lemma~\ref{l:transformtocase1} to a point lying in $\mathbb R^+$.

\proof[Proof of Theorem~\ref{t:debubbling}]
We consider a couple of the form
$(\lambda u , u)$ in $l^+$. Combining Lemma~\ref{l:transformtocase1}
and Proposition~\ref{p:existenceofhalf-bubble} we show
that after moving the branch points in $C^+$ and $C^-$ respectively,
one can find half-bubbles $(\eta_1 ^+, \eta_2^+)$ in $C^+$ and
$(\eta_1^- , \eta_2^-)$ in $C^-$ whose extremities are $(u_1^+, u_2^+
) = (\lambda u, u)$ and $(u_1^-, u_2^-) = (\lambda u, u )$. Since
geodesics in the components meet orthogonally the real curve
$S_{\mathbb R}$, the union of $(\eta_1 ^+, \eta_2^+)$  and
$(\eta_1^- , \eta_2^-)$ is a pair of smooth twin paths, and
the two regions between the
half-bubbles is in fact projectively equivalent to $\mathbb{CP}^1$
minus a closed segment, hence a bubble.\qed

\section{Degeneration dichotomy}\label{s:deg}

In this section, given a BPS with Fuchsian holonomy, we try to move all the branch points in a given  component of $S^{\pm}$ to a unique branch point of high branching order. By doing this, the surface may degenerate to a nodal curve. In such a situation, the degeneration allows us to find a bubble just before degenerating. If the structure does not degenerate, we succeed in our task.

\begin{prop}\label{l:degeneracion} Let $S$ be a closed surface
  equipped with a BPS with Fuchsian holonomy.
Let $C$ be a component of $S^{\pm}$,
with $R$ branch points of angle $4\pi$. Then, either we can
move branch points so that we find a bubble,
or we can move all the branch points of $C$ to a
single branch point of angle $2\pi (R+1)$.
\end{prop}

\begin{proof} The proof goes by induction, joining one by one the
branch points, the inductive step being summarized as ``either we can
join one more branch point or we find a bubble''.

To begin with, we choose a branch point $p_0$ and a positive constant
$\varepsilon$ smaller than the injectivity radius of
$C\f$.

\begin{lemma}
  We can move branch points so that after the moving, for every branch point~$p$
  $$d(p,p_0)<\varepsilon.$$
\end{lemma}
\begin{proof}
Let $p$ be one of the branch points which is closest to $p_0$ among
those with $d(p,p_0)\geq\varepsilon$ (if any). Let $\sigma$ be a
geodesic joining $p$ to $p_0$, and of length $d(p,p_0)$.
If $d(p,p_0)=\varepsilon$ we just move a little $p$ so that
$d(p,p_0)<\varepsilon$. Otherwise, let
$\delta=d(p,p_0)-\varepsilon>0$. The
initial segment of $\sigma$ of length $\delta$ does not
contain any other branch point because of the criteria for choosing $p$.

By Lemma~\ref{wecanmove} we
can move $p$ along an initial segment of $\sigma$ of length at least
$\min\{\varepsilon,\delta\}$. A recursive argument proves
the claim.
\end{proof}

We denote by $D_B$ a
developing map from $B(p_0,\varepsilon)$
to a disc of radius $\varepsilon$ in $\mathbb H^2$. Up to little
movements, we can assume that
distinct branch points have distinct images and
distances from $p_0$. At the beginning of the our strategy all conical
angles are $4\pi$, but when we join a branch point to $p_0$ its angle
increases. So we consider the general situation where the angle at
$p_0$ is $2k\pi$.

Let $p$ be the branch point (other than $p_0$) closest to $p_0$
and  $\sigma$ be
a path realizing the distance from $p_0$ to $p$. The dichotomy is now
the following:
\begin{enumerate}
\item There are at least two shortest paths $\sigma_0$ and $\sigma_1$
from $p_0$ to $p$ (and in this case our claim is that up to moving
branch points one founds a bubble).
\item $\sigma$ is the unique path from $p_0$ to $p$ minimizing
the distance (and
we claim that in this case one can join $p$ to $p_0$).
\end{enumerate}

Suppose we are in the first case.
Paths $\sigma_0$ and $\sigma_1$ are both embedded otherwise there
would exist a short-cut from $p$ to $p_0$. For the same reason, if
$\sigma_0(t)=\sigma_1(s)$ for some $s,t>0$, then $s=t$.
If $\sigma_0(t)=\sigma_1(t)$, for some $t>0$ then
$\sigma_0(t)=\sigma_1(t)$ is not a branch point because $p$ is the
closest to $p_0$, and also in this case we easily find shortcuts, as
one of the angles between $\sigma_0$ and $\sigma_1$ is smaller than $\pi$.
 Therefore $\sigma_0$ and $\sigma_1$ do not cross each other.

The images $D_B(\sigma_0)$ and $D_B(\sigma_1)$ are therefore smooth
geodesics in a disc, thus they are segments. In particular, they both
are the only segment between $D_B(p)$ and
$D_B(p_0)$. Therefore $\sigma_0$ and $\sigma_1$ are twin.
It follows that the loop $\sigma_0\sigma_1^{-1}$ has trivial holonomy,
hence it is homotopically trivial in $S$ because the holonomy is
faithful. Thus it bounds a disc $Q$ in
$S$. Note that by moving $p$ along $\sigma_0\cup\sigma_1$ we would
disconnect $S$.

The angle that $p_0$ forms on the side of $Q$ is $2h\pi$ with $h<k$.
If $h>1$, we move a little $p_0$ using twin paths $\tau_0,\tau_1$
contained in the interior of $Q$
so that, if $\sigma_0,\tau_0,\tau_1,\sigma_1$ are in cyclic order in
$Q$, then the angles $\alpha_a=\angle(\sigma_0,\tau_0)$ and
$\alpha_1=\angle(\tau_1,\sigma_1)$
satisfy $\alpha_0+\alpha_1=2\pi$. This separates
$p_0$ in $2$ branch points. One of them, which we still denote $p_0$,
is the end-point of $\sigma_0$
and $\sigma_1$,  and has an angle $2\pi$ on the side of $Q$,
the other has a total angle $2(h-1)\pi$ and belongs to the interior
of $Q$. We repeat the same construction outside $Q$ so to obtain that
$p_0$ has a total angle $4\pi$, with $2\pi$ in $Q$, and $2\pi$ outside $Q$.

We cut $S$ along $\sigma_0\cup\sigma_1$ and glue
the boundary of the disc $Q$ along $\sigma_0\cup\sigma_1$ by gluing
$\sigma_1(t)$ to $\sigma_0(t)$.  The result is a sphere $P$ with a
branched projective structure, a marked segment $\sigma$ with marked
end-points $p_0$ and $p$, both being regular. Similarly, on the other
piece of $S$ we get a surface $\bar S$ with a market segment $\sigma$
with regular end-points $p_0$ and $p$.

Now, via movings, we split any branch point present in $Q$ in some number of
$4\pi$-points.
We obtained a branched covering
$\mathbb{CP}^1\to\mathbb{CP}^1$ with branch-order at most two at any
point. By moving a little branch points we can assure that they have distinct
images (hence they are ``simple'', with the terminology
of~\cite{Hu91}.) The space of simple branched coverings from
$\mathbb{CP}^1$ to itself is connected
(\cite{Hu91}, see for instance~\cite[Thm 1.54, p. 34]{HsM} for a
proof). A detailed account on the topology under consideration can be
found in the Appendix.
By Corollary~\ref{c:secondindextheorem}
we know that the number of branch points is
even. Thus we can easily construct a BPS $P_1$ which is branched covering
$\mathbb{CP}^1\to\mathbb{CP}^1$, with the same number of branch points
as $P$, by consecutive bubblings.

By Corollary~\ref{lemmaunico}, applied with $A=S$, $B=\bar S$, $C=P$ and
$D=P_1$, the structure on $S$ is connected by moving branch points to
a cut-and-paste of $\bar S$ and $P_1$ which clearly contains a bubble.

Suppose now that we are in case $(2)$. The path $\sigma$ does not contain
other branch points because $p$ is the closest to $p_0$. If the image
$D_B(\sigma)$ contains the image of some other branch point $q$, then we
move $q$ a little away by using twin paths outside the ball
$B(p_0,d(p_0,p))$. This implies that after the moving of $q$,
the unique shortest path from $p$ to $p_0$ is still $\sigma$.
The path $D_B(\sigma)$
is the straight segment in $\mathbb H^2$ from $D_B(p_0)$ and $D_B(p)$.
The segment $D_B(\sigma)$ has two lifts $\sigma$ and $\tau$ emanating
from $p$. Since $p$ is the branch point closest to $p_0$, and since we
are not in case $(1)$, the
end-point of $\tau$ other than $p$ is regular (and it is not $p_0$).
By applying Lemma~\ref{l:embedded_twins} to $U=B(p_0,\varepsilon)$, $\Sigma=\mathbb H^2$, and $f=D_B$ we deduce that$\sigma\cup\tau$ is
embedded.
We move $p$ by cut-and-pasting along $\sigma\cup\tau$.
The result is that $p_0$ and $p$ join together to
give a branch point of angle $2(k+1)\pi$
(keep in mind Figure~\ref{f:movsing}). Now the induction on the number of
branch points other that $p_0$ concludes the proof.

\end{proof}

We remark that we just proved that when we collapse some branch points
to a single point, then the limit structure is a nodal
curve (possibly with no nodal
point if the limit is non-degenerate)
consisting in a bps on $S$ with some branched coverings of
$\mathbb{CP}^1$ attached to the nodal points.

\section{Moving branch points to the positive part}
\label{s:moving_positive}

In this section we prove the following result.
\begin{theorem}\label{t:thmsec7}
  Let $S$ be a closed surface equipped with a BPS with Fuchsian
  holonomy. Then one can move branch points so that either a bubble
  appears or all branch points belong to $S^+$.
\end{theorem}
\proof
The idea is to move one by one branch points
from the negative to the positive part. During the process bubbles
could possibly appear.

\begin{prop}\label{p:prop71} Let $\sigma$ be a
BPS with Fuchsian holonomy on a closed
surface $S$. If $p$ is a branch point contained
in a negative component whose complement is not
a union of discs, then $p$ can be moved to $S^+$.
\end{prop}

\begin{proof} Let $C^-$ be a the component of $S^-$ containing $p$.
We claim that there are two embedded twin paths $\gamma_1, \gamma_2$
in $S$ starting at $p$, disjoint
appart from $p$, and which end in the positive part. Hence, by moving
$p$ along the $\gamma_i$'s, we get the desired result.

Let $C_{fill}$ be the union of $C^-$ and the components of $S\setminus
C^-$ which are homeomorphic to discs;
this is a proper subsurface of
$S$. As $C_{fill}$ is not the whole $S$, its fundamental group is a
free group. Elementary considerations show that the map $\pi_1 (C_{fill})
\rightarrow \pi_1 (S)$ induced by the inclusion $C_{fill} \rightarrow
S$ is injective. So the image of the inclusion map $\pi_1 (C_{fill})
\rightarrow \pi_1 (S)$ is a free subgroup of $\pi_1
(S)$. Consequently, the surface $C^-\f = \rho(\pi_1(C_{fill}))
\backslash \mathbb H^2$ is not compact and of infinite area.

The developing map induces a projective (hence a local isometry)
map $D : C^- \rightarrow C^-\f$.
Let $\gamma : [0,\infty) \rightarrow C^-\f$ be an injective
semi-infinite geodesic path starting at $D(p)$ that goes to infinity
in $C^-\f$. We can suppose that $\gamma$ does not contain any
point of the form $D(q)$ where $q$ is a branch point of $C^-$, other
than $D(p)$ -- to ensure that we can change $p$ to a branch point $q$
so that $D(q)$ is the last point of this form that $\gamma$
encounters.

The geodesic $\gamma$ can be lifted to twin paths $\gamma_i: [0,
\infty ) \rightarrow C^-$ starting at $p$, with $i=1,2$, such that
$D(\gamma_i) = \gamma$. By Lemma~\ref{l:embedded_twins}
the images of $\gamma_1$ and $\gamma_2$ are disjoint appart from $p$
and they have distinct end-points on $S_{\mathbb{R}}$. So we can move
$p$ by cut and pasting anlog such twins paths and bring $p$ to $S^+$.
\end{proof}

It remains to deal with the case that $S\setminus C^-$ is a union of
discs.  By Proposition~\ref{l:degeneracion} either we find a bubble or
we can join together all the branch points belonging to the same
components.

\begin{prop} \label{p:_every_positive_component_is_a_disc} Let $C$ be
a negative (resp. positive)
component with a single branch point $p$ of angle $2\pi
(R+1)$. Suppose that $\partial C$ has a component $l$ which
is homotopically trivial in $S$. Then there exist two embedded twin
paths starting at $p$ with extremities in the positive
(resp. negative) component having $l$ in its boundary.
\end{prop}
\proof
We begin by some easy preliminaries on convex subsets of $\mathbb
H^2$. Suppose that $\mathcal C$ is a compact convex set in the
upper-half plane, with piecewise geodesic boundary. A vertex of
$\mathcal C$ is a point $v$ of $\partial \mathcal C$ such that the
boundary makes an exterior angle bigger than $\pi$ at $v$. We denote
such an exterior angular domain by
$[\theta, \theta']\subset \mathbb R / 2\pi \mathbb Z$,
and we introduce the angular
domain $B = [\theta + \frac{\pi}{2}, \theta' - \frac{\pi}{2} ]$ at
$v$.

\begin{lemma} \label{l:geodesics_do_not_intersect} Let $\gamma_1$ and
$\gamma_2$ be distinct semi-infinite geodesic rays starting at some
vertices $v_1$ and $v_2$ of $\mathcal C$ in the angular domains $B_1$
and $B_2$. Then $\gamma_1$ and $\gamma_2$ are disjoint, and their
limit at infinity are different. Moreover, each $\gamma_i$ does not
intersect $\mathcal C$ appart at $v_i$.
\end{lemma}

\begin{proof} By convexity of $\mathcal C$, it is clear that
$\gamma_1$ and $\gamma_2$ do not intersect $\mathcal C$ appart at
$v_1$ and $v_2$. If $v_1=v_2$, the first statement of the lemma is
clear. Suppose now that $v_1\neq v_2$. If $\gamma_1$ and $\gamma_2$
intersect in a point $p$ (even if this point is at infinity), then the
triangle $v_1 v_2 p$ has two angles not smaller than $\pi /2$, which
violates the fact that the sum of the angles of a triangle in the
hyperbolic plane is strictly less than $\pi$ if the points are not on
a same line.
\end{proof}

\begin{lemma} \label{l:sum_of_B} The sum of the angles of $B_v$ over
all vertex $v$ of $\mathcal C$ equals $2\pi + \mathrm{Area}(\mathcal
C)$.
\end{lemma}

\begin{proof} By Gauss-Bonnet formula, the sum of the interior angles
of $\mathcal C$ is equal to $(k-2)\pi - \mathrm{Area}(\mathcal C)$,
where $k$ is the number of vertices of $\mathcal C$. Because the angle
of $B_v$ of a vertex $v$ equals $\pi$ minus the interior angle at $v$,
we deduce the formula.
\end{proof}

We are now able to finish the proof of
Proposition~\ref{p:_every_positive_component_is_a_disc}. We suppose
$C\subset S^-$. The case $C \subset S^+$ has the same proof with
roles of $S^+$ and $S^-$ switched.
Let $\gamma $ be the peripheral geodesic associated to
$l$, and let $E$ be the end corresponding to $\gamma$. Recall (see
Lemma~\ref{halftub})
that $E$ is
an open Annulus bounded by $\gamma$ which may be not embedded as it
may self-intersect at $p$. The universal covering $\tilde E$ of $E$ is
a half-plane delimited by a lift $\tilde \gamma$ of $\gamma$.
Since $l$ has trivial
holonomy, a developing map $D$ for the BPS of $S$ induces a well-defined
projective map from $E$ to $\mathbb H^2$ which is a local
isometry. Also, this map lifts to a map, still called $D$,
from $\tilde E$ to $\mathbb
H^2$ which extends to $\partial \tilde E$. The image of $\partial
\tilde E$ is a closed polygonal $D(\tilde\gamma)$ that agrees
with the developed image of a lift of $\gamma$. The
vertices of such a polygon are images of lifts of $p$.
In particular, since there are at least two vertices,
$\gamma$ is not embedded.

Let $A_1,\ldots, A_k$ be the exterior (i.e. contained in $E$)
angular domains of $\gamma$ at $p$. Each $A_i$ is mapped
to an angular domain of $\mathbb H^2$ at some vertex of
$D(\tilde\gamma)$ delimited by two half-geodesics.
Such geodesics cut $\mathbb H^2$ into two angular
domains, a large one $L_i$ and a small one $S_i$. Since the angle of
$A_i$ is at least $\pi$, the image of $A_i$ by $D$
is $L_i$.

Let $\mathcal C$ be the convex hull of $D(\tilde\gamma)$.
Suppose first that $\mathcal C$ has non-empty interior. For each
vertex $v$ of $\mathcal C$, denote by $B_v$ the domain constructed
just before Lemma~\ref{l:geodesics_do_not_intersect}. Each vertex $v$
of $\mathcal C$ is a point of the form $D(p_i)$ for some lift $p_i$ of
$p$. Moreover, by the preceeding observation, the image of
$A_i$ contains the angular domain $B_v$. Let us
consider pre-images $C_v\subset A_v$
so that $D(C_v) = B_v$. Hence $C_v$ is an angular domain contained
in the exterior angle $A_i$ of $\gamma$ at $p$.

By Lemma~\ref{l:sum_of_B}, the sum of the angles of the domains $B_v$
is larger than $2\pi$. Hence, because the angle $C_v$ is equal to the
angle $B_v$, their sum exceeds $2\pi$. Thus a developing map at $p$
must overlaps some of the $C_v$ because of
the pigeon hole principle. This shows that there exists two
geodesic rays $\gamma_1$ and $\gamma_2$ starting at $p$, which are
contained in the union of the domains $C_v$, and differing by a
multiple of $2\pi$. Hence $\gamma_1$ and $\gamma_2$ are twin.
Moreover, they cannot be contained in the same $C_v$
because twin geodesics make angle at least $2\pi$ and any $B_v$ is
strictly less than $\pi$.
Hence we can suppose that $\gamma_1$
starts in $A_1$ and $\gamma_2$ in $A_2$.

We claim that these half-geodesic rays stay in $E$,
do not intersect, and tend to different points of $l$ at
infinity. The images $D(\gamma_1)$ and $D(\gamma_2)$ are
two half-geodesic rays in $\mathbb H^2$ that do
not intersect by Lemma~\ref{l:geodesics_do_not_intersect}, even at
infinity, and never intersect $\mathcal C$ appart from their starting
point. Lemma~\ref{l:embedded_twins} concludes.

Hence, the paths $\gamma_1$ and
$\gamma_2$ can be analytically extended across $c$, ending in the
positive disc, and the proposition is proved in the case $\mathcal C$
has non empty interior.

In the case where $\mathcal C$ is of empty interior, it is a geodesic
segment with distinct endpoints $q_1$ and $q_2$. Sectors $B_1$ and $B_2$
measure exactly $\pi$. Therefore there exists sectors $B_1'$ and
$B_2'$ each strictly larger than $\pi$ so that the conclusion of
 Lemma~\ref{l:geodesics_do_not_intersect} holds, and the proof goes
 now as in the general case.
This ends the proof of
Proposition~\ref{p:_every_positive_component_is_a_disc}.
\qed

\medskip
An induction based on
Propositions~\ref{p:prop71}
and~\ref{p:_every_positive_component_is_a_disc} proves
Theorem~\ref{t:thmsec7}.\qed

\section{Debubbling BPS with all
  branch points in
the positive part} \label{s:fuchspositivesing}

In this section, we consider a BPS $\sigma$
with Fuchsian holonomy on a closed compact surface $S$, having
branch points {\bf only in the positive part},
and we show that such a structure can always be
debubbled provided it has some branch point.

\begin{theorem}\label{t:mainpositive}
  Let $S$ be a closed surface equipped with a BPS with Fuchsian
  holonomy. Suppose that all branch points of $S$ belong to
  $S^+$. Then, if there is at least one branch point, we can move
  branch points so that the resulting BPS is a bubbling of another BPS
  with the same holonomy.
\end{theorem}
\proof
Fix a positive component $C$, let $k$ be the
number of branch points in $C$, and $n$ be the number of negative discs
$D_1,\ldots ,D_n$ adjacent to $C$. Let $l_i = \partial D_i$, with
the orientation induced by $C$.

\begin{lemma}\label{l:index01}
  The index of any $l_i$ is $1$, while the index of any other
  component of $\partial C$ is $0$.
\end{lemma}
\begin{proof}
The closed disc $\overline{D_i} = D_i \cup l_i$ does not contain
branch points, hence the developing map is a diffeomorphism from $D_i$
to $\mathbb H^-$, hence the index of the curve $l_i$ is $1$.

On the other hand, any negative component $C^-$ adjacent to $C$ has the
structure of a complete
hyperbolic surface with no cusps and,
if it is different from a disc,
its boundary has loxodromic holonomy. The claim follows
from the index formula \ref{t:first_index_theorem} applied to $C^-$.
\end{proof}

\begin{lemma} \label{l:negative_disc_component} $k = 2 n $.
\end{lemma}

\begin{proof}  By Lemma~\ref{l:index01}, the index formula applied to $C$ gives
\[ eu(\rho_{C}) = \chi (C) + k - n .\] To compute $eu(\rho_{C} ) $,
introduce the subsurface $C' = C \cup \overline{D_1} \cup \ldots \cup
\overline{D_n}$, oriented with the orientation of $S$. By the
above considerations $C'$ is an incompressible subsurface of
$S$. Thus by remark \ref{rem:eulerfuchsian} $eu(\rho_{C'}) =
\chi (C')$. Since
the twisted bundle $C' \times _{\rho} \mathbb R \mathbb P ^1$ is
trivial on the discs $D_i$'s , $eu(\rho_{C'}) = eu (\rho_{C})$. On the other hand, $\chi (C') = \chi
(C) + n$. We deduce that $eu (\rho _{C} ) = \chi (C) + n$, and the
lemma follows.
\end{proof}

Let us suppose that $k$ is strictly positive.
By Proposition~\ref{l:degeneracion}, either we can
find a bubble in $\sigma$ after moving branch points, or we can join
all the branch points to a single branch point $p$ of angle $2\pi
(k+1)$. For every $i= 1,\ldots , n$, let $\gamma_i$ be the peripheral
geodesic in $C$ corresponding to $l_i $. Since
$l_i$ is homotopically trivial in $S$, the developed image of
$\gamma_i$ is a closed piecewise geodesic of $\mathbb H^2$,
that has at least two vertices since it is closed.
Thus $\gamma_i$ passes at least twice through $p$.
Denote by $m_i = m_i ' + 2$ the number of times $\gamma_i $
passes through $p$, with $m_i' \geq 0$, and by
$\alpha^i_1,\ldots, \alpha^i_{m_i}$ the exterior angles of $\gamma_i$ at
$p$.

\begin{lemma} The developed image of $\gamma_i$ is the oriented
boundary of a convex polygon $\mathcal C_i$ of $\mathbb H^2$
whose exterior angles are $\alpha^i_1,\ldots, \alpha^i_{m_i}$.
\end{lemma}
\begin{proof} Let $y$ be a generic point of the developed image of $\gamma_i$
in $\mathbb H^2$ (viewed as the upper hemisphere of $\mathbb{CP}^1$).
We consider the extremity at
infinity $x_{\infty}(y)\in \mathbb R \mathbb P^1$ of the
geodesic ray starting from $y$ in the direction normal to $\gamma_i$,
on the right-side.

We follow $x_\infty$ as $y$ varies in the developed image of $\gamma_i$.
When $y$ passes through a vertex (a developed image of $p$),
we stop $y$ and let the normal vector describe the angular domain
$\alpha_j^i-\pi$.
As $y$ sweeps the whole developed image of $\gamma_i$,
the point $x_{\infty}(t)$ describes a closed curve in
$\mathbb R \mathbb P^1$ whose degree is the index of
$l_i$. Hence, the degree of this map is $1$. Because all the angles
$\alpha_j^i$ are at least $\pi$, the map $x_{\infty}(y)$ turns always
counterclockwise, and this implies that $\pi \leq \alpha^i_j \leq 2\pi$ for
every $j=1,\ldots,m_i$. This
means that the developed image of $\gamma_i$ always turns
counterclockwise,
with angles $\alpha^i_j$ on its right. Because the turning number of the
developed image of $\gamma_i$ is $1$, this implies that it bounds a
convex domain of $\mathbb H^2$ on its left, ant the lemma is
proved.
\end{proof}

The following result shows that there is always at least a peripheral
geodesic which is not too complicated.
\begin{lemma} \label{l:triangle} There are two possibilities:
\begin{enumerate}
\item \label{enum:2} either there exists a peripheral
geodesic which passes through the point $p$ exactly twice,
\item \label{enum:3} or $n=1$ -- equivalently $k=2$ -- and
the peripheral geodesic $\gamma_1$ passes through $p$ exactly three times.
\end{enumerate}
\end{lemma}

Observe that case~\ref{enum:2} occurs in a classical
bubbling, and an example of case~\ref{enum:3}
is described in Subsection~\ref{ss:_triangle}.

\begin{proof} The proof is based on the following Gauss-Bonnet
  formula: the sum of
the exterior angles of a convex hyperbolic polygon with $m$ vertices
equals its area plus $(m +2) \pi$.
Because the union of all the exterior angles of the
$\gamma_i$'s in the angular domain at $p$ are disjoint, we get
\[ \sum _{1\leq i \leq n } (m_i + 2) \pi + \mathrm{Area} (\mathcal C_i
) \leq 2\pi (k+1).  \] Dividing by $\pi$, by
Lemma~\ref{l:negative_disc_component}
and the relations $m_i = m_i ' + 2$, we get
\begin{equation}\label{eq:inequality} \sum_{1\leq i \leq n} m_i ' + \frac{1}{\pi} \sum _{1\leq i \leq n }
\mathrm{Area} (\mathcal C_i) \leq 2 . \end{equation}

If there exists $i$ such that $\mathrm{Area} (\mathcal C_i
)>0$, then $\sum_{1\leq i \leq n} m'_i\leq 1$ and  we infer
that either one of the $m_i'$'s vanishes, and we are in case~(\ref{enum:2}), or all are positive, and we are in case~(\ref{enum:3}). Suppose for a contradiction that $\mathrm{Int}(\mathcal C_i)=\emptyset$ for all $i$, $\sum_{1\leq i \leq n} m'_i\leq 2$ and  $(1)$ or $(2)$ hold. Then we are
necessarily in one of the following cases
\begin{enumerate}
\item [(a)] $n=1$ (and $k=2$) and $m_1 '=2$,
\item [(b)] $n=2$ (and $k=4$) and $m_1 '=m_2 '=1$.
\end{enumerate}
Since all the $\mathcal C_i$ have empty interior, the developed images of the $\gamma_i$'s are segments. Therefore they must have at least two exterior angles of $2\pi$. Since exterior angles are at least $\pi$, in both cases we see that the sum of the $\alpha_i^j$'s equals the total angle around $p$. In case~(a) we have angles
$\alpha^1_1,\alpha^1_2,\alpha^1_3,\alpha^1_4$ equal to $2\pi,2\pi,\pi,\pi$, and therefore they sum up to $6\pi=2(k+1)\pi$; in case~(b) we must have angles $\alpha^i_1,\alpha^i_2,\alpha^i_3$ equal to $2\pi,2\pi,\pi$ for $i=1,2$, and they sum up to $10\pi=2(k+1)\pi$. Therefore the ends corresponding to $l_i$ fill the whole of $C$. In particular in case~(a) the four petals of $\gamma_1$ are identified in pairs, whereas in case~(b) the petals of $\gamma_1$ and $\gamma_2$, which are each homeomorphic to a bouquet of $3$ circles, are identified in pairs of circles, so that in fact the union $\gamma_1 \cup \gamma_2$ is a bouquet of three circles. If we denote by $E_i$ the annular end corresponding to $\gamma_i$, we deduce that $S=E_1 \cup D_1\cup\gamma_1$ in case~(a); and $S=E_1 \cup E_2 \cup D_1\cup D_2\cup\gamma_1\cup \gamma_2$ in case~(b). In both cases the characteristic of $S$ is zero, hence $S$ is a torus, which does not carry Fuchsian representations. Hence we get the desired contradiction.
\end{proof}

\begin{lemma} Suppose that there is a peripheral geodesic
  $\gamma\in\{\gamma_1,\dots,\gamma_n\}$
which passes only twice through the point $p$. Then, by moving
$p$, one finds a bubble.
\end{lemma}

\begin{proof} In this case $\gamma$ is a bouquet of two
circles. Moreover, since the developed image of $\gamma$ is a segment,
this bouquet is formed by two pairs of twin
geodesics emanating from $p$. Let $(\nu_1 ^1, \nu_1^2)$ and $(\nu_2^1,\nu_2 ^2)$ be their germs at $p$, around the
two angles $\alpha_1$ and $\alpha_2$. Let us move the multiple branch
point $p$ along both twin paths. If the angle around $p$ is
$2(k+1)\pi$, the structure resulting
from these movements has three distinct branch points
$p',p_1,p_2$. The point $p'$ has angle $2(k-1)\pi$ (since $k=2n$,
we always have $(k-1)\geq 1$,
and $p'$ is a smooth point if $k=2$); the points $p_1$ and
$p_2$ are both of angle $4\pi$ and correspond
to the extremities of $\nu_1^i$'s and
$\nu_2^i$'s respectively. The peripheral geodesic corresponding to $l$
for this new structure is formed by two geodesic
segments going from $p_1$ to $p_2$.
The exterior angle are $2\pi$ at each point
$p_i$. Hence, since $l$ bounds a disc in $S^-$,
we are in the situation of a bubbling.
\end{proof}

\begin{lemma} \label{l:move2onesing} Suppose that $n=1$ and that
$\gamma_1$ passes through $p$ three times. Then one can move
branch points so that a bubble appears.
\end{lemma}
\begin{proof}
Let $D_1$ be the negative disc bounded by $l_1$.
By Proposition~\ref{p:_every_positive_component_is_a_disc},
we can find two embedded twin geodesics $\nu_1$ and $\nu_2$
starting from $p$ and going in $D_1$, in such a way that
they stay in the end corresponding to $l_1$ untill they
cross $l_1$. We move $p$ along these two
geodesics.

The resulting branched projective structure has the
following properties.
The disc $D_1$ gives raise to a negative annulus $A^-$, with a single
branch point of angle $4\pi$.
The component $C$ gives raise to one or two components whose union
$C^+$ contains a single branch point of angle $4\pi$ and so that
$\partial A \subset \partial C^+$.
Lemma~\ref{l:index01} and Theorem~\ref{t:first_index_theorem} assures that
Theorem~\ref{t:debubbling} apply in the present case and the proof is
complete.
\end{proof}

The proof of Theorem~\ref{t:mainpositive} is now complete.\qed

\section{Proof of Theorem~\ref{t:main}}\label{s:proofoftheorem}

First we prove that we can assume that the representation $\rho$ is Fuchsian.

\begin{prop}\label{p:homeomorphism}
Suppose that $\rho$ and $\rho'$ are two representations whose actions on the Riemann sphere are topologically conjugated. Then for every $k\geq 0$, the spaces $\mathcal M_{k,\rho}$ and $\mathcal M_{k,\rho'}$ are homeomorphic.
\end{prop}

\begin{proof}
Let $\Phi$ be a homeomorphism of $\mathbb {CP}^1$ such that $\rho'(\cdot) \circ \Phi = \Phi \circ \rho(\cdot)$. Then, if $(S,\sigma)\in \mathcal M_{k,\rho}$, there is a unique developing map $D: \widetilde{S}\rightarrow \mathbb {CP}^1 $ which is $\rho$-equivariant. Observe that the map $\Phi \circ D$ is $\rho'$-equivariant, and define a BPS $(S,\sigma')$ with holonomy $\rho'$ and $k$ branch points. The correspondance $(S,\sigma)\in \mathcal M_{k,\rho} \mapsto (S,\sigma') \in \mathcal M_{k,\rho'}$ is the desired homeomorphism.
\end{proof}

It remains to prove that if $\rho$ is a Fuchsian representation, and $k>0$ is an integer, then $\mathcal M_{k,\rho}$ is connected. Let $S$ be a connected closed oriented surface, and let $\sigma$ be a
branched projective structure on $S$. Suppose that $\sigma$
has holonomy $\rho$ and that it has at least one branch point.

By Corollary~\ref{c:secondindextheorem}
we know that the number $k$ of branch points is
even. Let $\Sigma$ be the uniformizing structure on $S$ with holonomy $\rho$. Let
$\sigma_0^k$ be the BPS obtained by applying $k/2$  bubblings to $\Sigma$. By
Corollary~\ref{c:bc} this does not depend on where we perform
bubblings.

We are going to prove by induction on $k$ that $\sigma$ is connected
to $\sigma_0^k$ by moving branch points, and this clearly implies
Theorem~\ref{t:main}.
The base for induction is $k=2$ as we suppose $k>0$ (for $k=0$ the
claim of the theorem is false).

First, by moving a little branch points we reduce to the case that no
branch point belongs to the real line $S_{\mathbb R}$.

By Theorem~\ref{t:thmsec7} we can move branch points so that either
we find a bubble, or all branch points belong to the positive part
$S^+$. Now, By Theorem~\ref{t:mainpositive} we find a bubble. In any
case, there is a finite sequence of movements of branch points that
connects $\sigma$ to a BPS $\sigma_1'$ which is a bubbling on a
BPS $\sigma_1$.

If $k>2$, by induction $\sigma_1$ is connected to
$\sigma_0^{k-1}$ by moving branch points. From
Corollary~\ref{lemmaunico} applied to $A=\sigma_1'$,
$B=\mathbb{CP}^1$, $C=\sigma_1$, and $D=\sigma_0^{k-1}$, we know that
$\sigma_1'$ is connected to a bubbling of $\sigma_0^{k-1}$ and by
Corollary~\ref{c:bc} such a bubbling is connected to $\sigma_0^k$. In
conclusion, $\sigma$ is connected to $\sigma_0^k$.

If $k=2$, then $\sigma_1$ has no branch points. By a celebrated theorem
of W. Goldman in~\cite{Goldman1},
$\sigma_1$ is obtained by grafting $\Sigma$ along a disjoint
union of simple closed curves $\gamma_1,\ldots, \gamma_n$. By moving branch points to the position of the bottom left part of Figure~\ref{fig:bubdebub}, where we put $\gamma=\gamma_n$, we get a BPS $\sigma_1^{(2)}$. By Theorem~\ref{t:graftingvsbubbling} it is precisely $\Sigma$ grafted along $\gamma_1,\ldots, \gamma_{n-1}$ and  bubbled once. We can repeat this process $n$ times to get, after appropriately moving the branch points, a BPS $\sigma_1^{(n)}$ that is $\Sigma$ grafted along $\emptyset$ and bubbled once, i.e. $\sigma_0^2$.
\qed

\section{Appendix: deformation spaces of branched projective structures} \label{s:moduli}
Let $S$ be a compact orientable surface of genus $g\geq
2$, equipped with a marking (i.e. an identification with the group of
covering transformations of $\pi:\tilde S\to S$ with a fixed group $\Gamma_g$). Let
$k$ be non-negative integer, and $\rho : \Gamma_g
\rightarrow \mathrm{PSL}(2,\mathbb C)$ be a non-elementary
representation. In this section, we endow the set $\mathcal
M_{k,\rho}$ of equivalence classes of BPS on $S$ with total branching order $k$ and holonomy conjugated to $\rho$, with
the structure of a non-singular complex manifold of dimension $k$. In
fact, we explicitly construct a complex atlas modeled on Hurwitz spaces (which
inherit the complex structure from a particular space of polynomials).

We will state our main result in terms of deformations of branched
projective structures, in the spririt of Kodaira-Spencer's theory of
deformations of complex manifolds, see e.g.~\cite{Kodaira}.

\begin{definition}\label{def_hol_fam}
Let $S$ be a marked compact surface of genus $g\geq 2$.
A \textit{holomorphic family of BPS} on $S$ is a quadruple $(X,B,\pi, \mathcal W)$ where
\begin{enumerate}
  \item $X$ and $B$ are complex manifolds, $ \pi : X \rightarrow B$ is
    a holomorphic submersion with compact  fibers $S_b = \pi^{-1} (b)$
  diffeomorphic to $S$ for all $b\in B$.
  \item $\mathcal W=\{(U_i,w_i)\}$ is a maximal set of holomorphic
    functions $w_i : U_i \rightarrow \mathbb {C P}^1 $  such that
    $\{U_i\}$ is an open cover of $X$, the restriction of $w_i$ to
    each fibre $S_b$ is non-constant and on each connected component
    $V$ of $U_i\cap U_j$ the functions $w_i$ and $w_j$ are related by
       \[w_i=\phi_{ij}(b) (w_j)\] where $\phi_{ij}:\pi(V)\rightarrow
       \mathrm{PSL}(2,\mathbb{C})$ is a holomorphic map.

\item there is a holomorphic  covering $\widetilde{X} \rightarrow X$
  whose covering group is $\Gamma_g$ and such that over each fiber
  $S_b$, it is the universal covering of $S_b$.
\end{enumerate}

The restriction of $\mathcal{W}$ to any  fibre $S_b$ defines a BPS
$\sigma_b=\mathcal W|_{S_b}$ on
$S$  whose holonomy representation (up to conjugation) will be
denoted by $\rho_{\mathcal{W}}(b):\Gamma_g\rightarrow\mathrm{PSL}(2,\mathbb{C})$.

When $(X,B,\pi,\mathcal W)$ is a holomorphic family of BPS,
we also say that $X$ is a holomorphic family of BPS over $B$, with atlas
$\mathcal W$.
\end{definition}

\begin{theorem} \label{t:universalfamily}
Given a non-elementary representation  $\rho : \Gamma_g
\rightarrow \mathrm{PSL}(2,\mathbb C)$, there exists a smooth complex manifold structure on $\mathcal
M_{k,\rho}$ such that,  for any holomorphic family $(X, B, \pi ,
\mathcal W )$ of BPS on $S$ with $k$ branch points counted with multiplicity
and holonomy $\rho_{\mathcal W}=\rho$, the map $$b\in
B\mapsto \sigma_b=\mathcal W|_{S_b} \in \mathcal M_{k,\rho}$$ is
holomorphic.
\end{theorem}

 It would be interesting to endow the set $\mathcal
  M_k=\cup_{\rho}\mathcal{M}_{k,\rho}$ of BPS with $k$ branch points
  counted with multiplicity  with a structure of a complex analytic
  space, by gluing the complex structures on $\mathcal M_{k,\rho}$
  together. This can be done over the set of non elementary
  representations, but we don't know how to manage this over the
  elementary representations. Some preliminaries are in order.

Let $T_g$ be the Teichm\"uller space consisting of equivalence classes
of marked Riemann surfaces of genus $g$, up to biholomorphism lifting
to the universal cover as a $\Gamma_g$-equivariant
diffeomorphism. Bers
equipped the set $T_g$ with the
structure of a complex manifold of dimension $3g-3$, and constructed
the tautological bundle $\Pi: \mathcal T_g \rightarrow T_g$ over it:
a holomorphic family of marked Riemann surfaces where the fiber over a
point $t\in T_g$ is biholomorphic to $t$. He showed that this family
of Riemann surfaces is universal in the sense of Kodaira-Spencer's
deformation theory, see~\cite{Bers},~\cite[Th\'eor\`eme 1.2]{Gro} and~\cite[p. 446]{ACG}.
Namely, if $\pi : X \rightarrow B$ is any
holomorphic family of marked Riemann surfaces, then there is a
holomorphic map $f : B \rightarrow T_g$ such that $(X,B,\pi)= f^*
(T_g,\mathcal T_g, \Pi)$.

To prove this result Bers introduced a holomorphic family of (unbranched) $\mathbb {CP}^1$-structures on $\mathcal T_g$, usually called Bers simultaneous uniformization. This is the data of a holomorphic family of quasi-Fuchsian representations $$t\in T_g \mapsto \rho _t \in \mathrm{Hom} \big(\Gamma_g, \mathrm{PSL}(2,\mathbb C)\big)$$
and of an open set $U \subset T_g \times \mathbb {CP}^1$ such that for every $t\in T_g$, the intersection $U\cap \{t\}\times \mathbb {CP}^1$ is of the form $\{t\}\times U_t $, where $U_t $ is a component of the domain of discontinuity of $\rho_t$. The set $\mathcal T_g$ is then constructed as the quotient of $U$ by the action $\gamma (t, z) = (t, \rho(t) (\gamma)(z))$ of $\Gamma_g$ on $U$. The projection onto the first coordinate gives a map $\Pi:\mathcal{T}_g\rightarrow T_g$ whose fiber over a point $t\in T_g$ is the Riemann surface structure $t$ on $S$.

Let us comment on the relationship between the two
natural projections defined so far on the space of branched projective
structures with total branching order $k$: the holonomy projection to
the $\mathrm{PSL}(2,\mathbb{C})$-character variety and the projection
to Teichm\" uller space. The first one provides, together with the
order of branching, the stratification by the sets
$\mathcal{M}_{k,\rho}$. When $k=0$, that is, on strata of unbranched
projective structures, the fibres of both projections intersect
transversally (in fact by Goldman's theorem and Baba's
generalization the holonomy fibers are discrete).
In this case, the holonomy and the underlying conformal
structure determine the projective structure
(see~\cite{Man2}). In the case of $0<k<2g-2$
where $g$ is the genus of the surface, it is still true that
the triple, holonomy, conformal structure and branching divisor
determine the branched projective structure
(see~\cite{Man2}). Moreover, for $0<k<2g-2$, the fibers of the
Teichmuller and holonomy projections are transverse: for a fixed
holonomy, moving branch points changes the conformal structure.
However, when $k\geq 2g-2$, there exist movements
of branch points that preserve the holonomy and the conformal
structure.  As suggested by G. Mondello,
they can be calculated directly by using the Beltrami
differentials associated to moving branch points and Riemann-Roch's
formula. A consequence is that the fibres of the two projections are not
transverse in this case (note that the complex
dimension of the Teichm\"uller
space of $S$ is $3g-3$ while the dimension of $\mathcal M_{\rho,k}$ is
$k$).

Coming back to the question of existence of rational curves in
$\mathcal{M}_{k,\rho}$ stated in subsection~\ref{ss:problems}, by considering the tautological bundle
associated to a rational curve, we get a holomorphic family of branched projective structures
parametrized by $\mathbb{CP}^1$ whence the
projection on Teichm\"uller space provides a rational curve, which
must be
constant, as Teichm\"uller space does not contain
rational curves. So the rational curve must be contained on the fiber
of the projection. In particular this implies $k\geq 2g-2$. The same holds true for a non trivial holomorphic family of BPS over a compact base $B$.

\subsection{The deformation spaces $\mathcal M_{k_1,\ldots , k_r,
    \rho}$}\label{ss:fixedtypes}

Let $k_1\leq \ldots \leq k_r$ be positive integers such that $k_1 +
\ldots + k_r = k$, and $\mathcal M_{k_1,\ldots, k_r}$ be the set of
BPS on $S$ whose branch divisor is of the form $\sum _{i= 1}^r k_i p_i$ for
some set of distinct points $\{p_1,\ldots, p_r\}$. The set $\mathcal
M_{k_1,\ldots, k_r}$ has a structure of an analytic space that can be
defined by using the usual Schwarzian derivative parametrization. We
recall this construction for convenience of the reader,
see~\cite[12.2, p. 679]{GKM} in the case where $k_i = 1$ for all $i$.
For any BPS in $\mathcal M_{k_1,\ldots,k_r}$, one can consider its underlying Riemann surface structure $t$ and, introducing the Bers coordinate $z\in U_t$ and a developing map $D$ for the given BPS, construct the meromorphic quadratic differential on $U_t$ defined by
\begin{equation} \label{eq:schwarzian} q(z) := \{ D(z), z \} dz^2,\end{equation}
where $\{D,z\}$ denotes the Schwarzian derivative. This meromorphic quadratic differential on $U_t$ does not depend on the choice of the developing map and is invariant by $\rho_t$, hence defines a meromorphic quadratic differential on $t$ that has pole set at the branch points $\{p_1, \ldots, p_r\}$ of the BPS. Its Laurent series expansion around $p_i$ in a coordinate $z$ where $p_i=z_i$ is
\begin{equation}\label{eq:laurentexpansion}  q(z) = \frac{1-(k_i+1)^2}{2(z-z_i)^2} + \sum _ {n\geq -1}  a_n^{(i)}(z-z_i) ^n .\end{equation}
A necessary and sufficient condition on the coefficients $a_n$ ,
$n\geq -1$, ensuring that a meromorphic quadratic differential of
type~\eqref{eq:laurentexpansion} is the quadratic meromorphic
differential associated to a BPS of $\mathcal M_{k_1,\ldots, k_r}$ as
in~\eqref{eq:schwarzian}, is that some polynomial equation on the
coefficients $a_n^{(i)}$ is satisfied, called the indicial
equation. This equation takes the form
\begin{equation}\label{eq:indicialequation} A_{k_i} ( a_{-1}^{(i)},
  \ldots, a_{k_i -1}^{(i)} ) = 0 , \end{equation}
where $A_{k_i}$ is a polynomial in $k_i$ variables with coefficients
in $\mathbb{C}$ (see~\cite[p. 268]{Man}).

Introduce the tautological fiber bundle $\mathcal T_g ^{(r)}$ over
Teichm\"uller space $T_g$, where the fibre over a point $t\in T_g$ is
biholomorphic to the set $t^r \setminus \Delta(t)$ where $\Delta(t)$
is the set of $r$-tuples $(p_1,\ldots p_r)\in t^r$ such that $p_i=
p_j$ for some distinct indices. Let $\mathcal Q_g^{(r)}$ be the vector
bundle over $\mathcal T_g ^{(r)}$ whose fiber over a point $t^{(r)}\in
\mathcal T_g^{r}$ consists of those meromorphic quadratic
differentials on $t^{r}$ whose poles are precisely at the $p_i$'s
corresponding to $t^r$ and that   satisfy
equation~\eqref{eq:laurentexpansion}.
Following the preceding discussion, the set $\mathcal M_{k_1,\ldots, k_r}$ can be identified with the subset of $\mathcal Q_g^{(r)}$ consisting of meromorphic quadratic differentials verifying the indicial equations~\eqref{eq:indicialequation}. This set has the structure of a complex analytic space. Moreover, the subset of $\mathcal M_{k_1,\ldots, k_r,\rho}$ consisting of those BPS in $\mathcal M_{k_1, \ldots, k_r}$ whose holonomy is conjugated to $\rho $ is an analytic subset. Indeed, since $\rho$ is non-elementary, $\mathcal{M}_{k_1,\ldots,k_r,\rho}$ is a fibre of the analytic map that associates to each $\sigma$ in $\mathcal M_{k_1, \ldots, k_r}$ the character $\Lambda_{\sigma}:\Gamma_g\rightarrow\mathbb{C}$ defined by $\Lambda_{\sigma}([\gamma])=\mathrm{Tr}^2(\rho_{\sigma}([\gamma]))$.

The structure of complex space structure on $\mathcal M_{k_1,\ldots,
  k_r, \rho}$ constructed in~\ref{ss:fixedtypes} cannot be extended
in an obvious manner to provide a complex analytic structure on
$\mathcal M_{ k, \rho}$. The reason is that the map that assigns to a
BPS its number of branch points is only semi-continuous, and in fact
discontinuous around points where some of the $k_i\geq 2$. This is
reflected in the Schwarzian coordinates. For instance, given a
continuous family $\sigma_b$ of elements in $\mathcal{M}_{k,\rho}$
where at $\sigma_0$ there is a branch point $p_0$ of multiplicity
$k_0\geq 2$ and for all $b\neq b_0$ all the branch points of
$\sigma_b$ are simple, there is a discontinuity in the family of
meromorphic quadratic differentials $q_b$ at the point $b_0$: its
number of poles is not constant (note that the poles of the
Schwarzian are always of order two, and the information about the
branching-index of a BPS is encoded in the principal coefficient of
the Schwarzian, see~\eqref{eq:laurentexpansion}). To overcome this difficulty, we will
adopt a different approach and use Hurwitz spaces to endow $\mathcal
M_{k,\rho}$ with the structure of a \emph{complex manifold}.

\subsection{The smooth topology on $\mathcal{M}_k$}
Given a point $\tau\in\mathcal M_k$ we consider its underlying complex structure on $S$ and name it $t(\tau)\in T_g$. Since $U_{t(\tau)}$ is the
complex analytic universal covering associated to $t$, we can consider a developing map $D_\tau$ for $\tau$. It is tautologically a holomorphic map from $U_\tau$ to $\mathbb{CP}^1$.
The {\bf smooth topology} on $\mathcal M_k$ is the topology induced by
the injective map $\tau\mapsto (\Gamma_{t(\tau)},D_\tau)$ using the topology of
the Bers slice one one component and the uniform convergence on compact sets for
developing maps. As we let $\mathrm{PSL}(2,\mathbb C)$ act on the space of the
developing maps, it is not completely obvious that this topology on $\mathcal{M}_k$ (or even its restriction to $\mathcal{M}_{k,\rho}$) is Hausdorff, so we are led to introduce the topological structure differently.
\subsection{Cut and paste topology on $\mathcal M_{k,\rho}$}

We begin by introducing a topology on $\mathcal M_{k,\rho}$ as
follows. The Riemann surface associated to a BPS in $\mathcal
M_{k,\rho}$ is equipped with its Poincar\'e metric coming from
uniformization. Given $\varepsilon >0$ and $\sigma\in\mathcal
M_{k,\rho}$ define $\mathcal V (\varepsilon,\sigma)\subset\mathcal
M_{k,\rho}$ as the set of elements $\sigma' \in \mathcal
M_{k,\rho}$ such that there exists a diffeomorphism $\Phi : S
\rightarrow S$ satisfying:
\begin{enumerate}
\item $\Phi$ is $(1+\varepsilon)$-bi-Lipschitz with respect to the
  Poincar\'e metrics on $S$ given by $\sigma$ and $\sigma'$ respectively,
\item $\Phi$ is projective outside the $\varepsilon$-neighborhoods of
  the branched set of $\sigma$,
\item $\Phi$ is isotopic to the identity.
\end{enumerate}
 The topology on $\mathcal M_{k,\rho}$ is defined as the one generated
 by the neighbourhoods $\mathcal V (\varepsilon, \sigma)$.

\begin{lemma}
The cut and paste topology on $\mathcal M_{k,\rho}$ is separated.
\end{lemma}

\begin{proof}
Suppose that $\sigma_i$, for $i=1,2$ are two elements of
$\mathcal M_{k,\rho}$ that cannot be separated. This means that for
every $\varepsilon >0$, there exists an element
$\sigma_{\varepsilon}$ in $\mathcal V (\varepsilon, \sigma_1)
\cap \mathcal V (\varepsilon,\sigma_2)$. By
definition, for every $\varepsilon>0$, and every $i=1,2$, we have a
diffeomorphism $\Phi_{\varepsilon, i} : S \rightarrow
S $ such that properties 1), 2) and 3) above are
satisfied w.r.t. $\sigma_\varepsilon$ and $\sigma_i$.
The diffeomorphism $\Psi_{\varepsilon}= \Phi_{\varepsilon,
  2}^{-1} \circ \Phi_{\varepsilon, 1} $ is
$(1+\varepsilon)^2$-bi-Lipschitz, w.r.t. the metrics given by
$\sigma_1$ and $\sigma_2$, is projective appart from the
$\varepsilon (1+\varepsilon)^2$-neighborhood of the branched set of
$\sigma_1$, and isotopic to the identity. By the theorem of
Arzela-Ascoli, one can find a sequence $\varepsilon_n$ tending to $0$
when $n$ tends to infinity such that $\Psi_{\varepsilon}$ converges to
an isometry $\Psi$ w.r.t the metrics given by $\sigma_1$ and
$\sigma_2$, which is projective appart
from the branched set of $\sigma_1$. Then $\Psi$ is projective
everywhere, and it is isotopic to the identity. Hence $\sigma_1
= \sigma_2$ in $\mathcal M_{k,\rho}$.
\end{proof}

\begin{lemma} \label{l:_c_is_continuous}
Given any holomorphic family $(X,B,\pi,\mathcal W)$ of BPS on $S$ with $k$
branch points and holonomy conjugated to $\rho$, the induced map from
$B$ to $\mathcal M_{k,\rho}$ (endowed with the cut-and-paste topology)
is continuous.
\end{lemma}

\begin{proof}
We have to prove that given $b_0\in B$ and $\varepsilon >0$, there is
a neighborhood of $b_0$ in $B$ such that for every $b$ in this
neighborhood, the BPS
$\sigma_b=\mathcal W|_{S_b}$ on $S$ belongs to $\mathcal V
(\varepsilon, \sigma_{b_0})$.

Since the holonomy of $\sigma_b$ is always conjugated to the non-elementary
representation $\rho$, there is a unique developing map $D_b:
\widetilde{S_b} \rightarrow \mathbb C \mathbb P^1$ which is
$\rho$-equivariant. The family of functions $\{ D_b\}$ defines a
holomorphic function $D:\widetilde{X}\rightarrow \mathbb{CP}^1$. By
$\rho$-equivariance, the foliation defined by $D=\mathrm{cst}$ is
invariant by the $\Gamma_g$-action on $\widetilde{X}$, and defines a
regular holomorphic foliation $\mathcal F$ on $X$ which is
transversally projective. By construction, the BPS on the curves $S_b$
is the restriction of the transversal projective structure of this
foliation, (see subsection~\ref{ss:examples}).
The foliation $\mathcal F$ is tangent to the curves $S_b$ precisely at
the set $B_b$ of branch points of $\sigma_b$. Hence, if we denote by
$B_b^{\varepsilon}$ the set of points of $S_b$ within distance
$\varepsilon$ from $B_b$, and if $b$ is sufficiently close to $b_0$,
there is a family of diffeomorphisms $\Phi_{b}: S_{b_0} \setminus
B_{b_0}^{\varepsilon} \rightarrow S_b$ depending differentiably on the
parameter $b$, that preserves each leaf of $\mathcal F$, and such that
$\Phi_{b_0}$ is equal to the identity. By elementary topological
arguments, this family can be extended differentiably to a family of
diffeomorphisms $\Phi_b$ between $S_{b_0}$ and $S_b$ such that
$\Phi_{b_0} = id$. Because the Poincar\'e metric on the fibers $S_b$
varies continuously with the parameter $b$, for $b$
close enough to $b_0$, this family of diffeomorphisms verifies the
conditions 1), 2) and 3) of the definition of $\mathcal V
(\varepsilon, \sigma_{b_0})$.
\end{proof}

\begin{cor}\label{p:_Ehresman} The cut-and-paste and the smooth topologies coincide on each stratum $\mathcal M_{k_1,\ldots , k_r,\rho}$.\end{cor}
\begin{proof}
  Let $\Pi:\mathcal T_g\to T_g$, $\mathcal T_g^{(r)}\to T_g$ and $\mathcal
  Q_g^{(r)}\to \mathcal T_g^{(r)}\to T_g$ be as before
  (namely, they are respectively the tautological bundle over the
  Teichm\"uller space of $S$, its
  $r$-th symmetric product, and a vector bundle of meromorphic quadratic
  differentials) so that $\mathcal M_{k_1,\ldots , k_r,\rho}$ is identified
  with the sub-set
  of $\mathcal Q_g^{(r)}$ of Schwarzian-integrable differentials with
  holonomy conjugated to $\rho$. Now we build a holomorphic family of
  BPS as follows. The base $B$ is just $\mathcal M_{k_1,\ldots ,
    k_r,\rho}$ endowed with its smooth topology induced by $\mathcal
  Q_q^{(r)}$. Then, we let $X$ be the pull-back of the
  tautological bundle $\pi_t:\mathcal T_g\to T_g$
  via the map $\mathcal Q_g^{(r)}\to T_g$. The set of maps $\mathcal W$ is given by the integral
  of the Schwarzian differentials.  By
  Lemma~\ref{l:_c_is_continuous}, the identity between  $\mathcal
  M_{k_1,\ldots , k_r,\rho}$ with the smooth topology and itself with
  the cut-and-paste topology is continuous. Since $\mathcal
  M_{k_1,\ldots , k_r,\rho}$ with the smooth topology is locally
  compact with countable basis, the identity is a homeomorphism
  between  the two topologies.
\end{proof}

At this point it is worth mentioning that there is yet another viewpoint for studying branched projective structures: that of flat holomorphic connections.
Namely, let $d\geq 0$ and $g\geq2$ be integers. Consider the set of all
quadruples $(C,E,\nabla,s)$ where $C$ is a compact marked Riemann surface of genus $g$
$E$ is a $\mathbb{CP}^1$-bundles over $C$, $\nabla$ is a flat
holomorphic connexion on $E$ (equivalently a horizontal foliation  $\mathcal F$
which is transverse to
the fibers) with irreducible  monodromy, and $s$
is a section which has $d$ points of tangencies with $\mathcal
F$. This provides a BPS on the surface $C$ (as in
subsection~\ref{ss:examples}), and the space of connections provides
analytic coordinates (on each stratum with a fixed number of tangencies).
For details about this viewpoint we refer the reader to~\cite{Gun,Man,Man2,Man3}.

\subsection{Hurwitz spaces} Our goal now is to prove that $\mathcal
M_{k,\rho}$ is locally modeled on a product of Hurwitz spaces,
i.e. moduli spaces of coverings of the disc.

\begin{definition}[Hurwitz spaces]\label{def:hur_sp} Let $U$ be a smooth closed disc,
  $\psi : U \rightarrow \mathbb D$ be a branched covering of degree
  $d$ with no critical values on the boundary, and let
  $f:\mathbb S^1\to \partial U$ be a diffeomorphism so that $\psi\circ
  f(z)=z^d$ on $\mathbb S^1$.  We consider the
set of smooth branched coverings $\psi '  : U' \rightarrow \mathbb D$
of degree $d$
from a smooth closed disc $U'$ to $\mathbb D$, with no critical value on the boundary,
together with an identification $f': \partial U\to\partial U'$ such that
$\psi= \psi'\circ f'$ on $\partial U$ (equivalently, so that $\psi'\circ f'\circ
f(z)=z^d$ on $\mathbb S^1$).
Two such coverings $\psi_i : U_i \rightarrow
\mathbb D$, $i=1,2$, are identified if the diffeomorphism $\varphi =
f_2\circ f_1^{-1}$ from $\partial U_1$ to $\partial U_2$ extends to a
diffeomorphism $\phi: U_1 \rightarrow U_2$ such that $\psi_1=\psi_2\circ\phi$.

The set of equivalence classes under this equivalence relation
is denoted by $\mathcal H (\psi)$, and will be called a Hurwitz
space of degree $d$ coverings.
\end{definition}

Observe that the pull-back via $\psi$ of the projective structure of
$\mathbb D$ given by its inclusion in $\mathbb{CP}^1$ induces a
BPS on $U$. Moreover, if $\psi_i$, $i=1,2$ are two elements of $\mathcal H(\psi)$,
then the diffeomorphism
$\varphi = f_2 \circ f_1 ^{-1}: \partial U_1 \rightarrow \partial U_2$
extends to a projective diffeomorphism from a neighborhood of
$\partial U_1$ to a neighborhood of $\partial U_2$. The
identifications of the boundaries of the discs with $\mathbb S^1$
define a marked point $1$. Namely, we set $1:=f(1)$ in $\partial U$
and $1:=f'(1)$ in any $\partial U'$.

We now prove that Hurwitz spaces are nicely parametrized by open sets
of complex vector spaces, and this will define a natural topology on $\mathcal H(\psi)$.

\begin{lemma}\label{l:Hurwitz} Any Hurwitz space $\mathcal H (\psi)$
  of degree $d$ coverings is a smooth complex
manifold of dimension $d-1$.
More precisely,  $\mathcal H(\psi)$ is in bijection with the set of
complex polynomials of the form
\[ P(z) = z^d + a_{d-1} z^{d-1} + \ldots + a_0 \]
with $a_{d-1} + \ldots + a_0 = 0$ and with all critical values in the
interior of the unit disc.
\end{lemma}

\begin{proof}
For every class $\psi'$ in $\mathcal H(\psi)$, the disc $U'$ is
equipped with a unique Riemann surface structure such that $\psi': U'
\rightarrow \mathbb D$ is holomorphic.
Let $\mathbb D^c$ be the exterior of the unit disc in $\mathbb {C P}^1$.
 We glue $U'$ and $\mathbb D ^c $  by
using the identification of their boundaries given by $f'\circ f$. We
obtain a Riemann surface of genus $0$ that we denote by $\mathbb
P_{[\psi']}$. The covering $\psi' :U'\rightarrow \mathbb D$
can be glued together with the covering $z\in \mathbb D^c \mapsto z^d
\in \mathbb D^c$ to give rise to a holomorphic branched covering
$\overline{\psi'} : \mathbb P_{[\psi']} \rightarrow \mathbb
{CP}^1$. Since the set of M\"obius transformations acts freely and
transitively on triples $(x,y,v)$ where $x\neq y\in\mathbb{CP}^1$ and
$0\neq v\in T_x(\mathbb{CP}^1)$, there is a unique biholomorphism
$\eta : \mathbb {CP}^1
\rightarrow \mathbb P_{[\psi']}$ such that
\begin{itemize}
\item $\eta(\infty)= \infty\in \mathbb D^c$;
\item $\eta(1)=1\in \mathbb S^1=\partial\mathbb D^c$;
\item $ \overline{\psi'} \circ \eta(w) = w^d + O(w^{d-1})$.
\end{itemize}
We denote by $P$ the polynomial $ \overline{\psi'} \circ \eta(w)$.
By construction it satisfies the assumptions of the lemma.
Reciprocally, if $P$ satisfies the assumptions of the lemma, denote by
$V_P:= P^{-1} (\overline{\mathbb D})=\{z:P(z)\leq 1\}$ and let $\psi'$ be the
restriction of $P$ to $V_P$. Because there is no critical value of
modulus $\geq 1$ apart from the point at infinity, $V_P$ is a disc and the covering
$\psi' : \partial V_P \rightarrow \partial \mathbb D$ is
cyclic. Observe that $P(1) = 1$, so that there is a unique
diffeomorphism $ f_P : \partial U\to\partial V_P$ such that $P (f_P(f(z))) = z^d$.
\end{proof}
\begin{rem}\label{rm:fhol}
  From the fact that $P(f_P(f(z)))=z^d$ we get that
$P(z)=(f^{-1}\circ f_P^{-1}(z))^d$ hence $f^{-1}\circ
f_P^{-1}(z)=\sqrt[d]{P(z)}$. It follows that $f_P$ depends
holomorphically on the variable $P$.
\end{rem}
\subsection{A holomorphic atlas on $\mathcal M_{k,\rho}$ modeled on Hurwitz
  spaces}

We begin with the following local description near a branch point $p$ of a
BPS $\sigma$ on a surface $S$. Let $U$ be a disk-neighbourhood of $p$,
with local complex coordinate $\zeta$ so that the map $\psi(\zeta)\to \zeta^d$ belongs to the
atlas of $\sigma$. Choose $f:\mathbb S^1\to\partial U$ as in
definition~\ref{def:hur_sp}.
We identify $\mathcal H(\psi)$ with the set of polynomials given by
Lemma~\ref{l:Hurwitz}. For any $P\in\mathcal H(\psi)$ let
$V_P=P^{-1}(\mathbb D)$ and define the set $X\subset\mathbb C\times
\mathcal H(\psi)$ by
$$X=\{(z,P)\in \mathbb C\times \mathcal H(\psi):\ z\in V_P\}.$$
 Let $\pi:X\to \mathcal H(\psi)$ be the natural projection, which is
 clearly a holomorphic submersion with fibres
 $\pi^{-1}(P)=V_P$. Moreover, the function $w:X\to \mathbb{CP}^1$
 given by $w(z)=P(z)\in \mathbb D\subset\mathbb{CP}^1$ defines a
 maximal atlas $\mathcal W$ so that $w\in\mathcal W$.
 Thus, $X$ can be viewed as a holomorphic family of BPS
 over $\mathcal H(\psi)$. (In our Definition~\ref{def_hol_fam} the
 fibres are diffeomorphic to $S$, but a similar definition can be given for
 families of BPS on a disk.)

The boundary $\partial X$ is a fibre bundle
$$\partial V_P\hookrightarrow\partial X\stackrel{\pi}{\to} \mathcal H(\psi).$$

The identification from $\partial U\times \mathcal
H(\psi)=\partial (U\times \mathcal H(\psi))$ to $\partial X$ given by
$(z,P)\to (f_P(z),P)$ is holomorphic and extends holomorphically from a
collar of $\partial U\times
\mathcal H(\psi)$ in $U\times \mathcal H(\psi)$ to a collar of $\partial X$ in $X$
by Remark~\ref{rm:fhol}.

It follows that by gluing $S\setminus U\times \mathcal H(\psi)$ with
$X$ along their common boundaries using the above identification we
get a complex manifold $\mathcal X$ and a holomorphic submersion
$\mathcal X\to\mathcal H(\psi)$ so that the fibre over $P$ is the BPS
obtained from $\sigma$ by replacing $U$ with $V_P$. Note that, after the identification
  of the collars of $\partial U\times \mathcal H(\psi)$ and of
  $\partial X$, the changes of charts near
  $\partial U$ are the identity by Remark~\ref{rm:fhol}. The developing
maps given by the atlas of $\sigma$ on $S\setminus U$ and by
$(z,P)\to P(z)$ on $X$, provide an atlas $\mathcal W$ as requested by
Definition~\ref{def_hol_fam}, and since $U$
is a disk, such a construction lifts to the universal covering of
$S$. Therefore, $\mathcal X$ is a holomorphic family of BPS over
$\mathcal H(\psi)$.

We are now ready to describe a complex atlas of $\mathcal M_{k,\rho}$
modeled on a product of Hurwitz spaces by repeating the above construction around
every branch point of $\sigma$.

Let $\sigma$ be a BPS on a surface $S$.
For every branch point $p$ of $S$, there exists a disc neighbourhood
$U_p$ with a complex coordinate $\zeta_p$ and an integer $k_p$ such
that the branched covering $\psi_p(\zeta) = \zeta_p ^{k_{p}}$ belongs
to the atlas of $\sigma$. By restricting $U_p$ if necessary and
composing $\psi_p$ on the left by an affine transformation, we may
assume that the $U_p$'s are disjoint and that the image of any
$\psi_p$ is the unit disc $\mathbb D\subset\mathbb{CP}^1$.

We denote by $\Psi:= (\psi_p)_p$, $\mathcal H (\Psi):= \prod_p
\mathcal H(\psi_p)$, and $k=\sum_p k_p$, where the index $p$ runs over all branch points
of $\sigma$. Given any element $\Psi' := (P_p)_p \in \mathcal H
(\Psi) $, we construct a new branched projective structure belonging
to $\mathcal M_{k,\rho}$ by cutting off $U_p$ and gluing back
$V_{P_p}$ via the identifications of the boundaries for every branch
point $p$. This defines a subset of $\mathcal M_{k,\rho}$ that will be
denoted by $\mathcal V ( \Psi )$. This procedure defines a map
$c(\Psi) : \mathcal H (\Psi) \rightarrow \mathcal V(\Psi)$.

\begin{lemma}\label{l:vnei} For every $\Psi$ as above, the set $\mathcal V(\Psi)$ is
  a neighborhood of $\sigma$ in the topology of $\mathcal
  M_{k,\rho}$.
\end{lemma}

\begin{proof} Since every element of $\mathcal V(\varepsilon,\sigma)$
is obtained from $\sigma$ by changing the projective structure only in
the $\varepsilon$-neighborhood of the branch-set of $\sigma$,
for a sufficiently small $\varepsilon>0$,
$\mathcal V(\varepsilon,\sigma)\subset \mathcal V(\Psi)$.
\end{proof}

\begin{lemma} \label{l:bijection} The map $c(\Psi)$ is a bijection. \end{lemma}

\begin{proof} To obtain this claim it is sufficient to prove that if
$\Psi'=(P_p)_p\in \mathcal H(\Psi)$
produces a branched projective structure $\sigma'$ equivalent to
$\sigma$, then $\Psi'= \Psi$.
If $\sigma'$ is equivalent to $\sigma$,
then there is a diffeomorphism
$\Phi:S\to S$ isotopic to the identity which is projective
w.r.t. $\sigma'$ and $\sigma$ respectively.  Let
$U'_p=\Phi(V_{P_p})$ and consider the inclusion
$i:S\setminus\{U_p\}\hookrightarrow S\setminus\{V_{P_p}\}$, which is
projective by definition. The map $h=\Phi\circ i$ is a projective
diffeomorphism from $S\setminus \{U_p\}$ to $S\setminus\{U'_p\}$
that is isotopic to the identity. It lifts to the universal covers
$$\tilde h:\widetilde S\setminus\{\widetilde{U_p}\}\to \widetilde
S\setminus \{\widetilde{U'_p}\}$$ as a $\Gamma_g$-equivariant
diffeomorphism. Let $D$ be a developing map for
$\sigma$. Since $\tilde h$ is locally a projective map, and since
$\widetilde S\setminus\{\widetilde {U_p}\}$ is connected, there exists
a M\"obius transformation $A$ such that
$$D\circ \tilde h=A\circ D.$$
By $\rho$-equivariance of $D$ it follows that $A$ commutes with the
image of $\rho$, hence $A=Id$ as $\rho$ is irreducible.

Now we choose local coordinates $\zeta_p$ near a branch point $p$ such that
$D(\zeta_p)=\zeta_p^{k_p}$. We get $(\tilde
h(\zeta_p))^{k_p}=\zeta_p^{k_p}$
which implies that $\tilde h$
can be analytically extended to the whole $\tilde S$. So $h$ extends
to a biholomorphism of $S$. Since $h$ is isotopic to the identity and
 $S$ admits only a finite number of automorphims, we get $h=Id$. It follows that
 $V_{P_p}=U_p$ and that $P_p(\zeta_p)=\zeta_p^{k_p}$, so $\Psi'=\Psi$.
\end{proof}

Using the complex coordinates for Hurwitz spaces given by
Lemma~\ref{l:Hurwitz}, let us now prove the following result.

\begin{lemma}\label{l:cisaholomorphicfamily} There is a
  holomorphic family
  $\pi:\mathcal X\to\mathcal H(\Psi)$ of BPS on $S$ so that the
  structure over a point $b\in \mathcal H(\Psi)$ is $c(\Psi)(b)$.
\end{lemma}

\begin{proof} For each branch point $p$ of $\sigma$, let
  $X_p$ be as before:
$$X_p=\{(z,P_p)\in\mathbb C\times\mathcal H(\psi_p):\ z\in V_{P_p}\}$$
and define
$$Y_p=X_p\times\Pi_{q\neq p} \mathcal H(\psi_q).$$

We have $\partial Y_p=\partial X_p\times\Pi_{q\neq p} \mathcal
H(\psi_q)$ which, as before, is identified with $\partial U_p\times
\mathcal H(\Psi)$ by using the maps $f_{P_p}:\partial U_p\to\partial V_{P_p}$.
Let $Z=(S\setminus \cup_p U_p )\times  \mathcal H (\Psi)$. Since
$\partial Z=\cup_p\partial U_p\times\mathcal H(\Psi)$, we can glue
$Z$ with $\cup_p Y_p$ along their common boundaries getting a complex
manifold $\mathcal X$. The natural projection $\pi:\mathcal X\to\mathcal
H(\Psi)$ is holomorphic. The maximal atlas $\mathcal W$ is defined as
follows. On $Z$ we use the atlas of $\sigma$. On each $Y_p$ the maps
are defined by
$$(z,P_p,(P_q)_{q\neq p})\mapsto P_p(z)$$
and then we extend this set of maps to a maximal one. After the
identifications via the maps $f_{P_p}$,  the
changes of charts between the atlas of $\sigma$ and the maps on the
$Y_p$'s are projective (because in the local charts with coordinate $\zeta_p$
the change of chart is the identity by by Remark~\ref{rm:fhol}).

Since the disks $U_p$ are disjoint, the whole construction lifts to
the universal cover and so the quadruple $(\mathcal X,\mathcal
H(\Psi),\pi,\mathcal W)$ is a holomorphic family of BPS on $S$.

Finally, by the construction of $\mathcal W$, it follows that
on the fibre over a point $b=(P_p)_p\in\mathcal
H(\Psi)$ we have the structure $\sigma_b=c(\Psi)(b)$.
\end{proof}

\begin{cor}The map $c(\Psi)$ is a
  homeomorphism. \end{cor}

\begin{proof} By Lemmas~\ref{l:cisaholomorphicfamily}
  and~\ref{l:_c_is_continuous}, we deduce that $c(\Psi)$ is a
  continuous map from $\mathcal H (\Psi)$ to the neighborhood
  $\mathcal V(\Psi)$ of $\sigma$. Because of
  Lemma~\ref{l:bijection} $c(\Psi)$ is bijective, and that $\mathcal
  H(\Psi)$ is locally compact, we conclude that it is a
  homeomorphism.
\end{proof}

This already proves that $\mathcal M_{k,\rho}$ is locally homeomorphic
to $\mathbb R^{2k}$.

\subsection{Proof of theorem~\ref{t:universalfamily}} We begin by the
following holomorphic version of Lemma~\ref{l:_c_is_continuous}.

\begin{lemma}\label{l:_holomorphic_change_of_coordinates} Let
  $(X,B,\pi, \mathcal W)$ be a holomorphic family of BPS with $k$
  branch points and holonomy conjugated to $\rho$. Let $\mathcal U$ be an open set in $B$ so that the induced map
  $\mathcal U\to\mathcal M_{k,\rho}$ is contained in a neighborhood
  modeled on a Hurwitz space $\mathcal H(\Psi)$. Then the induced map
  $b\mapsto\sigma_b=\mathcal W|_{S_b}$
  is holomorphic w.r.t. the complex structure given by the polynomial
  parametrization of $\mathcal H(\Psi)$ (that is to say, the map
  $c(\Psi)^{-1}(\sigma_\bullet):B\to\mathcal H(\Psi)$ is holomorphic).
\end{lemma}

\begin{proof} Let $k_1,\ldots, k_r$ be integers such that the generic element defined by the holomorphic family $(X,B,\pi,\mathcal W)$ belongs to the stratum $\mathcal M_{k_1,\ldots,k_r,\rho}$. Since we already have the continuity of the map by
  Lemma~\ref{l:_c_is_continuous}, it suffices, by Riemann's extension
  theorem, to show that it is holomorphic on the complement of some
  proper analytic set to deduce that it is holomorphic everywhere. The
  map defined on  $B$ that associates to a point $b$ the
  unordered set of branch points of $S_b$ (that is to say a point in
  the symmetric product of $S_b$) is
  holomorphic. Let $b_0$ be a point such that $\sigma_{b_0}$ belongs to $\mathcal M_{k_1,\ldots,k_r,\rho}$. Let $p_1,\dots,p_r$ be the branch points of
  $\sigma_{b_0}$ and $U_1,\dots,U_r$ be some disk neighbourhoods of the
  $p_i$'s used to define the Hurwitz neighbourhood of
  $\sigma_{b_0}$. The genericity of $b_0$ implies that the polynomials
  that arises near $b_0$ have all a single critical value in $\mathbb
  D$ of multiplicity $k_i$ and so they are completely determined by that (because of
  condition $P(1)=1$ of Lemma~\ref{l:Hurwitz}). We therefore have to
  show that the unique critical point of each such polynomial
  depends holomorphically on $b$.

Fix $i\in\{i,\dots, r\}$ and set $U^{b_0}=U_i$.
We use notation as in Lemma~\ref{l:_c_is_continuous} to define the
foliation $\mathcal F$ and diffeomorphisms $\Phi_b$. In particular, we
may suppose (up possibly to rescaling $U^{b_0}$) that $\Phi_b$ is defined from a
collar of $\partial U^{b_0}$ to a collar of
$\partial U^b$.

By construction, the developing map on $U^{b_0}$ is $\psi(z)= z^{k_i}$,
and the developing map on a collar of $\partial U^b$ is given by
$\psi\circ\Phi_b^{-1}$ (because $\Phi_b$ is constructed via the
$\mathcal F$-flow). The holomorphic map $\psi\circ\Phi_b^{-1}$
extends to a unique holomorphic map $\psi_b$ on $U^b$, which is a
fortiori the developing map for $\sigma_b$ on $U^b$. So, the covering
$\psi_b:U^b\to \mathbb D$ together with the identification
$\Phi_b:\partial U^{b_0}\to\partial U^b$ gives the element in
$\mathcal H(\psi)$.

Following the construction of Lemma~\ref{l:Hurwitz},
let $\eta:\mathbb{CP}^1\to\mathbb P_{\psi_b}$ be the change of
coordinates that give the requested polynomial $P$. Thus, we have
$P=\psi_b\circ\eta$ on $\eta^{-1}(U^b)$. It follows that the unique
critical value of $P$ does not depend on $\eta$, but only on $b$ and it is in fact the
unique critical value of $\psi_b$, which is in turn
the image $\xi(b)$ of the tangency point between $\mathcal F$ and the
curve $S_b$ under the map in $\mathcal W$ which extends $\psi$. This
shows that $\xi(b)$ depends holomorphically on $b$  in a neighbourhood of any point of $B^*:=\{b_0\in B: \sigma_{b_0}\in\mathcal{M}_{k_1,\ldots,k_r,\rho}\}$. If non-empty, the set $B\setminus B^*$ is a proper analytic set in $B$. In either case, the holomorphic map $B^*\rightarrow \mathcal{H}(\Psi)$ extends continuously to $B$ and hence holomorphically by Riemann's extension theorem.
 \end{proof}

\begin{cor}\label{c:holch}
Let $\sigma$ be a BPS on $S$ with $k$ branch points counted with
multiplicity, and holonomy $\rho$. Let $\Psi= (\psi_p)_p$ and $\Psi'=
(\psi_p')_p$ be some systems of projective coordinates around each of
the branch point of $\sigma$, as in Lemma~\ref{l:vnei}.
Then the map $$c(\Psi')^{-1} \circ
c(\Psi) : c(\Psi)^{-1}(\mathcal V(\Psi)\cap \mathcal V(\Psi') \big)
\rightarrow c(\Psi ')^{-1} \big( \mathcal V(\Psi)\cap \mathcal
V(\Psi')\big)$$ is holomorphic.
\end{cor}

\begin{proof}
By Lemma~\ref{l:cisaholomorphicfamily}, there is a holomorphic family
of BPS on $\mathcal X\to \mathcal H(\Psi)$ so that the map
$b\mapsto \sigma_b$ (where $b\in\mathcal H(\Psi)$) is $c(\Psi)$. Thus,
by Lemma~\ref{l:_holomorphic_change_of_coordinates},
$c(\Psi)$ is holomorphic w.r.t. the complex structure of $\mathcal
H(\Psi')$ on $c(\Psi)^{-1}(\mathcal V(\Psi)\cup\mathcal
V(\Psi'))$. That is to say, $c(\Psi')^{-1}\circ c(\Psi)$ is
holomorphic.
\end{proof}

We now finish the proof of Theorem~\ref{t:universalfamily}. For any
$\sigma$ there exists a system $\Psi$ of projective coordinates as in
Lemma~\ref{l:vnei} defining a neighbourhood $\mathcal V(\Psi)$ of
$\sigma$. The homeomorphism $c(\Psi)^{-1}:\mathcal V(\Psi)\to\mathcal
H(\Psi)$ provides local complex coordinates, and
Corollary~\ref{c:holch} tells us that the changes of charts are in
fact holomorphic. Thus $\mathcal M_{k,\rho}$ is a complex manifold, and
Lemma~\ref{l:_holomorphic_change_of_coordinates} completes the
proof. \qed

\end{document}